\documentclass[11pt, reqno]{amsart}   	
\usepackage[truedimen,margin=29truemm]{geometry}          		
\usepackage{graphicx}				
\usepackage{amssymb}
\usepackage{amsmath}
\usepackage{amsthm}
\usepackage{mathtools}
\usepackage{url}
\usepackage{bigints}
\usepackage{enumerate} 
\usepackage{color}

\newcommand{\loc}{\mathop{\mathrm{loc}}}

\newcommand{\sgn}{\mathop{\mathrm{sgn}}}

\newcommand{\supp}{\mathop{\mathrm{supp}}}

\newcommand{\N}{\mathbb{N}} 

\newcommand{\R}{\mathbb{R}} 
\newcommand{\C}{\mathbb{C}}

\newcommand{\jap}[1]{\langle #1 \rangle}
\def\pa{\partial}
\def\a{\alpha}
\def\b{\beta}
\def\c{\gamma}

\def\g{\psi}

\def\l{\lambda}
\def\m{\mu}
\def\n{\nu}

\def\s{\sigma}

\numberwithin{equation}{section}

\theoremstyle{plain}
\newtheorem{thm}{Theorem}[section]
\newtheorem{proposition}[thm]{Proposition}
\newtheorem{lemma}[thm]{Lemma} 
\newtheorem{corollary}[thm]{Corollary}

\theoremstyle{definition}

\newtheorem{assump}[thm]{Assumption}

 \newtheorem{remark}[thm]{Remark}

 \newtheorem*{remarks*}{Remarks}
\newtheorem*{remark*}{Remark}

\title[Dispersive estimates for negative Coulomb operators]{Dispersive estimates for Schr\"odinger operators with negative Coulomb-like potentials in one dimension}
\author{Akitoshi Hoshiya}
\address{Graduate School of Mathematical Sciences, The University of Tokyo, 3-8-1 Komaba, Meguro-ku, Tokyo 153-8914, Japan}
\email{hoshiya@ms.u-tokyo.ac.jp}

\author{Kouichi Taira}
\address{Faculty of Mathematics, Kyushu University, 744, Motooka, Nishi-ku, Fukuoka 819-0395, Japan}
\email{taira.kouichi.800@m.kyushu-u.ac.jp}

\begin{document}



\maketitle

\begin{abstract}
In this paper, we consider the dispersive estimates for Schr\"odinger operators with Coulomb-like decaying potentials, such as $V(x)=-c|x|^{-\mu}$ for $|x|\gg 1$ with $0<\m<2$, in one dimension.
As an application, we establish both the standard and orthonormal Strichartz estimates for this model. One of the difficulties here is that perturbation arguments, which are typically applicable to rapidly decaying potentials, are not available. To overcome this, we derive a WKB expression for the spectral density and use a variant of the degenerate stationary phase formula to exploit its oscillatory behavior in the low-energy regime.
\end{abstract}

\section{Introduction and main results}\label{2508241543}

\subsection{Background}

We consider the Coulomb-like Schr\"odinger operator
\begin{align*}
P=-\pa_x^2+V(x) \quad \text{on}\quad \R,
\end{align*}
where $V$ is a real-valued function with a Coulomb decay such as $V(x)=-c|x|^{-\m}$ with $\m\in (0,2)$ and $c>0$ for $|x|\gg 1$. Such potentials are important in physics because they represent the one-dimensional analogue of the hydrogen atom's Hamiltonian $H=-\Delta-\frac{c}{|x|}$ on $\R^3$. In mathematical aspect, Coulomb-like Schr\"odinger operators have the following features:

\begin{itemize}

\item Since the potential $V$ decays slowly at infinity, the low-energy behavior of $P$ is markedly different from that of $P_0=-\pa_x^2$. Due to its negativity and slow decay, the potential exerts an attractive effect, and consequently, we find that $P$ possesses infinitely many negative eigenvalues \cite[Theorem XIII.6 (Vol IV)]{RS}. On the other hand, the negativity of $V$ slightly lowers the threshold between scattering energies and binding energies, which implies that the zero-energy states exhibit scattering behavior. For studies on stationary scattering states, see \cite{DS,FS,IT,Su,Y-1} and references therein.

\item The potential $V$ exhibits a long-range interaction near infinity at positive energies when the parameter $\m$ satisfies $\m\in (0,1]$. Consequently, the asymptotic behavior of the generalized eigenfunctions and the long-time asymptotics of the time propagator are different from those of the free motion $P_0=-\pa_x^2$. This difference is crucial when considering the scattering theory for $P$ and the studies on the scattering matrices. For example, see \cite[Chapter 10]{Y} and references therein.

\end{itemize}

In this paper, we study dispersive $L^1\to L^{\infty}$ estimates 
\begin{align}\label{eq:dispersive}
\|e^{-itP}E_{\mathrm{ac}}(P)\|_{L^1\to L^{\infty}}\lesssim\frac{1}{|t|^{\frac{d}{2}}}\quad (\text{for $t\in \R\setminus \{0\}$}),
\end{align}
where $d$ is the dimension and $E_{\mathrm{ac}}(P)$ denotes the orthogonal projection onto the absolutely continuous subspace of $P$. Despite its importance in physics, there are fewer studies on Schr\"odinger operators with colulomb-like potentials $V(x)=-c|x|^{-\m}$, particularly concerning long-time behavior, compared to those with fast-decaying potentials. The main reason is that the potential functions $V$ cannot be treated as a perturbation of $P_0=-\pa_x^2$ since they decay slowly at infinity. Conversely, potentials that decay fast, like $|x|^{-\m}$ with $\m>2$, can be treated perturbatively using the Born series (\cite{GS,JK}). As far as the authors know, there are only two works on dispersive type estimates for $0<\m<2$ but with positive charges $V(x)=c|x|^{-\m}$ ($c>0$). The first one is due to Mizutani \cite{M} who proved Strichartz estimates for $P=-\Delta+c|x|^{-\m}$ with $c>0$ and $0<\m<2$ when the dimension is strictly greater than two (see also \cite{T} for the two-dimensional case). The second one is the result by Black-Toprak-Vergara-Zhou \cite{BTVZ}, where they showed dispersive estimates for $P=-\Delta+c|x|^{-1}$ with $c>0$ when the initial data is radially symmetric and the dimension is three. On the other hand, dispersive estimates or Strichartz estimates for potentials with negative charges $V(x)=-c|x|^{-\m}$ have not been studied to date. This paper presents the first attempt to investigate these estimates in this model.

The dispersive estimates \eqref{eq:dispersive} have been extensively investigated for both fast-decaying and inverse square potentials. In one dimension, Goldberg and Schlag \cite{GS} established the estimate under specific decay assumptions on $V$: $\int_{\R}(1+|x|)V(x)dx<\infty$ when the operator $P$ has no zero-energy resonances, and $\int_{\R}(1+|x|)^2V(x)dx<\infty$  otherwise. For three-dimensional cases, the sharpest results were provided by Beceanu and Goldberg \cite{BG} for generic potentials and by Beceanu \cite{B} for exceptional ones. For results in higher dimensions, see \cite{GG}, and for counterexamples showing the necessity of regularity assumptions on $V$, see \cite{GV}. Furthermore, works on Schr\"odinger operators with inverse square potentials include the three-dimensional case \cite{FFFP} and the higher-dimensional cases \cite{T3}. Finally, specific results regarding Schr\"odinger operators on conic manifolds with inverse square potentials can be found in \cite{JZ,JZ2,T3}. See also a great survey \cite{Sc} on the dispersive estimate.

In related work, local $L^2$ time decay estimates of the form 
\begin{align}\label{eq:localtime}
\|e^{-itP}E_{\mathrm{ac}}(P)\|_{L^2(K)\to L^2(K)}\lesssim \frac{1}{|t|^{N}},\quad \text{$K\subset \R$ is a compact set}
\end{align}
have been studied following the pioneering work of Jensen and Kato \cite{JK}. For $P=-\Delta+V(x)$ on $\R^d$ with fast-decaying potentials $V$, the optimal decay rate is given by $N=\frac{d}{2}$, which occurs in generic cases \cite{JK}. Regarding inverse square potentials, $V(x)=c|x|^{-2}$, Wang \cite{W} showed that the optimal decay rate $N$ is given by $N=\sqrt{\frac{(d-2)^2}{4}+c}+1$ for $d\geq 2$. Notably, this rate depends on the coupling constant $c$.
 On the other hands, for slowly decaying potentials $V(x)=\frac{c}{|x|^{\m}}$ ($0<\m<2$), Yafaev \cite{Y-1} showed that $N$ can be taken arbitrary large when $c>0$ and $N=1$ is the optimal rate when $c<0$ in the one-dimensional case. Later, both results were generalized to higher-dimensional cases by Nakamura \cite{N1} for $c>0$ and Fournais-Skibsted \cite{FS} for $c<0$. See \cite{FSY,T2} for fractional cases and \cite{BB,G} on asymptotically conic manifolds.

\subsection{Main result}

Here we consider smooth potentials unlike the exact model of Hydrogen atom since the singularity of $|x|^{-\m}$ at $x=0$ is too strong in one dimension (at least when $\m\geq 1/2$). We note that the results in \cite{BTVZ,B,BG} suggest that a weak singularity, such as $|x|^{-1}$ in the three-dimensional case, would not affect the decay property of the propagator.

\begin{assump}\label{assump:V}
A function $V \in C^{\infty} (\R; \R)$ is a smooth real-valued function satisfying
\begin{align*}
|\partial^{\alpha} _{x} V(x)| \lesssim \langle x \rangle^{-\mu -\alpha}, \quad -V(x) \gtrsim \langle x \rangle^{-\mu}
\end{align*}
for $\alpha \in \N_0$ and some $\mu \in (0, 2)$ uniformly in $x \in \R$, where we write $\jap{x}=(1+|x|^2)^{\frac{1}{2}}$.
\end{assump}

\begin{remark}
In our proof, the assumptions for $\a=0,1,2,3$ are only needed.
\end{remark}

Under Assumption \ref{assump:V}, the Schr\"odinger operator $P=-\pa_x^2+V(x)$ is self-adjoint with domain $H^2(\R)$ and its spectrum consists of the essential spectrum $[0,\infty)$ and infinitely many negative eigenvalues accumulating to $0$ by \cite[Theorem 1.1]{FS} and \cite[Theorem XIII.6 (Vol IV)]{RS}\footnote{Rigorously speaking, \cite[Theorem XIII.6 (Vol IV)]{RS} deals with the three dimensional case only, however, the proof given there works in any dimension.}. 
Moreover, it turns out from Theorem \ref{2509021637} that the equation $Pu=\l^2 u$ has two linearly independent non-$L^2$ solutions for each $\l\geq 0$. By the uniqueness of solutions for the ODE, we find that $P$ does not have non-negative eigenvalues and hence is absolutely continuous on $[0,\infty)$. Regarding the absence of embedded eigenvalues, we could use Rellich's theorem of the forms \cite{FHHH} for positive eigenvalues and \cite[Theorem 2.4]{FS} for zero eigenvalues, where the latter result requires an additional assumption on the derivative of $V$, see \cite[Assumption 2.3 (2)]{FS}.
The main result in this paper is

\begin{thm}\label{thm:main}
Under Assumption \ref{assump:V}, the dispersive estimate \eqref{eq:dispersive} holds (for $d=1$).

\end{thm}

\begin{remark}
$(1)$ The projection $E_{\mathrm{ac}}(P)$ onto the absolutely continuous subspace is needed; otherwise, the estimates for $|t|\gg 1$ fail due to the existence of eigenvalues. Moreover, we can write $E_{\mathrm{ac}}(P)=I-E_{\mathrm{pp}}(P)$, where $E_{\mathrm{pp}}(P)$ is the orthogonal projection onto the eigenspaces due to the absence of the singular continuous spectrum.

$(2)$ The estimate for $|t|\leq T$ ($T>0$) directly follows from Yajima's result (\cite[Theorem 1.1]{Ya}), which are also true for more general subquadratic potentials. The contribution in our main result is to give a uniform estimate for $|t|\gg 1$.

$(3)$ The decay rate as $|t|\to \infty$ is expected to be optimal, even though the optimal exponent for the local $L^2$ time decay estimate \eqref{eq:localtime} is $N=1$ \cite{FS,Y-1}.

$(4)$ While Assumption \ref{assump:V} imposes that $V$ is globally negative, the conclusion of Theorem \ref{thm:main} remains valid even if $V$ is non-negative on a bounded interval. The proof is based on a combination of our method and the regularity theory of ODEs with respect to parameters. We provide an outline of the proof in Section \ref{251271359}.
\end{remark}

By the result in \cite{KT} and \cite[Theorem 1.4]{H}, we have the following usual and orthonormal Strichartz estimates:
\begin{corollary}
Let $(q,r), (\tilde{q},\tilde{r})\in [2,\infty]^2$ satisfy the admissible condition $\frac{2}{q}+\frac{1}{r}=\frac{1}{2}$ and let $\b=\frac{2r}{r+2}$. Then
\begin{align*}
&\|e^{-itP}E_{\mathrm{ac}}(P)u_0\|_{L^q_tL^r_x}\lesssim \|u_0\|_{L^2},\quad \left\|\int_0^te^{-i(t-s)P}E_{\mathrm{ac}}(P)F(s)ds\right\|_{L^q_tL^r_x}\lesssim\|F\|_{L^{\tilde{q}^*}_tL^{\tilde{r}^*}_x},\\
&\left\|\sum_{j=1}^{\infty}\n_j|e^{-itP}E_{\mathrm{ac}}(P)f_j|^2 \right\|_{L^{\frac{q}{2}}_tL^{\frac{r}{2}}_x}\lesssim \|\n\|_{\ell^{\b}}
\end{align*}
for all $u_0\in L^2(\R)$, all $F\in L^{\tilde{q}^*}(\R_t;L^{\tilde{r}^*}(\R_x))$, and all orthonormal system $\{f_j\}_{j=1}^{\infty}\subset L^2(\R)$ and complex-valued sequences $\n=\{\n_j\}_{j=1}^{\infty}$. Here the orthonormal estimates in the second line are true only when $r<\infty$ holds.

\end{corollary}

\subsection{Idea of the proof}\label{subsec:Idea}
As mentioned before, a perturbation argument such as in \cite{B,BG,GS,HS} does not work since the potential function $V$ decays slowly at infinity. To go around this, we use the spectral representation of the time propagator
\begin{align*}
e^{-itP} E_{\mathrm{ac}}(P)=\int_{0}^{\infty} e^{-it\lambda^2} \tilde{E} (\lambda) d\lambda,\quad \text{where}\quad \tilde{E} (\lambda) := \frac{R(\lambda^2 +i0)-R(\lambda^2 -i0)}{2\pi i} \cdot 2\lambda,
\end{align*}
and $R(\l^2\pm i0):=(P-\l^2\mp i0)^{-1}$ are the outgoing/incoming resolvents of $P$. The Sturm-Liouville theory in the one-dimensional case allows us to deduce WKB expressions of the integral kernel of $\tilde{E}(\l)$ (see Theorem \ref{2509041030}):
\begin{align*}
\tilde{E} (\lambda, x, x') = \sum_{\sigma_{1}, \sigma_{2} \in \{\pm\}} b_{\sigma_{1}, \sigma_{2}} (\lambda, x, x') e^{iS_{\sigma_{1}, \sigma_{2}} (\lambda, x, x')},
\end{align*}
where $S_{\sigma_{1}, \sigma_{2}}$ and $b_{\sigma_{1}, \sigma_{2}}$ satisfy $S_{\sigma_{1}, \sigma_{2}} (\lambda, x, x') = \sigma_{1} \int_{0}^{x} \sqrt{\lambda^2 -V(s)} ds + \sigma_{2} \int_{0}^{x'} \sqrt{\lambda^2 -V(s)} ds$ and $|\partial^{\alpha} _{\lambda} b_{\sigma_{1}, \sigma_{2}} (\lambda, x, x')| \lesssim \min (\lambda^{1-\alpha} \langle x \rangle^{\frac{\mu}{4}} \langle x' \rangle^{\frac{\mu}{4}}, \lambda^{-\alpha})$. In this way, we can reduce the proof of the dispersive estimates to that of a time decay estimate for an oscillatory integral of the form
\begin{align}\label{eq:outlineint}
\int_0^{\infty}b_{\s_1,\s_2}(\l)e^{i\Phi_{\s_1,\s_2}(\l)} d\l,\quad \Phi_{\s_1,\s_2}(\l):=-t\l^2+S_{\s_1,\s_2}(\l,x,x')
\end{align}
uniformly in $x,x'\in \R$. It essentially suffices to consider the cases $t>0$, $x\geq x'$, $|x|\geq |x'|$. By symmetry, we have only to deal with $(\s_1,\s_2)=(+,-)$ or $(\s_1,\s_2)=(+,+)$.

First, we consider the $(+,-)$-case.
When $\l$ is sufficiently large compared to $\jap{x}^{-\m/2}$, namely $\l\gg \jap{x}^{-\m/2}$, then the phase function essentially becomes $\Phi_{+,-}\approx -t\l^2+\l(x-x')$, which is the same phase function  for the free propagator $e^{it\pa_x^2}$ in the Fourier transform side. Applying a quantitative version of the stationary phase theorem (Lemma \ref{Lem:Qstph1}), we obtain the dispersive estimates in this case.
On the other hand, when $\l$ is small compared to $\jap{x}^{-\m/2}$, namely $\l\lesssim \jap{x}^{-\m/2}$, the situation becomes different. To see this, we perform Taylor's expansion of $\Phi_{+,-}=-t\l^2+\int_{x'}^x\sqrt{\l^2-V(s)}ds$ at $\l=0$ to obtain
\begin{align}\label{eq:lowTaylor}
\Phi_{+,-}=\left(-t+\frac{M_1}{2}\right)\l^2+\frac{1}{4}M_2\l^4+O(\l^5)\quad (\l\to 0),
\end{align}
where we set $M_j=\int_{x'}^x|V(s)|^{-\frac{2j-1}{2}}ds$ ($j=1,2$). This expansion implies that the phase function $\Phi_{+,-}$ is degenerate when $t=M_1/2$, however, its fourth derivative does not vanish. Using this and the fact that the amplitude $b_{+,-}$ vanishes at $\l=0$: $b_{+,-}=O(\jap{x}^{\frac{\m}{2}}\l)$, we can apply a degenerate stationary phase theorem (Lemma \ref{lem:Qstph2}) and find that the integral in \eqref{eq:outlineint} decays like $O(\jap{x}^{\frac{\m}{2}}M_2^{-\frac{1}{2}})$. Finally, we observe that $M_2\approx \jap{x}^{\m}M_1$ holds (by Lemma \ref{lem:phasederiunivlow}). Consequently, we conclude that the integral in \eqref{eq:outlineint} is $O(t^{-\frac{1}{2}})$ when $t=M_1/2$. Otherwise, we can use the stationary phase theorem instead. When $V$ is an even function, the proof is completed since $\Phi_{+,+}(x,x')=\Phi_{+,-}(x,-x')$ in this case.

When $V$ is not an even function, then we need more works to treat the $(+,+)$-case. Performing Taylor's expansion of $\Phi_{+,+}=-t\l^2+\int_{0}^x\sqrt{\l^2-V(s)}ds+\int_0^{x'}\sqrt{\l^2-V(s)}ds$ at $\l=0$ to obtain a similar expansion as \eqref{eq:lowTaylor} replacing $M_j$ by $\tilde{M}_j=\int_{0}^x|V(s)|^{-\frac{2j-1}{2}}ds+\int_{0}^{x'}|V(s)|^{-\frac{2j-1}{2}}ds$ ($j=1,2$). For an anisotropic potential $V(x)=-\jap{x}^{-\m}$ for $x\gg 1$ and $V(x)=-2\jap{x}^{-\m}$ for $x\ll -1$, we have $\tilde{M}_1\approx \jap{x}^{1+\frac{\m}{2}}-2^{\frac{1}{2}}\jap{x'}^{1+\frac{\m}{2}}$ and $\tilde{M}_2\approx \jap{x}^{1+\frac{3\m}{2}}-2^{\frac{3}{2}}\jap{x'}^{1+\frac{3\m}{2}}$  for $x\gg 1$ and $x'\ll -1$. This implies that $\tilde{M}_2\ll \jap{\max(|x|,|x'|)}^{\m}\tilde{M}_1$ may occur and therefore the method for $\Phi_{+,-}$ described above does not work here. To overcome this difficulty, we use refined bound $|b_{+,-}(\l)|\lesssim \l^{-2}\max(|x|^{-2},|x'|^{-2})$  for $x>0$ and $x'<0$ proved in Theorem \ref{2509041030}. When the worst scenario $t=\tilde{M}_1/2$ occurs, then we split the integral \eqref{eq:outlineint} into 
\begin{align*}
\int_0^{\sqrt{t}^{-1}}b_{+,-}(\l)e^{i\Phi_{+,-}(\l)} d\l+\int_{\sqrt{t}^{-1}}^{\infty}b_{+,-}(\l)e^{i\Phi_{+,-}(\l)} d\l.
\end{align*}
We can see that the first integral is bounded by $t^{-\frac{1}{2}}$ and the second integral is bounded by 
\begin{align*}
\int_{\sqrt{t}^{-1}}^{\infty}\l^{-2}\max(|x|^{-2},|x'|^{-2}) d\l=\sqrt{t}\max(|x|^{-2},|x'|^{-2}).
\end{align*}
Since we have assumed $t=\tilde{M}_1/2$ and $\tilde{M}_2\ll \jap{\max(|x|,|x'|)}^{\m}\tilde{M}_1$, this bound is sufficient to prove the dispersive estimate. For the detail, see Subsection \ref{subsec:off-diagonal}.

\subsection{Further observation}

\underline{\textbf{Derivation of the local decay estimate}}: Our method described above recovers Yafaev's result \cite{Y-1}, which states that \eqref{eq:localtime} holds for $N=1$. To see this, we observe $b_{+,-}=O(\l)$ uniformly in $x,x'\in K$, where $K\subset \R$ is a fixed compact set. As mentioned before, we have $\Phi_{+,-}\approx -t\l^2+\l(x-x')$ for $\l\gg \jap{x}^{-\m/2}$ and \eqref{eq:lowTaylor} for $\l\lesssim \jap{x}^{-\m/2}$. In both cases, the second derivative of $\Phi_{+,-}$ with respect to $\l$ does not vanish for $t\gg 1$ and $x,x'\in K$ since $M_1=\int_{x'}^x|V(s)|^{-\frac{1}{2}}ds$ is uniformly bounded there. Now the stationary phase theorem with a vanishing symbol implies that the integral in \eqref{eq:outlineint} is $O(t^{-1})$ for $x,x'\in K$.

\noindent\underline{\textbf{Dispersive estimate for positive repulsive potentials}}: When we consider positive repulsive $P=-\pa_x^2+c\jap{x}^{-\mu}$ with $c>0$ and $0<\m<2$, we could use our strategy in the regime $\l\gg \max(\jap{x}^{-\mu},\jap{x'}^{-\mu})$. On the other hands, for $\l\ll \max(\jap{x}^{-\mu},\jap{x'}^{-\mu})$, Yafaev's result \cite[\S 7]{Y} (or \cite[Theorem 1.1]{N1}) suggests that the spectral projection should not oscillate as $\l\to 0$. Thus, the usual stationary phase theorem with the phase $-t\l^2$ would prove the dispersive estimate in this regime (see also \cite[\S 2.1]{BTVZ}).
When $\l\sim \max(\jap{x}^{-\mu},\jap{x'}^{-\mu})$, the phase function $\Phi_{+,-}$ has turning points, which are also called caustics in the WKB theory. We could use Airy function approximation there as in \cite[\S 2.2]{BTVZ}. It would be an interesting problem to justify this heuristic argument.

\noindent\underline{\textbf{Generalization to the higher dimensional cases}}: The dispersive estimate \eqref{eq:dispersive} in our model (e.g. $V$ satisfies Assumption \ref{assump:V}) is known to fail\footnote{The authors are grateful to Haruya Mizutani for pointing out it.} at least when $d\geq 3$. In fact, Fournais-Skibsted \cite[Theorem 1.4]{FS} showed that $|t|^{-1}\lesssim \|\chi \g(P)e^{-itP}E_{\mathrm{ac}}(P)\chi\|_{L^2\to L^2}$ as $|t|\to \infty$ for some $\chi\in C_c^{\infty}(\R^d)$ and $\g\in C_c^{\infty}(\R)$ supported near $0$. On the other hand, if $V$ is sufficiently smooth, then the Sobolev embedding theorem shows $|t|^{-1}\lesssim \|\chi \g(P)e^{-itP}E_{\mathrm{ac}}(P)\chi\|_{L^2\to L^2}\lesssim \|e^{-itP}E_{\mathrm{ac}}(P)\|_{L^1\to L^{\infty}}$. Thus, the possible time decay rate is at most $|t|^{-1}$ in any dimension. Nevertheless, there is a possibility to have Strichartz estimates for some pairs $(q,r)$, which are more important in applications to non-linear theory. This is deferred to future works.

\subsection{Notations}\label{2510191800}
We write $\N =\{1, 2, \dots\}$ and $\N_0 = \N \cup \{0\}$.

\subsection*{Acknowledgment}
AH is partially supported by FoPM, WINGS Program, the University of Tokyo and JSPS Research Fellowship for Young Scientists KAKENHI Grant Number JP25KJ0736. KT is partially supported by JSPS KAKENHI Grant Number JP23K13004. He is grateful to Haruya Mizutani for suggesting this problem.

\section{Preliminary: Quantitative estimates of oscillatory integrals}\label{section:prelim}

\subsection{Partition of unity}

The next lemma is very useful for our purpose. The proof is a direct adaption of Leibniz's rule and omitted here. 

\begin{lemma}[Partition of unity]\label{lem:PUest}
Let $M>0$ and $\chi\in C_c^{\infty}(\R;[0,1])$ such that $\chi(\l)=1$ for $|\l|\leq 1$ and $\chi(\l)=0$ for $|\l|\geq 2$. Then there exists $c>0$ \textbf{depending only on $\chi$} such that
\begin{align*}
&\sum_{j=1}^2\left(\|a_j\|_{L^{\infty}}+ \|\l\pa_{\l}a_j\|_{L^{\infty}}\right)\leq  c(\|a\|_{L^{\infty}}+ \|\l\pa_{\l}a\|_{L^{\infty}}),\\
&\sum_{j=1}^2\left(\|\l^{-1}a_j\|_{L^{\infty}}+ \|\pa_{\l}a_j\|_{L^{\infty}}\right)\leq  c(\|\l^{-1}a\|_{L^{\infty}}+ \|\pa_{\l}a\|_{L^{\infty}}),
\end{align*}
where $a_1(\l):=\chi(M\l)a(\l)$ and $a_2(\l):=(1-\chi(M\l))a(\l)$, for all $a\in C^1((0,\infty))$ when the right hand side is bounded. 
\end{lemma}

\subsection{Estimates for oscillatory integrals}

Here we consider the oscillatory integral of the form
\begin{align*}
I(a):=\int_0^{\infty}a(\l)e^{i\Phi(\l)}d\l,
\end{align*}
where $\Phi$ is a real-valued $C^2$-function and $a$ is a $C^1$-function on $(0,\infty)$. The purpose of this section is to give some quantitative estimates of $I(a)$ based on the stationary phase argument. 

\begin{lemma}[Quantitative stationary phase]\label{Lem:Qstph1}
Let $a$ be a $C^1$-function on $(0,\infty)$.
Suppose that $M,C_1,C_2,m>0$ satisfy either 
\begin{align}\label{eq:stphaseas1}
|\pa_{\l}\Phi(\l)|\geq C_1m,\quad |\pa_{\l}^2\Phi(\l)|\leq C_2\frac{m}{\l}\quad \text{for $\l\in \supp a$}, \quad \supp a\subset [0,m/M]
\end{align}
or
\begin{align}\label{eq:stphaseas2}
|\pa_{\l}\Phi(\l)|\geq C_1M|\l-m|,\quad |\pa_{\l}^2\Phi(\l)|\leq C_2M\quad \text{for $\l\in \supp a$}.
\end{align}
Then, there exists $C>0$ \textbf{depending only on $C_1$ and $C_2$} such that
\begin{align}\label{eq:Qststate1}
|I(a)|\leq M^{-\frac{1}{2}}C( \|a\|_{L^{\infty}}+ \|\l\pa_{\l}a\|_{L^{\infty}})
\end{align}
when the right hand side is bounded.

\end{lemma}

\begin{remark}
This lemma is a quantitative version of the stationary phase theorem, which asserts that if $\Phi'(0)=0$ and $\Phi^{''}(0)\sim M$ holds, and $a$ is a smooth function supported around $0$, then $I(a)=O(M^{-\frac{1}{2}})$.
\end{remark}

\begin{proof}

Let $\chi\in C_c^{\infty}(\R;[0,1])$ such that $\chi(\l)=1$ for $|\l|\leq 1$ and $\chi(\l)=0$ for $|\l|\geq 2$. 

\noindent$(i)$ We consider the case where \eqref{eq:stphaseas1} holds.
When $m\leq M^{\frac{1}{2}}$, then $|I(a)|\leq  \frac{m}{M}\|a\|_{L^{\infty}}\leq M^{-\frac{1}{2}}\|a\|_{L^{\infty}}$ by the support property of $a$. Thus we assume 
\begin{align*}
m\geq M^{\frac{1}{2}}.
\end{align*}
in the following.
Set
\begin{align*}
a_1(\l):=&a(\l)\chi(M^{\frac{1}{2}}\l),\quad a_2(\l):=a(\l)(1-\chi(M^{\frac{1}{2}}\l)).
\end{align*}
Thanks to Lemma \ref{lem:PUest}, it suffices to estimate $I(a_1)$ and $I(a_1)$ separately.

Clearly, we have $|I(a_1)|\leq 2M^{-\frac{1}{2}}\|a\|_{L^{\infty}}$ by the support property of $a_1$. On the other hands, the integration by parts yields
\begin{align*}
|I(a_2)|=&\left|i\int_{M^{-\frac{1}{2}}}^{\frac{m}{M}}\left(\frac{\pa_{\l}^2\Phi(\l)a_2(\l)}{|\pa_{\l}\Phi(\l)|^2}+\frac{\pa_{\l}a_2(\l)}{|\pa_{\l}\Phi(\l)|}  \right)   e^{i\Phi(\l)}d\l\right|\\
\leq&\int_{M^{-\frac{1}{2}}}^{\frac{m}{M}}\frac{C_2m\l^{-1}}{C_1^2m^2}d\l\|a_2\|_{L^{\infty}}+  \int_{M^{-\frac{1}{2}}}^{\frac{m}{M}}\frac{1}{\l C_1m}     d\l\|\l\pa_{\l}a_2\|_{L^{\infty}}\\
=&m^{-1}\left(C_2C_1^{-2}\log\left(mM^{-\frac{1}{2}}\right)\|a_2\|_{L^{\infty}}+C_1^{-1}\log\left(mM^{-\frac{1}{2}}\right)\|\l\pa_{\l}a_2\|_{L^{\infty}} \right).
\end{align*}
Since $m\geq M^{\frac{1}{2}}$ and $c_0:=\sup_{0<s\leq 1}s\log s^{-1}<\infty$, we have $m^{-1}\log(mM^{-\frac{1}{2}})\leq c_0M^{-\frac{1}{2}}$ and hence
\begin{align*}
|I(a_2)|\leq c_0M^{-\frac{1}{2}}(C_2C_1^{-2}\|a_2\|_{L^{\infty}}+C_1^{-1}\|\l\pa_{\l}a_2\|_{L^{\infty}}).
\end{align*}
Now we conclude $|I(a)|\leq |I(a_1)|+|I(a_2)|\lesssim M^{-\frac{1}{2}}( \|a\|_{L^{\infty}}+ \|\l\pa_{\l}a\|_{L^{\infty}})$.

\vspace{2mm}

\noindent$(ii)$ Next, we turn to the case where \eqref{eq:stphaseas2} holds.
Set
\begin{align*}
a_1(\l):=&a(\l)\chi(M^{\frac{1}{2}}\l),\quad a_2(\l):=a(\l)(1-\chi(M^{\frac{1}{2}}\l)),\\
a_3(\l):=&a_2(\l)\chi(M^{\frac{1}{2}}(\l-m)),\quad a_4(\l):=a_2(\l)(1-\chi(M^{\frac{1}{2}}(\l-m))).
\end{align*}
We note $|I(a_1)|,|I(a_3)|\leq 2M^{-\frac{1}{2}}\|a\|_{L^{\infty}}$ by the support properties of $a_1$ and $a_3$.


First, we consider the case where $m^2M\leq \frac{1}{2}$. By Lemma \ref{lem:PUest}, it suffices to estimate $I(a_2)$ since $a=a_1+a_2$. We observe
\begin{align*}
|\pa_{\l}\Phi(\l)|\geq C_1M|\l-m|\geq \frac{1}{2}C_1M|\l|\quad (\text{for $\l\in \supp a_2$}).
\end{align*}
The integration by parts yields
\begin{align*}
|I(a_2)|=&\left|i\int_{M^{-\frac{1}{2}}}^{\infty}\left(\frac{\pa_{\l}^2\Phi(\l)a_2(\l)}{|\pa_{\l}\Phi(\l)|^2}+\frac{\pa_{\l}a_2(\l)}{|\pa_{\l}\Phi(\l)|}  \right)   e^{i\Phi(\l)}d\l\right|\\
\leq&\int_{M^{-\frac{1}{2}}}^{\infty}\frac{4C_2M}{C_1^2M^2\l^2}d\l\|a_2\|_{L^{\infty}}+  \int_{M^{-\frac{1}{2}}}^{\infty}\frac{2}{\l\cdot C_1M\l}     d\l\|\l\pa_{\l}a_2\|_{L^{\infty}}\\
=&M^{-\frac{1}{2}}\left(\frac{4C_2}{C_1^2}\|a_2\|_{L^{\infty}}+\frac{2}{C_1}\|\l\pa_{\l}a_2\|_{L^{\infty}} \right),
\end{align*}
which proves \eqref{eq:Qststate1} when $m^2M\geq \frac{1}{2}$.

Next we consider $m^2M\geq \frac{1}{2}$. By the identity $a=a_1+a_3+a_4$ and Lemma \ref{lem:PUest}, we only need to estimate $I(a_4)$.
We write $K:=\{\l\in (0,\infty)\mid |\l|\geq M^{-\frac{1}{2}},\,\, |\l-m|\geq M^{-\frac{1}{2}}\}$
The integration by parts yields
\begin{align*}
|I(a_4)|=&\left|i\int_{K}\left(\frac{\pa_{\l}^2\Phi(\l)a_4(\l)}{|\pa_{\l}\Phi(\l)|^2}+\frac{\pa_{\l}a_4(\l)}{|\pa_{\l}\Phi(\l)|}  \right)   e^{i\Phi(\l)}d\l\right|\\
\leq&\int_{K}\frac{C_2M}{C_1^2M^2|\l-m|^2}d\l\|a_4\|_{L^{\infty}}+  \int_{K}\frac{1}{\l\cdot C_1M|\l-m|}     d\l\|\l\pa_{\l}a_4\|_{L^{\infty}}.
\end{align*}
Now we observe that $\int_{K}|\l-m|^{-2}d\l\leq \int_{|\l-m|\geq M^{\frac{1}{2}}}|\l-m|^{-2}d\l\leq M^{-\frac{1}{2}}$ and 
\begin{align*}
\int_{K}\l^{-1}|\l-m|^{-1}d\l=&m^{-1}\int_{\l\geq m^{-1}M^{-\frac{1}{2}},|\l-1|\geq m^{-1}M^{-\frac{1}{2}}}\l^{-1}|\l-1|^{-1}d\l\\
\leq& C_3m^{-1}|\log(m^{-1}M^{-\frac{1}{2}})|=M^{\frac{1}{2}}\cdot C_3m^{-1}M^{-\frac{1}{2}}|\log(m^{-1}M^{-\frac{1}{2}})|\\
\leq&C_4M^{\frac{1}{2}},
\end{align*}
where $C_3,C_4>0$ are universal constants and we have used $\sup_{s\in (0,\sqrt{2})}|s\log (s^{-1})|<\infty$ and $0<m^{-1}M^{-\frac{1}{2}}\leq \sqrt{2}$. In this way, we conclude
\begin{align*}
|I(a_4)|\leq M^{-\frac{1}{2}}\left( \frac{C_2}{C_1^2}\|a_4\|_{L^{\infty}}+\frac{C_4}{C_1}\|\l\pa_{\l}a_4\|_{L^{\infty}}\right),
\end{align*}
which proves \eqref{eq:Qststate1} when $m^2M\geq \frac{1}{2}$.
\end{proof}

\begin{lemma}[Quantitative degenerate stationary phase]\label{lem:Qstph2}
Let $a$ be a $C^1$-function on $(0,\infty)$.
Suppose that $M,C_1,C_2>0$ and $m\in \R$ satisfy
\begin{align*}
|\pa_{\l}\Phi(\l)|\geq C_1M\l|\l^2-m|,\quad |\pa_{\l}^2\Phi(\l)|\leq C_2M(\l^2+|m|)
\end{align*}
for $\l\in \supp a$. Then, there exists $C>0$ \textbf{depending only on $C_1$, $C_2$ and $C_3$} such that
\begin{align}\label{eq:Qststate3}
|I(a)|\leq M^{-\frac{1}{2}}C( \|\l^{-1} a\|_{L^{\infty}}+ \|\pa_{\l}a\|_{L^{\infty}} )
\end{align}
when the right hand side is bounded.
\end{lemma}

\begin{remark}
This lemma is a quantitative version of the degenerate stationary phase theorem with a vanishing symbol, which asserts that if $\Phi^{(j)}(0)= 0$ ($j=1,2,3$) and $\Phi^{(4)}(0)\sim M$ holds, and $a$ is a smooth function supported around $0$ with $a(\l)=O(\l)$ ($\l\to 0$), then $I(a)=O(M^{-\frac{1+1}{4}})=O(M^{-\frac{1}{2}})$. Without the assumption  $a(\l)=O(\l)$ as $\l\to 0$, the optimal decay would be $I(a)=O(M^{-\frac{1}{4}})$.
\end{remark}

\begin{proof}
Let $\chi\in C_c^{\infty}(\R;[0,1])$ be such that $\chi(\l)=1$ for $|\l|\leq 1$ and $\chi(\l)=0$ for $|\l|\geq 2$. Set
\begin{align*}
a_1(\l):=&a(\l)\chi(M^{\frac{1}{4}}\l),\quad a_2(\l):=a(\l)(1-\chi(M^{\frac{1}{4}}\l)),\\
a_3(\l):=&a_2(\l)\chi(M^{\frac{1}{4}}(\l-|m|^{\frac{1}{2}})),\quad a_4(\l):=a_2(\l)(1-\chi(M^{\frac{1}{4}}(\l-|m|^{\frac{1}{2}}))).
\end{align*}
By the support properties of $a_1$ and $a_3$, we have $|I(a_1)|\leq \|\l\|_{L^1(0,2M^{-\frac{1}{4}})} \|\l^{-1}a\|_{L^{\infty}}=2M^{-\frac{1}{2}}\|\l^{-1}a\|_{L^{\infty}}$ and  $|I(a_1)|\leq 2M^{-\frac{1}{2}}\|\l^{-1}a\|_{L^{\infty}}$.

First, we consider the case where $m^2M\leq \frac{1}{4}$. Due to the decomposition $a=a_1+a_2$ and Lemma \ref{lem:PUest}, we only need to estimate $I(a_2)$. Then
\begin{align*}
|\pa_{\l}\Phi(\l)|\geq C_1M\l|\l^2-m|\geq \frac{1}{2}C_1M\l^3,\quad |\pa_{\l}^2\Phi(\l)|\leq \frac{3}{2}C_2M\l^2
\end{align*}
for $\l\in \supp a_2$.
The integration by parts yields
\begin{align*}
|I(a_2)|=&\left|i\int_{M^{-\frac{1}{4}}}^{\infty}\left(\frac{\pa_{\l}^2\Phi(\l)a_2(\l)}{|\pa_{\l}\Phi(\l)|^2}+\frac{\pa_{\l}a_2(\l)}{|\pa_{\l}\Phi(\l)|}  \right)   e^{i\Phi(\l)}d\l\right|\\
\leq&\int_{M^{-\frac{1}{4}}}^{\infty}\frac{6C_2M\l^2}{C_1^2M^2\l^6}\cdot \l d\l\|\l^{-1}a_2\|_{L^{\infty}}+  \int_{M^{-\frac{1}{4}}}^{\infty}\frac{2}{C_1M\l^3}     d\l\|\pa_{\l}a_2\|_{L^{\infty}}\\
=&M^{-\frac{1}{2}}\left(\frac{3C_2}{C_1^2}\|\l^{-1}a_2\|_{L^{\infty}}+\frac{1}{C_1}\|\pa_{\l}a_2\|_{L^{\infty}} \right).
\end{align*}

Next we consider $m^2M\geq \frac{1}{4}$. By the change of variable, we have
\begin{align*}
I(a)=|m|^{\frac{1}{2}}\int_0^{\infty}\tilde{a}(\l)e^{i\tilde{\Phi}(\l)}d\l=:|m|^{\frac{1}{2}}\tilde{I}(\tilde{a}),\quad \tilde{a}(\l):=a(|m|^{\frac{1}{2}}\l),\quad \tilde{\Phi}(\l)=\Phi(|m|^{\frac{1}{2}}\l).
\end{align*}
Setting $\tilde{M}=Mm^2\geq 1/4$, we have 
\begin{align*}
|\pa_{\l}\tilde{\Phi}(\l)|\geq C_1\tilde{M}\l|\l^2-(\sgn m)|,\quad |\pa_{\l}^2\tilde{\Phi}(\l)|\leq C_2\tilde{M}(\l^2+1).
\end{align*}
In particular, we have $|\pa_{\l}\tilde{\Phi}(\l)|\geq \frac{C_1}{2}\tilde{M}\l$ for $0\leq \l\leq 1/2$, $|\pa_{\l}\tilde{\Phi}(\l)|\geq \frac{C_1}{4}\tilde{M}|\l-(\sgn m)|$ for $1/4\leq \l\leq 4$, and $|\pa_{\l}\tilde{\Phi}(\l)|\geq C_1\tilde{M}\l^3$ for $\l\geq 2$.

Define
\begin{align*}
\tilde{a}_1(\l):=\chi(4\l)\tilde{a}(\l),\quad \tilde{a}_2(\l):=(1-\chi(4\l))\chi(\l/2)\tilde{a}(\l),\quad \tilde{a}_3(\l):=\tilde{a}(\l)-\tilde{a}_1(\l)-\tilde{a}_2(\l).
\end{align*}
To estimate an integral regarding $\tilde{a}_1$, we write 
\begin{align*}
\tilde{a}_1(\l)=\chi(\tilde{M}^{\frac{1}{2}}\l)\tilde{a}_1+(1-\chi(\tilde{M}^{\frac{1}{2}}\l))\tilde{a}_1=:\tilde{a}_{11}(\l)+\tilde{a}_{12}(\l).
\end{align*}
Then it turns out from $\|\l^{-1}\tilde{a}_{11}\|_{L^{\infty}}\leq |m|^{\frac{1}{2}}\|\l^{-1}a\|_{L^{\infty}}$ that
\begin{align*}
|m|^{\frac{1}{2}}|\tilde{I}(\tilde{a}_{11})|\leq |m|^{\frac{1}{2}}\cdot \underbrace{4\tilde{M}^{-\frac{1}{2}}}_{\mathclap{\text{the volume bound of $\supp \tilde{a}_{11}$}}}\cdot \|\l^{-1}\tilde{a}_{11}\|_{L^{\infty}}\leq 4|m|\tilde{M}^{-\frac{1}{2}}\|\l^{-1}a\|_{L^{\infty}}
=4M^{-\frac{1}{2}}\|\l^{-1}a\|_{L^{\infty}}.
\end{align*}
On the other hand, the integration by parts yields\footnote{The following estimates are not optimal for $\tilde{M}\gg 1$ (but optimal for $\tilde{M}\sim 1$). In fact, we have the bound \begin{align*}\lesssim |m|\tilde{M}^{-1}|\log(\tilde{M})|(\|\l^{-1}a\|_{L^{\infty}}+\|\pa_{\l}a\|_{L^{\infty}}).\end{align*}}
\begin{align*}
|m|^{\frac{1}{2}}|\tilde{I}(\tilde{a}_{12})|\leq& |m|^{\frac{1}{2}}\int_0^{\infty}\left(\frac{|\pa_{\l}^2\tilde{\Phi}(\l)||\tilde{a}_{11}(\l)|}{|\pa_{\l}\tilde{\Phi}(\l)|^2}+\frac{|\pa_{\l}\tilde{a}_{11}(\l)|}{|\pa_{\l}\tilde{\Phi}(\l)|}\right)  d\l\\
\leq&|m|^{\frac{1}{2}}\int_{2\tilde{M}^{-\frac{1}{2}}\leq \l\leq 1/2}\left(\frac{4C_2\tilde{M}(\l^2+1)}{C_1^2\tilde{M}^2\l}\cdot (\l^{-1}\tilde{a}_{11}(\l))+\frac{2|\pa_{\l}\tilde{a}_{11}(\l)|}{C_1\tilde{M}\l}\right)  d\l\\
\leq&|m|^{\frac{1}{2}}\tilde{M}^{-\frac{1}{2}}C_3\left(\frac{C_2}{C_1^2}+\frac{1}{C_1}\right)(\|\l^{-1}\tilde{a}_{11}\|_{L^{\infty}}+\|\pa_{\l}\tilde{a}_{11}\|_{L^{\infty}})\\
\leq&|m|\tilde{M}^{-\frac{1}{2}}C_3\left(\frac{C_2}{C_1^2}+\frac{1}{C_1}\right)(\|\l^{-1}a\|_{L^{\infty}}+\|\pa_{\l}a\|_{L^{\infty}})\\
\leq&4M^{-\frac{1}{2}}C_3\left(\frac{C_2}{C_1^2}+\frac{1}{C_1}\right)(\|\l^{-1}a\|_{L^{\infty}}+\|\pa_{\l}a\|_{L^{\infty}}),
\end{align*}
where $C_3>0$ is a universal constant and we use Lemma \ref{lem:PUest}. To sum up, we obtain
\begin{align*}
|m|^{\frac{1}{2}}\left|\int_0^{\infty}\tilde{a}_1(\l)e^{i\tilde{\Phi}(\l)}d\l\right|\leq C_4M^{-\frac{1}{2}}(\|\l^{-1}a\|_{L^{\infty}}+\|\pa_{\l}a\|_{L^{\infty}}),
\end{align*}
where $C_4>0$ depends only on $C_1$ and $C_2$. Estimates for $\tilde{a}_2$ and $\tilde{a}_3$ can be deduced similarly. Consequently, we obtain \eqref{eq:Qststate3}.
\end{proof}

\section{Asymptotic expansions of the spectral projection}\label{2508241544}
In this section we show that the spectral projection (more precisely the spectral density, see below) of the Schr\"odinger operator with slowly decaying attractive potentials can be written as a WKB state. The key step is the construction of the Jost functions which have different asymptotic behavior compared to the Schr\"odinger operator with fast decaying potentials. In addition to the construction, we also need symbolic decay estimates on the coefficients of the Jost functions with respect to the energy parameter. We split the proof into several steps as follws. In Subsection \ref{2510231112}, we construct solutions of a vector-valued ODE to which the Schr\"odinger equation with inverse square type potentials are transformed. Then in Subsection \ref{2510231245}, we construct Jost functions for the Schr\"odinger equation with inverse square type potentials. We also investigate their asymptotic properties in detail. In Subsection \ref{2510231114} we construct our Jost functions for the Schr\"odinger equation with slowly decaying potentials using results in the previous two subsections. Finally, in Subsection \ref{2510231110}, we give an expression of the outgoing/incoming resolvent and the proof of our main Theorem \ref{2509041030}.



For $\lambda > 0$, we define the spectral density of $P=-\pa_x^2+V(x)$ by
\begin{align}\label{def:spdensity}
\tilde{E} (\lambda) := \frac{R(\lambda^2 +i0)-R(\lambda^2 -i0)}{2\pi i} \cdot 2\lambda = E' _P (\lambda^2) \cdot 2\lambda
\end{align}
so that
\begin{align*}
e^{-itP} E_{\mathrm{ac}}(P)=e^{-itP} E_P (\R_{ \ge 0}) = \int_{0}^{\infty} e^{-it\mu} E' _{P} (\mu) d\mu = \int_{0}^{\infty} e^{-it\lambda^2} \tilde{E} (\lambda) d\lambda,
\end{align*}
where $E_P (\cdot)$ is the spectral projection associated with $P$. As is mentioned in the introduction, Theorem \ref{2509021637} below ensures that $P$ has no embedded eigenvalues and hence $\tilde{E}'(\l)$ is well-defined.
The following theorem is the main result in this section.

\begin{thm}\label{2509041030}
Suppose that $V$ satisfies Assumption \ref{assump:V}. Then, for $\lambda >0$, we have 
\begin{align}
\tilde{E} (\lambda, x, x') = \sum_{\sigma_{1}, \sigma_{2} \in \{\pm\}} b_{\sigma_{1}, \sigma_{2}} (\lambda, x, x') e^{iS_{\sigma_{1}, \sigma_{2}} (\lambda, x, x')}, \label{2509041043}
\end{align}
where $S_{\sigma_{1}, \sigma_{2}}$ and $b_{\sigma_{1}, \sigma_{2}}$ satisfy
\begin{align}
&S_{\sigma_{1}, \sigma_{2}} (\lambda, x, x') = \sigma_{1} \int_{0}^{x} \sqrt{\lambda^2 -V(s)} ds + \sigma_{2} \int_{0}^{x'} \sqrt{\lambda^2 -V(s)} ds, \label{2510201028}\\
&|\partial^{\alpha} _{\lambda} b_{\sigma_{1}, \sigma_{2}} (\lambda, x, x')| \lesssim \min (\lambda^{1-\alpha} \jap{x}^{\frac{\mu}{4}} \jap{x'}^{\frac{\mu}{4}}, \lambda^{-\alpha}) \label{2510201029}
\end{align}
for all $x, x' \in \R$, $\lambda >0$ and $\alpha = 0, 1, 2$. Moreover, 
\begin{align}\label{eq:ampaddecay}
|b_{+,+}(\l,x,x')|+|b_{-,-}(\l,x,x')|\lesssim |\l|^{-2}\max(|x|^{-2},|x'|^{-2})
\end{align}
for $x> 0$, $x'< 0$ and $\l>0$.

\end{thm}

\begin{remark}\label{rem:spprojWKBxx'<0}
We have another expression of $\tilde{E}(\l,x,x')$ with amplitudes $\tilde{b}_{\s_1,\s_2}$ that satisfies \eqref{2510201029} for $x,x'\in \R$ and $\l>0$ and \eqref{eq:ampaddecay} for $x< 0$, $x'>0$ and $\l>0$. 
The proof is a slight modification of the following argument. In fact, we have only to insert the WKB expression of $u_{\pm}$ given in Theorem \ref{2509021637} to \eqref{eq:spprojJost2} instead of \eqref{eq:spprojJost}.
\end{remark}

\begin{remark}
An analogous WKB expression of the microlocalized spectral density for $P=-\Delta_g$ on non-trapping scattering manifolds was derived in \cite[Proposition 1.4]{HZ}. In contrast to this (in one-dimension), our amplitude functions $b_{\s_1,\s_2}$ vanish at $\l=0$ due to the strong negativity of the potential $V$.
\end{remark}

\begin{remark}
Additional decay estimates \eqref{eq:ampaddecay} of the amplitudes $b_{+,+}$ and $b_{-,-}$ are not needed to prove Theorem  \ref{thm:main} if $V$ is an even function. 
\end{remark}

\begin{remark}\label{rem:spdensityWKBext}
An analogous statement remains valid for $\l\gg 1$ even if $V$ may be positive on a bounded interval (e.g. a real-valued function $V$ and $\m>0$ satisfy $|\pa_x^{\a}V(x)|\lesssim \jap{x}^{-\a-\mu}$ for all $\a\in\mathbb{N}_0$). More precisely, we fix $C_0 \ge1$ such that $|V(x)| \le \frac{C_0}{2}$ for all $x \in \R$. Then $y(x,\l)$ is well-defined for $\l\geq C_0$ and $x\in \R$. Moreover, $\tilde{E}(\l,x,x')$ is written as \eqref{2509041043} for $\l\geq C_0$, $x,x'\in \R$ and the corresponding $b_{\s_1,\s_2}$ enjoys the properties \eqref{2510201029} and \eqref{eq:ampaddecay}. The proof is a simple modification of the following argument with the bound
\begin{align*}
\frac{1}{|\lambda^2 -V(x)|^{\alpha}} \lesssim \min \left(\frac{1}{|\lambda|^{2\alpha}}, \langle x \rangle^{\mu \alpha}\right)
\end{align*}
for $\alpha >0$, $x \in \R$ and $|\lambda| \ge C_0$.

\end{remark}

Some of the results below hold under weaker assumptions of $V$, but we do not pursue such generalizations here. The following argument is a refinement of the theory of the asymptotic behavior of solutions to ODEs \cite{Ha}.

\subsection{ODEs with non-oscillating solutions}\label{2510231112}
In this subsection we construct non-oscillating solutions of ODEs to which a Schr\"odinger equation with a short range potential is transformed in the next subsection.

\begin{thm}\label{2508311359}
Let $B \in C^{\infty} (\R_{y} \times (\R_{\lambda} \setminus \{0\}); M_2 (\C)) \cap C (\R_{y} \times \R_{\lambda}; M_2 (\C))$ be a matrix-valued function such that
\begin{align}\label{eq:Bassump}
\|\pa_{\l}^{\a}B(y,\l)\|_{\mathbb{C}^2\to\mathbb{C}^2}\lesssim |\l|^{-\a}\jap{y}^{-2},\quad \|\pa_yB(y,\l)\|_{\mathbb{C}^2\to\mathbb{C}^2}\lesssim \jap{y}^{-2}
\end{align}
for $\a\in\mathbb{N}_0$.
Then there exist solutions $\textbf{w}_{\pm}\in C(\R_y\times \R_{\l};\mathbb{C}^2)\cap C^{\infty}(\R_y\times \R_{\l}\setminus \{0\};\mathbb{C}^2)$ of
\begin{align}\label{2510132029}
\partial_{y} \textbf{w}_{\pm} (y, \lambda) = B(y, \lambda)\textbf{w}_{\pm} (y, \lambda)
\end{align}
which also smoothly depend on $\lambda \in \R \setminus \{0 \}$ and continuously depend on $\lambda \in \R$. Moreover $\textbf{w}_{\pm} (y, \lambda) = (w_{1, \pm} (y, \lambda), w_{2, \pm} (y, \lambda))$ satisfy the following properties:

\vspace{1mm}
\noindent$(i)$ For each $\sigma \in \{\pm\}$, we have $\sup\limits_{y \in \R, \lambda \in \R } \| \textbf{w}_{\sigma} (y, \lambda)\|_{\C^2} < \infty$ and the limits
\begin{align*}
\textbf{w}_{\sigma} (\pm \infty, \lambda) := \lim_{y \to \pm \infty} \textbf{w}_{\sigma} (y, \lambda)
\end{align*}
exist and $\textbf{w}_{\pm} (\pm \infty, \lambda)=\begin{pmatrix}1\\ \pm i\end{pmatrix}$.

\vspace{1mm}
\noindent$(ii)$ We have the uniform convergence:
\begin{align*}
&\lim_{y\to \pm\infty}\sup_{\lambda \in \R} \|\textbf{w}_{\sigma} (\pm \infty, \lambda) - \textbf{w}_{\sigma} (y, \lambda)\|_{\C^2} =0, \quad\lim_{y\to \pm\infty}\sup_{\lambda \in \R} \| \partial_{y} \textbf{w}_{\sigma} (y, \lambda)\|_{\C^2} = 0.
\end{align*}

\vspace{1mm}
\noindent$(iii)$ For  $\a,\b,\c\in\mathbb{N}_0$ satisfying $\b+\c\leq 2$,
\begin{align*}
\| \partial^{\alpha} _{\lambda} \textbf{w}_{\pm} (y, \lambda)\|_{\C^2} \lesssim | \lambda |^{- \alpha},\quad \|\partial^{\alpha} _{\lambda} \textit{\textbf{w}}_{\pm} (\mp\infty, \lambda)\|_{\mathbb{C}^2}\lesssim |\l|^{-\a},\quad \| \pa_y^{\b}\partial^{\c}_{\l} \textbf{w}_{\pm} (y, \lambda)\|_{\mathbb{C}^2} \lesssim | \lambda |^{-\c}\jap{y}^{-\b} 
\end{align*}
uniformly in $y \in \R$ and $\lambda \in \R \setminus \{0\}$.

\vspace{1mm}
\noindent$(iv)$ We further assume that
\begin{align*}
B(y,\l)=-W(y, \lambda)\begin{pmatrix}
\sin y \cos y & \sin^2 y \\
-\cos^2 y & -\sin y \cos y
\end{pmatrix}
\end{align*}
for some $W\in C^{\infty} (\R_{y} \times (\R_{\lambda} \setminus \{0\})) \cap C (\R_{y} \times \R_{\lambda})$ satisfying $|W(y,\l)|\lesssim \jap{y}^{-2}$ and $|\pa_yW(y,\l)|\lesssim \jap{y}^{-3}$ uniformly in $y\in \R$ and $\l\in \R$. Then we have
\begin{align*}
|w_{1,+}(y,\l)+iw_{2,+}(y,\l)|\lesssim \jap{y}^{-2},\quad |w_{1,-}(y',\l)-iw_{2,-}(y',\l)|\lesssim \jap{y'}^{-2}
\end{align*}
uniformly in $y\geq 0$, $y'\leq 0$ and $\l\in \R$.

\end{thm}

\begin{proof}
The following proof is more or less a modification of the fixed point argument in the ODE theory.
We write $\|\cdot \|_{\mathbb{C}^2}=\|\cdot \|$ and $\|\cdot\|_{\mathbb{C}^2\to \mathbb{C}^2}=\|\cdot\|$ in short.
The first inequality in \eqref{eq:Bassump} gives
\begin{align}
\sup_{\lambda \in \R} \int_{-\infty}^{\infty} \|B(y, \lambda)\| dy < \infty, \quad \lim_{|y| \to \infty} \sup_{\lambda \in \R} \|B(y, \lambda)\| =0. \label{2508311406}
\end{align}

We only consider $\textit{\textbf{w}}_{+} (y, \lambda)$ since $\textit{\textbf{w}}_{-} (y, \lambda)$ is constructed similarly. We fix $\lambda \in \R$ and consider the integral equation
\begin{align}\label{eq:winteq}
\textit{\textbf{w}}_{+} (y, \lambda) = 
\begin{pmatrix}
1 \\
i
\end{pmatrix}
 + \int_{\infty}^{y} B(s, \lambda) \textit{\textbf{w}}_{+} (s, \lambda) ds \quad (=: \Phi [\textit{\textbf{w}}_{+} ]).
\end{align}
For $R \ge 1$ we set a complete metric space $X_R := \{\textit{\textbf{w}} \in C \cap L^{\infty} ([R, \infty); \C^2) \mid \|\textit{\textbf{w}}\|_{X_R} \le R \}$, where $\|\cdot\|_{X_R} := \|\cdot\|_{L^{\infty} ([R, \infty); \C^2)}$ and the metric is given by $d(\textit{\textbf{w}}_1, \textit{\textbf{w}}_2) = \|\textit{\textbf{w}}_1 - \textit{\textbf{w}}_2\|_{X_R}$ for $\textit{\textbf{w}}_1, \textit{\textbf{w}}_2 \in X_{R}$. Then we obtain, for $\textit{\textbf{w}}, \tilde{\textit{\textbf{w}}} \in X_{R}$,
\begin{align*}
&\|\Phi[\textit{\textbf{w}}]\|_{X_R} \le \sqrt{2} + \sup_{y \ge R} \left| \int_{\infty}^{y} \|B(s, \lambda)\| \|\textit{\textbf{w}} (s) \| ds \right| \le \sqrt{2} + \|\textit{\textbf{w}}\|_{X_R} \int_{R}^{\infty} \|B(s, \lambda)\| ds, \\
&\|\Phi[\textit{\textbf{w}}] - \Phi [\tilde{\textit{\textbf{w}}}]\|_{X_R} \le \sup_{y \ge R} \left| \int_{\infty}^{y} \|B(s, \lambda)\| \|\textit{\textbf{w}} (s)- \tilde{\textit{\textbf{w}}} (s) \| ds \right| \le \|\textit{\textbf{w}}  - \tilde{\textit{\textbf{w}} }\|_{X_R} \int_{R}^{\infty} \|B(s, \lambda)\|ds.
\end{align*}
Since there exists $R_{\lambda} \ge 2\sqrt{2}$ such that $\int_{R_{\lambda}}^{\infty} \|B(s, \lambda)\| ds < \frac{1}{2}$ by (\ref{2508311406}), $\Phi$ is a contraction on $X_{R_{\lambda}}$. We denote its unique fixed point by $\textit{\textbf{w}}_{+} (\cdot, \lambda) \in X_{R_{\lambda}}$. Then $\textit{\textbf{w}}_{+}$ satisfies \eqref{2510132029} if $y \ge R_{\lambda}$. Now we extend $\textit{\textbf{w}}_{+}$ as a unique solution to \eqref{2510132029} and obtain $\textit{\textbf{w}}_{+} \in C(\R \times \R;\mathbb{C}^2)\cap C^{\infty}(\R\times \R\setminus \{0\};\mathbb{C}^2)$ from the standard theory of ODE with a parameter.

\vspace{1mm}
\noindent$(i)$ 
To show an upper bound of $\|\textit{\textbf{w}}_{+} (y, \lambda)\|_{\C^2}$, we compute
\begin{align}
\frac{d}{dy} \left(e^{-\int_{y}^{\infty} \|B(s, \lambda)\| ds} \int_{y}^{\infty} \|B(s, \lambda)\ \|\textit{\textbf{w}}_{+} (s, \lambda)\| ds \right) \ge -\sqrt{2} \|B(y, \lambda)\| e^{-\int_{y}^{\infty} \|B(s, \lambda)\| ds}, \label{2510151849}
\end{align}
where we have used
\begin{align}
\|\textit{\textbf{w}}_{+} (y, \lambda)\| \le \sqrt{2} + \int_{y}^{\infty} \|B(s, \lambda)\| \|\textit{\textbf{w}}_{+} (s, \lambda)\| ds, \label{2510151859}
\end{align}
which follows from \eqref{eq:winteq}. By integrating \eqref{2510151849} on $[y, \infty)$ we get
\begin{align*}
e^{-\int_{y}^{\infty} \|B(s, \lambda)\| ds} \int_{y}^{\infty} \|B(s, \lambda)\| \|\textit{\textbf{w}}_{+} (s, \lambda)\| ds \le \sqrt{2} \int_{y}^{\infty} \|B(s, \lambda)\| e^{-\int_{s}^{\infty} \|B(t, \lambda)\| dt} ds.
\end{align*}
Inserting this inequality into (\ref{2510151859}), we obtain
\begin{align*}
\|\textit{\textbf{w}}_{+} (y, \lambda)\| \le \sqrt{2} +\sqrt{2} \int_{y}^{\infty} \|B(s, \lambda)\| e^{\int_{y}^{s} \|B(t, \lambda)\| dt} ds
\end{align*}
for any $y \in \R$. 
In particular we have
\begin{align*}
\sup_{(y, \lambda) \in \R^2} \|\textit{\textbf{w}}_{+} (y, \lambda)\| \le \sqrt{2} + \sqrt{2} \cdot \sup_{\lambda \in \R} \int_{-\infty}^{\infty} \|B(s, \lambda)\| e^{\int_{-\infty}^{\infty} \|B(t, \lambda)\| dt} ds < \infty
\end{align*}
by (\ref{2508311406}) and both $\textit{\textbf{w}}_{+} (\infty, \lambda):= \lim_{y \to \infty} \textit{\textbf{w}}_{+} (y, \lambda) = (1 \hspace{2mm}i)$ and
\begin{align*}
\textit{\textbf{w}}_{+} (-\infty, \lambda):= \lim_{y \to -\infty} \textit{\textbf{w}}_{+} (y, \lambda) =
\begin{pmatrix}
1 \\
i
\end{pmatrix}
+  \int_{\infty}^{-\infty} B(s, \lambda) \textit{\textbf{w}}_{+} (s, \lambda) ds
\end{align*}
exist. 

\vspace{1mm}
\noindent$(ii)$ Here, we use
\begin{align*}
\sup_{\lambda \in \R} \left| \int_{\pm y}^{\pm \infty} \|B(s, \lambda)\| ds\right| \to 0 \quad \textrm{as}\quad y \to \infty,
\end{align*}
which also follows from \eqref{eq:Bassump}. We have
\begin{align*}
\|\textit{\textbf{w}}_{+} (\infty, \lambda)-\textit{\textbf{w}}_{+} (y, \lambda)\| &\le \int_{y}^{\infty} \|B(s, \lambda)\| \|\textit{\textbf{w}}_{+} (s, \lambda)\| ds \\
& \le \sup_{(y, \lambda) \in \R^2} \|\textit{\textbf{w}}_{+} (y, \lambda)\| \cdot \sup_{\lambda \in \R} \int_{y}^{\infty} \|B(s, \lambda)\| ds \to 0
\end{align*}
uniformly in $\lambda \in \R$ as $y \to \infty$. Similar calculation yields $\|\textit{\textbf{w}}_{+} (-\infty, \lambda)-\textit{\textbf{w}}_{+} (y, \lambda)\| \to 0$ uniformly in $\lambda \in \R$ as $y \to -\infty$. Finally the second condition in (\ref{2508311406}) is used to deduce
\begin{align*}
\sup_{\lambda \in \R} \|\partial_{y} \textit{\textbf{w}}_{+} (y, \lambda)\| \le \sup_{(y, \lambda) \in \R^2} \|\textit{\textbf{w}}_{+} (y, \lambda)\| \cdot \sup_{\lambda \in \R} \|B(y, \lambda)\| \to 0
\end{align*}
as $|y| \to \infty$.

\vspace{1mm}
\noindent$(iii)$ To prove the first and second inequalities for $\a\in\mathbb{N}_0$, we use
\begin{align*}
\sup_{\lambda \in \R} |\lambda|^{\alpha} \int_{-\infty}^{\infty} \| \partial^{\alpha} _{\lambda} B(y, \lambda)\|_{\C^2 \to \C^2} dy < \infty, 
\end{align*}
which is a consequence of \eqref{eq:Bassump}.
The case $\alpha =0$ is already proved in $(ii)$. 
The result for general $\a$ is established by a standard argument using Gronwall's inequality as follows:
We assume $\|\partial^{\a}_{\lambda} \textit{\textbf{w}}_{\pm} (y, \lambda)\|\lesssim |\l|^{-\a}$ for $\alpha = 0, \dots, j-1$, where $j \in \N$. 
By the Leibniz rule we have
\begin{align*}
\|\partial_{y} \partial^j _{\lambda} \textit{\textbf{w}}_{\pm} (y, \lambda)\|\lesssim\sum_{k=0}^j \|\partial^k _{\lambda} B(y, \lambda)\| \|\partial^{j-k}_{\lambda} \textit{\textbf{w}}_{\pm} (y, \lambda)\|
\end{align*}
and hence
\begin{align*}
\pa_y\|\partial^j _{\lambda} \textit{\textbf{w}}_{\pm} (y, \lambda)\|^2\leq& 2\|\partial^j _{\lambda} \textit{\textbf{w}}_{\pm} (y, \lambda)\|\|\pa_y\partial^j _{\lambda} \textit{\textbf{w}}_{\pm} (y, \lambda)\|\\
\lesssim&\|B(y,\l)\| \|\partial^j _{\lambda} \textit{\textbf{w}}_{\pm} (y, \lambda)\|^2+\sum_{k=1}^j \|\partial^k _{\lambda} B(y, \lambda)\| \|\partial^{j-k}_{\lambda} \textit{\textbf{w}}_{\pm} (y, \lambda)\| \|\partial^{j}_{\lambda} \textit{\textbf{w}}_{\pm} (y, \lambda)\|\\
\lesssim&\left(\|B(y,\l)\|+\sum_{k=1}^j|\l|^{k} \|\partial^k _{\lambda} B(y, \lambda)\|  \right) \|\partial^j _{\lambda} \textit{\textbf{w}}_{\pm} (y, \lambda)\|^2+\sum_{k=1}^j |\l|^{-2j+k}\|\partial^{k}_{\lambda}B(y,\l) \|,
\end{align*}
where we have used the induction hypothesis and the Cauchy-Schwarz inequality in the last line. Similarly to the proof of Gronwall's inequality, we have $\sup_{y \in \R} \|\partial^j _{\lambda} \textit{\textbf{w}}_{\pm} (y, \lambda)\|\le C_{\lambda}$. Therefore, we can change the order of differentiation and integration below:
\begin{align*}
\|\partial^j _{\lambda} \textit{\textbf{w}}_{\pm} (y, \lambda)\| \le \left| \int_{\pm \infty}^{y} \partial^j _{\lambda} (B\textit{\textbf{w}}_{\pm}) (t, \lambda) dt \right| &\lesssim \sum_{k=0}^{j} \left| \int_{\pm \infty}^{y} \|\partial^k _{\lambda} B(t, \lambda)\| \|\partial^{j-k} _{\lambda} \textit{\textbf{w}}_{\pm} (t, \lambda)\| dt \right| \\
& \to 0 \quad \textrm{as} \quad y \to \pm \infty.
\end{align*}
Using the argument analogous to the proof of Gronwall's inequality again, 
\begin{align*}
\|\partial^j _{\lambda} \textit{\textbf{w}}_{\pm} (y, \lambda)\|^2\lesssim \|\partial^j _{\lambda} \textit{\textbf{w}}_{\pm} (y', \lambda)\|^2+\sum_{k=1}^j |\l|^{-2j+k}\int_{y'}^y\|\partial^{k}_{\lambda}B(s,\l) \|ds\lesssim \|\partial^j_{\lambda} \textit{\textbf{w}}_{\pm} (y', \lambda)\|^2+|\l|^{-2j}.
\end{align*}
Taking $y'\to \pm\infty$, we obtain $\|\partial^{j}_{\lambda} \textit{\textbf{w}}_{\pm} (y, \lambda)\|\lesssim |\l|^{-j}$. Since $\textit{\textbf{w}}_{\pm}$ satisfies
\begin{align*}
\textit{\textbf{w}}_{\pm} (\mp\infty, \lambda) =
\begin{pmatrix}
1 \\
\pm i
\end{pmatrix}
+ \int_{\pm\infty}^{\mp\infty} B(s, \lambda) \textit{\textbf{w}}_{\pm} (s, \lambda) ds
\end{align*}
and this yields
\begin{align*}
\|\partial^{\alpha} _{\lambda} \textit{\textbf{w}}_{-} (\infty, \lambda)\| \lesssim \sum_{0 \le \beta \le \alpha} \int_{-\infty}^{\infty} \|\partial^{\beta} _{\lambda} B(s, \lambda)\| \|\partial^{\alpha - \beta} _{\lambda} \textit{\textbf{w}}_{-} (s, \lambda)\| ds \lesssim |\lambda|^{-\alpha}. 
\end{align*}

It remains to prove the third inequality in $(iii)$. We only need to consider $(\b,\c)=(1,0)$, $(2,0)$ or $(1,1)$.
By \eqref{eq:Bassump}, \eqref{2510132029} and $(i)$, we have
\begin{align*}
\|\partial_{y} \textit{\textbf{w}}_{\pm} (y, \lambda)\| &\le \|B(y, \lambda)\|\cdot \|\textit{\textbf{w}}_{\pm} (y, \lambda)\| \lesssim \jap{y}^{-2},
\end{align*}
which proves for $\b=1$ and $\c=0$. Using  \eqref{eq:Bassump} and \eqref{2510132029} again, we have
\begin{align*}
\|\partial_{y}^2 \textit{\textbf{w}}_{\pm} (y, \lambda)\|  &\leq \|\partial_y B(y, \lambda))\|\cdot \|\textit{\textbf{w}}_{\pm} (y, \lambda)\| + \|B(y, \lambda)\|\cdot \|\partial_y \textit{\textbf{w}}_{\pm} (y, \lambda)\| \lesssim \jap{y}^{-2}
\end{align*}
which yields the inequality for $\b =2$ and $\c =0$. 
For the case $\b = \c =1$, we calculate as
\begin{align*}
\|\partial_{y} \partial_{\lambda} \textit{\textbf{w}}_{\pm} (y, \lambda)\| &\lesssim \|\partial_{\lambda} B(y, \lambda)\|\cdot  \|\textit{\textbf{w}}_{\pm} (y, \lambda)\| + \|B(y, \lambda)\|\cdot  \|\partial_{\lambda} \textit{\textbf{w}}_{\pm} (y, \lambda)\|\lesssim |\lambda|^{-1}\jap{y}^{-2} 
\end{align*}
The proof is completed.

\noindent$(iv)$ We only consider $w_{1,+}(y,\l)+iw_{2,+}(y,\l)$ since the other is handled similarly. By \eqref{eq:Bassump} and \eqref{eq:winteq},
\begin{align*}
\sup_{\lambda \in \R} \left\|\textit{\textbf{w}}_{+} (y, \lambda) - 
\begin{pmatrix}
1 \\
i
\end{pmatrix} 
\right\|_{\C^2} \lesssim \jap{y}^{-1}\quad (y\geq 0).
\end{align*}
 Then the explicit formula of $B(s, \lambda)$ gives
\begin{align*}
\textit{\textbf{w}}_{+} (y, \lambda) &= 
\begin{pmatrix}
1 \\
i
\end{pmatrix}
-\int_{\infty}^{y} W(s, \lambda)
\tilde{R}(s)
\begin{pmatrix}
1 \\
i
\end{pmatrix}
ds-\int_{\infty}^{y} W(s, \lambda)
\tilde{R}(s)
\left(\textit{\textbf{w}}_{+} (s, \lambda)-
\begin{pmatrix}
1 \\
i
\end{pmatrix}
\right)
ds \\
&=
\begin{pmatrix}
1 \\
i
\end{pmatrix}
-\int_{\infty}^{y} W(s, \lambda)
\tilde{R}(s)
\begin{pmatrix}
1 \\
i
\end{pmatrix}
ds + O (\jap{y}^{-2})\quad \text{for}\quad y\geq 0
\end{align*}
by setting $\tilde{R}(s)=\begin{pmatrix}
\sin s \cos s & \sin^2 s \\
-\cos^2 s & -\sin s \cos s
\end{pmatrix}$. Therefore we obtain
\begin{align*}
w_{1, +} (y, \lambda) +iw_{2, +} (y, \lambda) =& i\int_{\infty}^{y} W(s, \lambda) e^{2is} ds + O (\jap{y}^{-2})\\
=&\frac{1}{2}W(y,\l)-\frac{1}{2}\int_{\infty}^{y} (\pa_sW)(s, \lambda) e^{2is} ds + O (\jap{y}^{-2})=O(\jap{y}^{-2}) 
\end{align*}
for $y \geq 0$, where we use the integration by parts. 
\end{proof}

\subsection{ODEs with simple oscillating solutions}\label{2510231245}

In this subsection, we construct solutions to the Schr\"odinger equation with short-range potentials that exhibit oscillations at infinity. The key idea is to transform the solutions of the Schr\"odinger equation into those of ODEs with non-oscillatory solutions, to which the results of the previous subsection can be applied.

\begin{thm}\label{2509021616}
Let $W \in C^{\infty} (\R_{y} \times (\R_{\lambda} \setminus \{0\}); \R) \cap C (\R_{y} \times \R_{\lambda}; \R)$ be a real-valued function such that 
\begin{align}
|\pa_{\l}^{\a}W(y,\l)|\lesssim |\l|^{-\a}\jap{y}^{-2},\quad |\pa_yW(y,\l)|\lesssim \jap{y}^{-3} \label{2509021619}
\end{align}
for $\a\in\mathbb{N}_0$.
Then there exist solutions $v_{\pm}\in C(\R_y\times \R_{\l})\cap C^{\infty}(\R_y\times \R_{\l}\setminus \{0\})$ of
\begin{align}
-\partial^2 _{y} v_{\pm} (y, \lambda) +W(y, \lambda)v_{\pm} (y, \lambda) = v_{\pm} (y, \lambda). \label{2509021621}
\end{align}
Moreover, we have the following:

\vspace{1mm}
\noindent$(i)$ There exist $a_{\s_1,\s_2}\in C(\R_y\times \R_{\l})\cap C^{\infty}(\R_y\times \R_{\l}\setminus \{0\})$ $(\s_1,\s_2\in \{\pm\})$ such that
\begin{align}
v_{\pm} (y, \lambda) = a_{\pm, +} (y, \lambda) e^{iy} + a_{\pm, -} (y, \lambda) e^{-iy}, \label{2509021633}
\end{align}
and the limits $\displaystyle a_{\s_1,\s_2}(\pm\infty,\l):=\lim_{y\to \pm\infty}a_{\s_1,\s_2}(y,\l)$ exist and the convergence is uniform in $\l$. Furthermore,
\begin{align}
a_{+,+}(\infty,\l)=a_{-,-}(-\infty,\l)=1,\quad a_{+,-}(\infty,\l)=a_{-,+}(-\infty,\l)=0. \label{eq:apmasym}
\end{align}

\vspace{1mm}
\noindent$(ii)$ For  $\a,\b,\c\in\mathbb{N}_0$ satisfying $\b+\c\leq 2$,
\begin{align*}
|\partial^{\alpha} _{\lambda} a_{\pm,\pm} (y, \lambda)| \lesssim |\lambda|^{-\alpha},\quad |\partial^{\alpha} _{\lambda} a_{\s_1,\s_2}(\pm\infty, \lambda)|\lesssim |\l|^{-\a},\quad | \pa_y^{\b}\partial^{\c}_{\l} a_{\s_1,\s_2} (y, \lambda)| \lesssim | \lambda |^{-\c}\jap{y}^{-\b}.
\end{align*}
For $(\b,\c)=(1,0)$, we have $|\pa_ya_{\s_1,\s_2}(y,\l)|\lesssim |W(y,\l)|$.

\vspace{2mm}
\noindent$(iii)$ We have
\begin{align*}
|a_{+,-}(y,\l)|\lesssim \jap{y}^{-2},\quad |a_{-,+}(y',\l)|\lesssim \jap{y'}^{-2}
\end{align*}
uniformly in $y\geq 0$, $y'\leq 0$ and $\l\in \R$.

\end{thm}

\begin{remark}
In particular, we have
\begin{align*}
&v_{\pm} (y, \lambda) = e^{\pm iy} + o(1) \quad \text{as $y\to \pm\infty$},
&\partial_{y} v_{\pm} (y, \lambda) = \pm ie^{\pm iy} + o(1) \quad \text{as $y\to \pm\infty$} 
\end{align*}
uniformly in $\lambda \in \R$.
\end{remark}

\begin{proof}
The Schr\"odinger equation \eqref{2509021621} is rewritten as
\begin{align}
\partial_{y} 
\begin{pmatrix}
v (y, \lambda) \\
v_y (y, \lambda)
\end{pmatrix}
=
\begin{pmatrix}
0 & 1 \\
W(y, \lambda) -1 & 0
\end{pmatrix}
\begin{pmatrix}
v(y, \lambda) \\
v_y (y, \lambda)
\end{pmatrix}
\quad \textrm{for} \quad (y, \lambda) \in \R^2.
\label{2510152237}
\end{align}
Setting
\begin{align*}
\textit{\textbf{w}} (y, \lambda) := R(y)
\begin{pmatrix}
v (y, \lambda) \\
v_y (y, \lambda)
\end{pmatrix}
\quad \textrm{and} \quad R(y) :=
\begin{pmatrix}
\cos y & -\sin y \\
\sin y & \cos y
\end{pmatrix},
\end{align*}
we have
\begin{align*}
\partial_{y} \textit{\textbf{w}} (y, \lambda) = \underbrace{\left( R(y)
\begin{pmatrix}
0 & 1 \\
W(y, \lambda) -1 & 0
\end{pmatrix}
+ \partial_{y} R(y) \right) R^{-1} (y)}_{=:B(y,\l)}
\begin{pmatrix}
v (y, \lambda) \\
v_y (y, \lambda)
\end{pmatrix}
=: B(y, \lambda) \textit{\textbf{w}} (y, \lambda)
\end{align*}
by \eqref{2510152237}. We calculate it as
\begin{align}\label{eq:BdefW2}
B(y, \lambda) = -W(y, \lambda)
\begin{pmatrix}
\sin y \cos y & \sin^{2} y \\
-\cos^2 y & -\sin y \cos y
\end{pmatrix}
.
\end{align}
By \eqref{2509021619}, the matrix valued function $B$ satisfies \eqref{eq:Bassump}. Hence we can construct the two solutions $\textit{\textbf{w}} _{\pm}(y,\l)={}^t(w_{1,\pm}(y,\l),w_{2,\l}(y,\l))$ as in Theorem \ref{2508311359}. The corresponding solutions $v_{\pm}$ are clearly uniformly bounded and satisfy
\begin{align}
v_{\pm} (y, \lambda) &= w_{1, \pm} (y, \lambda) \cos y + w_{2, \pm} (y, \lambda) \sin y \notag \\
&= \frac{w_{1, \pm} (y, \lambda) -i w_{2, \pm} (y, \lambda)}{2} e^{iy} +  \frac{w_{1, \pm} (y, \lambda) +i w_{2, \pm} (y, \lambda)}{2} e^{-iy}.\label{2510161027}
\end{align} 
Thus we define
\begin{align*}
a_{\pm, +} (y, \lambda) = \frac{w_{1, \pm} (y, \lambda) -i w_{2, \pm} (y, \lambda)}{2} \quad \textrm{and} \quad a_{\pm, -} (y, \lambda) = \frac{w_{1, \pm} (y, \lambda) +i w_{2, \pm} (y, \lambda)}{2}.
\end{align*}
The remaining properties of $a_{\s_1,\s_2}$ directly follow from Theorem \ref{2508311359}.
\end{proof}

\subsection{Construction of the Jost functions}\label{2510231114}
In this subsection we construct the Jost functions for Schr\"odinger equations with slowly decaying potentials. The point is to use the Liouville transform, Lemma \ref{2509041100}, to transform Schr\"odinger equations with slowly decaying potentials to those with short range potentials. In the following of this section, we assume that $V$ satisfies Assumption \ref{assump:V}.

\begin{lemma}\label{2509041100}
For $\lambda \in \R \setminus \{0\}$, the map
\begin{align}
\R \ni x \mapsto y(x, \lambda) := \int_{0}^{x} \sqrt{\lambda^2 -V(s)} ds \in \R \label{2510131941}
\end{align}
has a smooth inverse $x(y, \lambda)$, which also smoothly depends on $\lambda \in \R \setminus \{0\}$.
\end{lemma}

\begin{proof}
By  Assumption \ref{assump:V}, we have $\pm \int_{0}^{\pm \infty} \sqrt{\lambda^2 -V(s)} ds = \infty$ for any $\lambda \in \R$, which implies $y(x,\lambda) \to \pm\infty$ as $x \to \pm\infty$. Then the claim follows from the inverse function theorem with $\partial_x y(x,\lambda)=\sqrt{\lambda^2-V(x)}>0$ due to the assumption $V<0$. Note that the smooth dependence on $\lambda \in \R \setminus \{0\}$ is a consequence of $\partial_{y} x(y, \lambda) = 1/ \sqrt{\lambda^2 -V(x(y, \lambda))}$ and the standard regularity theory of ODE with a parameter.
\end{proof}
We state a symbolic estimate of $x(y, \lambda)$ which is used to show a symbolic estimate of solutions of Schr\"odinger equations.

\begin{lemma}\label{2509031823}
We have 
\begin{align}
|\partial^{\alpha} _{\lambda} x(y, \lambda)| \lesssim |\lambda|^{-\alpha} |x(y, \lambda)| \label{2509041027}
\end{align}
for any $\alpha \in \N_{0}$ uniformly in $y \in \R$ and $\lambda \in \R \setminus \{0\}$.
\end{lemma}

\begin{proof}
We show
\begin{align}
|\partial^{j} _{\lambda} x(y, \lambda)| \lesssim \frac{1}{(\lambda^2 -V(x(y, \lambda)))^{\frac{j}{2}}} |x(y, \lambda)| \label{2510101753}
\end{align}
for any $j \in \N_0$ by an induction argument. For $j=0$, (\ref{2510101753}) is trivial. Suppose (\ref{2510101753}) holds for $j=0, 1, \dots, p$. First note that, by differentiating (\ref{2510131941}) $p+1$ times, we have
\begin{align}
\sum_{q=0}^{p} c_{p, q}(\partial^{p-q} _{\lambda} \sqrt{\lambda^2 -V(x(y, \lambda))}) \cdot \partial^{q+1} _{\lambda} x(y, \lambda) + \partial^p _{\lambda} \left(\int_{0}^{x(y, \lambda)} \frac{\lambda}{\sqrt{\lambda^2 -V(s)}} ds \right) =0.\label{2510101850}
\end{align}
By the Fa\`a di Bruno formula\footnote{$\pa_x^n(f(g(x)))=\sum_{\textbf{m}}c_{n,\textbf{m}}f^{(m_1+\cdots+m_n)}(g(x))\prod_{j=1}^n(g^{(j)}(x))^{m_j}$, where $\textbf{m} = (m_1, \dots, m_n)$ satisfies $\sum_{j=1}^{n} jm_j =n$.}, for $n \le p$, 
\begin{align*}
\partial^n _{\lambda} \sqrt{\lambda^2 -V(x(y, \lambda))} = \sum c_{n, \textbf{m}} (\lambda^2 -V(x(y, \lambda)))^{\frac{1}{2} -(m_1 + \cdots +m_n)} \prod_{j=1}^{n} (\partial^{j} _{\lambda} (\lambda^2 -V(x(y, \lambda))))^{m_j},
\end{align*}
where $\textbf{m} = (m_1, \dots, m_n)$ satisfies $\sum_{j=1}^{n} jm_j =n$. For $n = 0, \dots, p$, we estimate
\begin{align*}
|\partial^n _{\lambda} V(x(y, \lambda))| &\lesssim \sum |V^{(m_1 + \cdots +m_n)} (x(y, \lambda))| \prod_{j=1}^{n} |\partial^j _{\lambda} x(y, \lambda)|^{m_j} \\
&\lesssim \sum \langle x(y, \lambda) \rangle^{-\mu-(m_1 + \cdots +m_n)} \prod_{j=1}^{n} (\lambda^2 -V(x(y, \lambda)))^{-\frac{jm_j}{2}} \cdot |x(y, \lambda)|^{m_1 + \cdots +m_n} \\
&\lesssim \langle x(y, \lambda) \rangle^{-\mu} (\lambda^2 -V(x(y, \lambda)))^{-\frac{n}{2}},
\end{align*}
where the sum is taken with respect to $\textbf{m} = (m_1, \dots, m_n)$ satisfying the same condition as above. Therefore we obtain
\begin{align*}
\left| \prod_{j=1}^{n} (\partial^{j} _{\lambda} (\lambda^2 -V(x(y, \lambda))))^{m_j} \right| &\lesssim \prod_{j=1}^{n} \sum_{k=0}^{m_j} |(\partial^j _{\lambda} \lambda^2)^{m_j -k} (\partial^j _{\lambda} V(x(y, \lambda)))^k | \\
&\lesssim \prod_{j=1}^{n} \sum_{k=0}^{m_j} (\lambda^2 -V(x(y, \lambda)))^{\frac{(2-j)(m_j -k)}{2}} (\langle x(y, \lambda) \rangle^{-\mu} (\lambda^2 -V(x(y, \lambda)))^{-\frac{j}{2}})^k \\
&\lesssim \prod_{j=1}^{n} \sum_{k=0}^{m_j} (\lambda^2 -V(x(y, \lambda)))^{\frac{2m_j -2k-jm_j}{2}} |V(x(y, \lambda))|^k \\
&\lesssim \prod_{j=1}^{n} \sum_{k=0}^{m_j} (\lambda^2 -V(x(y, \lambda)))^{\frac{2m_j -jm_j}{2}} \lesssim (\lambda^2 -V(x(y, \lambda)))^{(m_1 + \cdots +m_n)-\frac{n}{2}},
\end{align*}
which yields 
\begin{align}
|\partial^n _{\lambda} \sqrt{\lambda^2 -V(x(y, \lambda))}| \lesssim (\lambda^2 -V(x(y, \lambda)))^{\frac{1-n}{2}} \label{2510131956}
\end{align}
for $n \le p$. As a result (\ref{2510101850}) gives
\begin{align}
|\sqrt{\lambda^2 -V(x(y, \lambda))} \partial^{p+1} _{\lambda} x(y, \lambda)| &\lesssim \sum_{q=0}^{p-1} (\lambda^2 -V(x(y, \lambda)))^{\frac{1-(p-q)}{2}} \cdot \frac{|x(y, \lambda)|}{(\lambda^2 -V(x(y, \lambda)))^{\frac{q+1}{2}}} \notag \\
& \qquad + \left| \partial^p _{\lambda} \left(\int_{0}^{x(y, \lambda)} \frac{\lambda}{\sqrt{\lambda^2 -V(s)}} ds \right) \right| \notag \\
& \lesssim \frac{|x(y, \lambda)|}{(\lambda^2 -V(x(y, \lambda)))^{\frac{p}{2}}} + \left| \partial^p _{\lambda} \left(\int_{0}^{x(y, \lambda)} \frac{\lambda}{\sqrt{\lambda^2 -V(s)}} ds \right) \right| \label{2510131916}
\end{align}
and the proof is reduced to showing that the last term in (\ref{2510131916}) is bounded above by the first term. By differentiating iteratively, it suffices to show the following for $1 \le r \le p$.
\begin{align}
&|\partial^{p-r} _{\lambda} (\partial^{r} _{\lambda} \sqrt{\lambda^2 -V(s)} |_{s=x(y, \lambda)} \cdot \partial_{\lambda} x(y, \lambda))| \lesssim \frac{|x(y, \lambda)|}{(\lambda^2 -V(x(y, \lambda)))^{\frac{p}{2}}}, \label{2610131924} \\
&\left| \int_{0}^{x(y, \lambda)}  \partial^{p+1} _{\lambda} \sqrt{\lambda^2 -V(s)} ds \right| \lesssim \frac{|x(y, \lambda)|}{(\lambda^2 -V(x(y, \lambda)))^{\frac{p}{2}}}. \label{2510131927}
\end{align}
Concerning (\ref{2610131924}), we use $\partial^{r} _{\lambda} \sqrt{\lambda^2 -V(s)} |_{s=x(y, \lambda)} = \sum a_k (\lambda^2 -V(x(y, \lambda)))^{-\frac{2k+1}{2}} \lambda^{2k+2-r}$, where the sum is taken with respect to $k \in \N_{0}$ such that $2k+2-r \ge 0$ and $a_k \in \C$ are zero except for finitely many $k \in \N_0$. This is proved by an induction argument with respect to $r \in \N$. By a similar calculation as in the proof of (\ref{2510131956}), we have $|\partial^{q} _{\lambda} (\lambda^2 -V(x(y, \lambda)))^{-\frac{2k+1}{2}}| \lesssim  (\lambda^2 -V(x(y, \lambda)))^{-\frac{2k+1}{2} -\frac{q}{2}}$ for any $k \in \N_0$ and $q =0, \dots, p-r$. Therefore we have
\begin{align*}
&|\partial^{p-r} _{\lambda} ((\lambda^2 -V(x(y, \lambda)))^{-\frac{2k+1}{2}} \lambda^{2k+2-r} \partial_{\lambda} x(y, \lambda))| \\
&\quad \lesssim \sum_{\alpha_1 + \alpha_2 + \alpha_3 =p-r} (\lambda^2 -V(x(y, \lambda)))^{-\frac{2k+1}{2} -\frac{\alpha_1}{2}} \lambda^{2k+2-r-\alpha_2} |\partial^{1+\alpha_3} _{\lambda} x(y, \lambda)| \\
&\quad \lesssim \sum_{\alpha_1 + \alpha_2 + \alpha_3 =p-r} (\lambda^2 -V(x(y, \lambda)))^{-\frac{2k+1}{2} -\frac{\alpha_1}{2}} (\lambda^2 -V(x(y, \lambda)))^{k+1-\frac{r+\alpha_2}{2}} (\lambda^2 -V(x(y, \lambda)))^{-\frac{1+\alpha_3}{2}} |x(y, \lambda)| \\
&\quad \lesssim \frac{|x(y, \lambda)|}{(\lambda^2 -V(x(y, \lambda)))^{\frac{p}{2}}},
\end{align*}
which implies (\ref{2610131924}). Concerning (\ref{2510131927}), we calculate similarly as follows:
\begin{align*}
\left| \int_{0}^{x(y, \lambda)}  \partial^{p+1} _{\lambda} \sqrt{\lambda^2 -V(s)} ds \right| &\lesssim \sum_{k} \left| \int_{0}^{x(y, \lambda)} (\lambda^2 -V(s))^{-\frac{2k+1}{2}} \lambda^{2k+2-(p+1)} ds \right| \\
& \lesssim \left| \int_{0}^{x(y, \lambda)} (\lambda^2 -V(s))^{-\frac{p}{2}} ds\right| \lesssim \left| \int_{0}^{x(y, \lambda)} (\lambda^2 + \langle s \rangle^{-\mu})^{-\frac{p}{2}} ds\right| \\
& \lesssim \frac{|x(y, \lambda)|}{(\lambda^2 + \langle x(y, \lambda) \rangle^{-\mu})^{\frac{p}{2}}} \lesssim \frac{|x(y, \lambda)|}{(\lambda^2 -V(x(y, \lambda)))^{\frac{p}{2}}}.
\end{align*}
The proof is completed.
\end{proof}

Now we apply the Liouville transform to construct oscillating solutions of (\ref{2509021638}) which are counterparts of the ordinary Jost functions for Schr\"odinger equations with short range potentials.

\begin{thm}\label{2509021637}
There exist solutions $u_{\pm}\in C(\R_x\times \R_{\l})\cap C^{\infty}(\R_x\times \R_{\l}\setminus \{0\})$ of
\begin{align}
-\partial^2 _{x} u_{\pm} (x, \lambda) +V(x)u_{\pm} (x, \lambda) = \lambda^2 u_{\pm} (x, \lambda). \label{2509021638}
\end{align}

\vspace{1mm}
\noindent$(i)$ There exist $\tilde{a}_{\s_1,\s_2}\in C(\R_x\times \R_{\l})\cap C^{\infty}(\R_x\times \R_{\l}\setminus \{0\})$ $(\s_1,\s_2\in \{\pm\})$ such that
\begin{align}\label{2509031817}
u_{\pm} (x, \lambda) =\frac{1}{(\l^2-V(x))^{\frac{1}{4}}} \left(\tilde{a}_{\pm, +} (x, \lambda) e^{iy(x,\l)} + \tilde{a}_{\pm, -} (x, \lambda) e^{-iy(x,\l)}\right)
\end{align}
and the limits $\displaystyle \tilde{a}_{\s_1,\s_2}(\pm\infty,\l):=\lim_{x\to \pm\infty}\tilde{a}_{\s_1,\s_2}(x,\l)$ exist uniformly in $\l\in \R$. Furthermore,
\begin{align}\label{eq:asymptildea}
\tilde{a}_{+,+}(\infty,\l)=\tilde{a}_{-,-}(-\infty,\l)=1,\quad \tilde{a}_{+,-}(\infty,\l)=\tilde{a}_{-,+}(-\infty,\l)=0
\end{align}
and in general, $\displaystyle \tilde{a}_{\s_1,\s_2}(\pm\infty,\l)=a_{\s_1,\s_2}(\pm\infty,\l)$, where $a_{\s_1,\s_2}$ is as in Theorem \ref{2509021616} for the choice of $W(y,\l)$ in \eqref{2510011536}.

\vspace{1mm}
\noindent$(ii)$ For $\a=0,1,2$,
\begin{align}\label{2509031822}
|\partial^{\a} _{\lambda} \tilde{a}_{\s_1,\s_2} (x, \lambda)| \lesssim |\lambda|^{-\a},\quad |\pa_x\tilde{a}_{\s_1,\s_2}(x,\l)|\lesssim \jap{x}^{\frac{\m}{2}-2}
\end{align}
uniformly in $x\in \R$ and $\l\in \R\setminus \{0\}$.

\vspace{2mm}
\noindent$(iii)$ 
We have
\begin{align}\label{261271003}
|\tilde{a}_{+,-}(x,\l)|\lesssim \min(\jap{x}^{\mu -2},|\l|^{-2}|x|^{-2}),\quad |\tilde{a}_{-, +} (x', \lambda)| \lesssim \min(\jap{x'}^{\mu -2},|\l|^{-2}|x'|^{-2})
\end{align}
uniformly in $\lambda \in \R\setminus \{0\}$, $x>0$ and $x' <0$.

\noindent$(iv)$ Let $\mathrm{Wr}(\lambda) := u_{+} (x, \lambda) \partial_{x} u_{-} (x, \lambda) - \partial_{x} u_{+} (x, \lambda) u_{-} (x, \lambda)$ be the Wronskian. Then $\mathrm{Wr} (\lambda)$ satisfies
\begin{align*}
|\mathrm{Wr}(\lambda)| \sim 1, \quad \left| \left(\frac{d}{d\lambda} \right)^{\alpha} \mathrm{Wr}(\lambda) \right| \lesssim |\lambda|^{-\alpha} \quad \text{for any $\alpha \in \N$}
\end{align*}
uniformly in $\lambda \in \R \setminus \{0\}$.

\end{thm}

\begin{remark}
The asymptotic properties \eqref{2509031817}, \eqref{eq:asymptildea} and the second estimate in \eqref{2509031822} imply
\begin{align}
&u_{\pm} (x, \lambda) = \frac{1}{(\lambda^2 - V(x))^{\frac{1}{4}}} \left(e^{\pm iy (x, \lambda)} + o(1) \right) \quad \text{as $x \to \pm\infty$}, \label{2509031759} \\
&\partial_{x} u_{\pm} (x, \lambda) = \pm i (\lambda^2 - V(x))^{\frac{1}{4}} \left(e^{\pm iy (x, \lambda)} + o(1) \right) \quad \text{as $x \to \pm\infty$} \label{2509031801}
\end{align}
uniformly in $\lambda \in \R \setminus \{0\}$. From this and the standard formula of the scattering matrix $S(\l)$ (\cite[p.213]{Y}), we have
\begin{align*}
S(\l)=\begin{pmatrix}
a_{+,+}(-\infty,\l)^{-1}&a_{-,-}(\infty,\l)^{-1}a_{-,+}(\infty,\l) \\
a_{+,+}(-\infty,\l)^{-1}a_{+,-}(-\infty,\l)&a_{-,-}(\infty,\l)^{-1}
\end{pmatrix}
\quad \text{for $\l>0$}.
\end{align*}
In particular, as $\lambda \to +0$, $S(\lambda)$ converges to a unitary matrix. Such a phenomenon has been observed in higher dimensions for homogeneous potentials (see \cite{DS2, Fra}) but our result implies the symmetry of $V$ is not essential in one dimension.

\end{remark}

\begin{remark}
This result is closely related to the construction of distorted Fourier transform \cite{HS}, which is also called the wave matrix or the generalized Fourier transform in scattering theory (see \cite[\S 6.2]{DS} and \cite[p.214]{Y}).
\end{remark}

We want to reduce the proof of Theorem \ref{2509021637} to Theorem \ref{2509021616}. To do this, we define
\begin{align}
&U(x,\lambda)=\frac{V''(x)}{4(\lambda^2-V(x))^2}-\frac{5(V'(x))^2}{16(\lambda^2-V(x))^{3}} ,\quad W(y,\lambda):=U(x(y,\l),\l), \label{2510011536}
\end{align}
where $x(y,\l)$ is the inverse function of the Liouville transform $y(x,\lambda)=\int_0^x\sqrt{\lambda^2-V(s)}ds$ introduced in Lemma \ref{2509041100}.

\begin{lemma}\label{lem:Wsymbol}
$W(y,\l)$ defined in \eqref{2510011536} satisfies \eqref{2509021619}. 
\end{lemma}

\begin{proof}
To prove that $W$ satisfies \eqref{2509021619}, we show
\begin{align}
|\partial^{\alpha} _{x} \partial^{\beta} _{\lambda} U(x, \lambda)| \lesssim \langle x \rangle^{-\alpha} (\lambda^2-V(x))^{-\frac{\b}{2}} \left\{ \frac{\langle x \rangle^{-2-\mu}}{(\lambda^2-V(x))^2} + \frac{\langle x \rangle^{-2-2\mu}}{(\lambda^2-V(x))^{3}}\right\} \label{2510171716}
\end{align}
for any $\alpha, \beta \in \N_0$. By an inductive argument, for any $\beta \in \N_0$,
\begin{align*}
\partial^{\beta} _{\lambda} U(x, \lambda) = V'' (x) \sum_{j-2k = -\beta} c_{j, k, \beta} \lambda^{j} (\lambda^2-V(x))^{-k-2} -(V'(x))^2 \sum_{j-2k=-\beta} C_{j, k, \beta} \lambda^{j} (\lambda^2-V(x))^{-k-3}
\end{align*}
holds, where $c_{j, k, \beta}, C_{j, k, \beta} \in \R$ and the sum is taken with respect to finite pairs of $(j, k) \in \N^2 _0$ satisfying $j-2k=-\beta$. Therefore
\begin{align*}
\partial^{\alpha} _{x} \partial^{\beta} _{\lambda} U(x, \lambda) &= \sum_{j-2k = -\beta} c_{j, k, \beta, \alpha, \gamma} \lambda^{j} (\partial^{\alpha -\gamma} _{x} V''(x)) \partial^{\gamma} _x (\lambda^2-V(x))^{-k-2} \\
&\quad + \sum_{j-2k=-\beta} C_{j, k, \beta, \alpha, \gamma} \lambda^{j} (\partial^{\alpha -\gamma} _{x} (V'(x))^2) \partial^{\gamma} _x (\lambda^2-V(x))^{-k-3},
\end{align*}
where $c_{j, k, \beta, \alpha, \gamma}, C_{j, k, \beta, \alpha, \gamma} \in \R$, $(j, k) \in \N^2 _0$ satisfies the above condition and $0 \le \gamma \le \alpha$. To estimate the right hand side, by an inductive argument, we have
\begin{align*}
&\partial^{\gamma} _x (\lambda^2-V(x))^{-k-2} = \sum_{\alpha_1 + \dots +\alpha_l =\gamma} C_{\alpha_1, \dots, \alpha_l} V^{(\alpha_1)} \dots V^{(\alpha_l)} (\lambda^2-V(x))^{-k-2-l}, \\
&|\partial^{\gamma} _x (\lambda^2-V(x))^{-k-2}| \lesssim \sum_{\alpha_1 + \dots +\alpha_l =\gamma} (\lambda^2-V(x))^{-k-2-l} \langle x \rangle^{-(\alpha_1 + \cdots + \alpha_l)-\mu l} \lesssim \langle x \rangle^{-\gamma} (\lambda^2-V(x))^{-k-2},
\end{align*}
where the sum is finite and $C_{\alpha_1, \dots, \alpha_l} \in \R$. Thus we obtain
\begin{align*}
|\lambda^{j} (\partial^{\alpha -\gamma} _{x} V''(x)) \partial^{\gamma} _x (\lambda^2-V(x))^{-k-2}| &\lesssim |\lambda|^j \langle x \rangle^{-(\alpha-\gamma)-2-\mu} \langle x \rangle^{-\gamma} (\lambda^2-V(x))^{-k-2} \\
&\lesssim \langle x \rangle^{-\alpha} \frac{\langle x \rangle^{-2-\mu}}{(\lambda^2-V(x))^{\frac{\b}{2}+2}}
\end{align*}
if $j-2k=-\beta$ holds. For such $(j, k)$, a similar calculation yields
\begin{align*}
|\lambda^{j} (\partial^{\alpha -\gamma} _{x} (V'(x))^2) \partial^{\gamma} _x (\lambda^2-V(x))^{-k-3}| \lesssim   \langle x \rangle^{-\alpha} \frac{\langle x \rangle^{-2-2\mu}}{(\lambda^2-V(x))^{\frac{\b}{2}+3}}
\end{align*}
and \eqref{2510171716} follows from these observations. 

Now we estimate $\partial^{j} _{\lambda} W(y, \lambda)$. By an inductive argument, we have
\begin{align*}
\partial^{j} _{\lambda} W(y, \lambda) = \sum_{\beta + \alpha_1 +\cdots + \alpha_l =j} c_{\alpha, \beta} (\partial^{l} _x \partial^{\beta} _{\lambda} U)(x(y, \lambda), \lambda)\left(\prod_{k=1}^{l} \partial^{\alpha_k} _{\lambda} x(y, \lambda)\right)
\end{align*}
for any $j \in \N_0$, where the sum is finite and $c_{\alpha, \beta} \in \R$. By \eqref{2510101753} and \eqref{2510171716}, we obtain
\begin{align}
&\left| (\partial^{l} _x \partial^{\beta} _{\lambda} U)(x(y, \lambda), \lambda)\left(\prod_{k=1}^{l} \partial^{\alpha_k} _{\lambda} x(y, \lambda)\right)\right| \notag \\
&\lesssim \langle x(y, \lambda) \rangle^{-l} \left\{ \frac{\langle x(y, \lambda) \rangle^{-2-\mu}}{(\lambda^2-V(x(y, \lambda)))^{\frac{\b}{2}+2}} + \frac{\langle x(y, \lambda) \rangle^{-2-2\mu}}{(\lambda^2-V(x(y, \lambda)))^{\frac{\b}{2}+3}}\right\} \left(\prod_{k=1}^{l}\frac{|x(y,\l)|}{(\l^2-V(x(y,\l)))^{\frac{\a_k}{2}}}\right) \notag \\
&\lesssim (\l^2-V(x(y,\l)))^{-\frac{j}{2}} \left\{ \frac{\langle x(y, \lambda) \rangle^{-2-\mu}}{(\lambda^2-V(x(y, \lambda)))^2} + \frac{\jap{x(y, \lambda)}^{-2-2\mu}}{(\lambda^2-V(x(y, \lambda)))^{3}}\right\} \label{2510182144}
\end{align}
if $\beta + \alpha_1 +\cdots + \alpha_l =j$. Since $(\l^2-V(x))^{-1}\lesssim \min(|\l|^{-1},\jap{x}^{\frac{\m}{2}})$, we have
\begin{align*}
 \frac{\langle x(y, \lambda) \rangle^{-2-\mu}}{(\lambda^2-V(x(y, \lambda)))^2} + \frac{\jap{x(y, \lambda)}^{-2-2\mu}}{(\lambda^2-V(x(y, \lambda)))^{3}}
\lesssim&\jap{x}^{-\m-2}\min(|\l|^{-1},\jap{x}^{\frac{\m}{2}})^4+\jap{x}^{-2\m-2}\min(|\l|^{-1},\jap{x}^{\frac{\m}{2}})^6\\
\lesssim& \jap{x}^{-2}\min(|\l|^{-1},\jap{x}^{\frac{\m}{2}})^2.
\end{align*}
Moroever,
\begin{align*}
|y(x,\l)|\leq \left| \int_{0}^{x} \sqrt{\lambda^2 -V(s)} ds\right|\lesssim\int_0^x(|\l|+|V(s)|^{\frac{1}{2}})ds\lesssim \l\jap{x}+\jap{x}^{1-\frac{\m}{2}}
\lesssim \jap{x}\left(\min(|\l|^{-1},\jap{x}^{\frac{\m}{2}})\right)^{-1}.
\end{align*}
Combining them, we have $(\text{RHS of \eqref{2510182144}})\lesssim  (\l^2-V(x(y,\l)))^{-\frac{j}{2}}\jap{y}^{-2}$ and hence
\begin{align*}
|\pa_{\l}^jW(y,\l)|\lesssim (\l^2-V(x(y,\l)))^{-\frac{j}{2}}\jap{y}^{-2}\lesssim |\l|^{-j}\jap{y}^{-2}.
\end{align*}

Concerning $\partial_y W(y, \lambda)$, by (\ref{2510171716}) and $\partial_y x = (\l^2-V(x(y,\l)))^{-\frac{1}{2}}$, we have
\begin{align*}
|\partial_y W(y, \lambda)|  =& \left|(\partial_x U)(x(y, \lambda), \lambda) \frac{\partial x}{\partial y} (y, \lambda) \right| \lesssim \left\{ \frac{\langle x(y, \lambda) \rangle^{-3-\mu}}{(\lambda^2-V(x(y, \lambda)))^{\frac{5}{2}}} + \frac{\langle x(y, \lambda) \rangle^{-3-2\mu}}{(\lambda^2-V(x(y, \lambda)))^{\frac{7}{2}}}\right\}\\
\lesssim& \jap{x(y, \lambda)}^{-3-\mu}\min(|\l|^{-1},\jap{x}^{\frac{\m}{2}})^5+\jap{x(y, \lambda)}^{-3-2\mu}\min(|\l|^{-1},\jap{x}^{\frac{\m}{2}})^7\lesssim\jap{y}^{-3}
\end{align*}
similarly as above. This completes the proof.
\end{proof}

Now we turn to prove Theorem \ref{2509021637}.

\begin{proof}[Proof of Theorem \ref{2509021637}]
We use the Liouville transform $y(x,\lambda)=\int_0^x\sqrt{\lambda^2-V(s)}ds$ introduced in Lemma \ref{2509041100} and define $W(y,\l)$ as \eqref{2510011536} and
\begin{align}\label{eq:vdefviau}
v(y, \lambda)=\mu (x)^{-1}u(x)|_{x=x(y,\lambda)},\quad  \mu (x):=(\lambda^2-V(x))^{-\frac{1}{4}},
\end{align}
Then, the equation $-\partial_x^2u(x,\lambda)+V(x)u(x,\lambda)=\lambda^2u(x,\lambda)$ is equivalent to
\begin{align*}
(-\partial_y^2+W(y,\lambda))v(y,\lambda)=v(y,\lambda).
\end{align*}
In fact, since $\frac{dy}{dx}=\sqrt{\lambda^2-V(x)}$ and $\frac{dx}{dy}=(\lambda^2-V(x))^{-\frac{1}{2}}=\mu(x)^2$, we have
\begin{align*}
\frac{d}{dy}v=&\frac{dx}{dy}\frac{d}{dx}\left(\mu(x)^{-1}u(x) \right)|_{x=x(y,\lambda)}=(-\mu'(x)u(x)+\mu(x)u'(x))|_{x=x(y,\lambda)},\\
\frac{d^2}{dy^2}v=&\frac{dx}{dy}\frac{d}{dx}\left(-\mu'(x)u(x)+\mu(x)u'(x)) \right)|_{x=x(y,\lambda)}
\underbrace{=}_{\mathclap{u''=-(\lambda^2-V)u=-\mu^{-4}u}}(-\mu(x)^2\mu''(x)u(x)-\mu(x)^{-1}u(x))|_{x=x(y,\lambda)}\\
\underbrace{=}_{u=\mu v}&(-\mu(x)^3\mu''(x)-1)|_{x=x(y,\lambda)}v(y),
\end{align*}
and $-\mu(x)^3\mu''(x)|_{x=x(y,\lambda)}=W(y,\lambda)$.
These calculation imply $(-\partial_y^2+W(y, \lambda))v(y, \lambda)=v(y, \lambda)$. 
By virtue of Lemma \ref{lem:Wsymbol}, we can apply Theorem \ref{2509021616} and take $a_{\s_1,\s_2}$ $(\s_1,\s_2\in \{\pm\})$ as in Theorem \ref{2509021616}. Define
\begin{align*}
\tilde{a}_{\s_1,\s_2}(x,\l):=a_{\s_1,\s_1}(y(x,\l),\l).
\end{align*}
Then, Theorem \ref{2509021616} $(i)$ proves the part $(i)$.

\vspace{1mm}
\noindent$(ii)$ It is easy to see that
\begin{align*}
|\partial^{\alpha} _{\lambda} y(x, \lambda)| \lesssim \left| \int_{0}^{x} \partial^{\alpha} _{\lambda} \sqrt{\lambda^2 -V(s)} ds\right| \lesssim |\lambda|^{1-\alpha} |x|
\end{align*}
for any $\alpha \in \N$. Therefore
\begin{align*}
|\partial^{\beta} _{\lambda} \tilde{a}_{\s_1, \s_2} (x, \lambda)| &\lesssim \sum_{\alpha_1 + \cdots +\alpha_{\delta} +\gamma =\beta} \left|(\partial^{\gamma} _{\lambda} \partial^{\delta} _y a_{\s_1, \s_2} )(y(x, \lambda), \lambda)\prod_{j=1}^{\delta} \partial^{\alpha_j} _{\lambda} y(x, \lambda) \right| \\
&\lesssim\sum_{\alpha_1 + \cdots +\alpha_{\delta} +\gamma =\beta} |\lambda|^{-\gamma} \langle y(x, \lambda) \rangle^{-\delta}\prod_{j=1}^{\delta} |\lambda|^{1-\alpha_j} |x|\\
&\lesssim \sum_{\alpha_1 + \cdots +\alpha_{\delta} +\gamma =\beta} |\lambda|^{-\gamma} (|\lambda||x|)^{-\delta}\prod_{j=1}^{\delta} |\lambda|^{1-\alpha_j} |x|  \lesssim |\lambda|^{-\beta}
\end{align*}
holds for any $\beta \in \N$, where we have used Theorem \ref{2509021616} $(ii)$ and $|y(x, \lambda)| \ge \left| \int_{0}^{x} |\lambda| ds\right|=|\lambda||x|$.

To prove the bound of $\pa_x\tilde{a}_{\s_1,\s_2}$, we use the second statement of Theorem \ref{2509021616}: $|\pa_ya_{\s_1,\s_2}(y,\l)|\lesssim |W(y,\l)|$. By Lebniz's rule and $\pa_xy(x,\l)=\sqrt{\l^2-V(x)}$, we obtain
\begin{align*}
|\pa_x\tilde{a}_{\s_1,\s_2}(x,\l)|=&|\sqrt{\l^2-V(x)}(\pa_ya_{\s_1,\s_2})(y(x,\l),\l)|\lesssim |\sqrt{\l^2-V(x)}W(y(x,\l))|\\
\lesssim&(\l^2-V(x))^{-\frac{3}{2}}|V''(x)|+(\l^2-V(x))^{-\frac{5}{2}}|V'(x)|^2\lesssim\jap{x}^{\frac{\m}{2}-2},
\end{align*}
where we recall $W(y(x,\l))=U(x)$ from \eqref{2510011536}.

\vspace{1mm}
\noindent$(iii)$ By Theorem \ref{2509021616} $(iii)$, we have $|a_{+,-}(y,\lambda)|\lesssim \jap{y}^{-2}$ for $y\geq 0$.
Since $y(x,\l)\gtrsim \max(|x|\jap{x}^{-\frac{\m}{2}},\l|x|)$, we obtain
\begin{align*}
|\tilde{a}_{+,-}(x,\l)|=|a_{+, -} (y(x, \lambda), \lambda)| \lesssim \jap{y(x, \lambda)}^{-2} \lesssim \min(\jap{x}^{\mu -2},|\l|^{-2}|x|^{-2})
\end{align*}
for $x>0$. The estimate for $|\tilde{a}_{-,+}(x,\l)|$ is similarly proved.

\vspace{1mm}
\noindent$(iv)$ For solutions $f(\cdot, \lambda), g(\cdot, \lambda) \in C^1 (\R)$ of (\ref{2509021638}), we define the Wronskian by $w\{f, g\} (\lambda) =f(x, \lambda) \partial_x g(x, \lambda) - \partial_x f(x, \lambda)g(x, \lambda)$, which is constant with respect to $x \in \R$. In particular, $\mathrm{Wr}(\lambda)=w\{u_+, u_-\} (\lambda)$ is also constant.

 By \eqref{2509031759} and \eqref{2509031801}, we have
\begin{align*}
w \{u_+, \overline{u_+}\} (\lambda) &= u_+ (x, \lambda) \partial_{x}  \overline{u_+} (x, \lambda) -\partial_x  u_+ (x, \lambda)  \overline{u_+} (x, \lambda) \\
&= (e^{iy(x, \lambda)} + o(1)) \cdot (-i)(e^{-iy(x, \lambda)} + o(1)) -i(e^{iy(x, \lambda)} +o(1)) \cdot (e^{-iy(x, \lambda)} +o(1)) \\
&=-2i + o(1)
\end{align*}
as $x \to \infty$. Hence $w \{u_+, \overline{u_+}\} (\lambda) =-2i \neq 0$ and $u_+,  \overline{u_+}$ are linearly independent solutions of $(-\partial^2 _x +V(x)-\lambda^2)u=0$. Similarly we obtain $w \{u_-, \overline{u_-}\} (\lambda) = 2i$ and $u_-, \overline{u_-}$ are linearly independent. Therefore there exist $a(\lambda), b(\lambda), c(\lambda), d(\lambda) \in \C$ such that
\begin{align}\label{eq:Jostrelation}
u_+ (x, \lambda)= a(\lambda) u_- (x, \lambda) + b(\lambda) \overline{u_-} (x, \lambda), \quad  u_- (x, \lambda)= c(\lambda) u_+ (x, \lambda)+ d(\lambda) \overline{u_+} (x, \lambda),
\end{align}
which imply
\begin{align*}
u_+(x,\l)=(a(\l)c(\l)+b(\l)\overline{d(\l)})u_+(x,\l)+(a(\l)d(\l)+b(\l)\overline{c(\l)}))\overline{u_+}(x,\l)
\end{align*}
By computing $w\{u_+, u_-\} (\lambda)$ and $w\{u_+, \overline{u_-}\} (\lambda)$, we obtain
\begin{align}\label{eq:Jostcoeff}
a(\lambda) =-\frac{i}{2} w\{u_+, \overline{u_-}\} (\lambda),\,\, b(\lambda)= \frac{i}{2}\mathrm{Wr}(\lambda),\,\, c(\lambda)=-\overline{a(\l)},\,\, d(\lambda)=b(\l)
\end{align}
where we have used $w \{f, g\} = -w\{g, f\}$ and $w\{u_+, u_-\}=\mathrm{Wr}(\lambda)$. In particular,
\begin{align}\label{eq:Jostcoeffrel}
a(\l)c(\l)+b(\l)\overline{d(\l)}=|a(\l)|^2-|b(\l)|^2,\quad a(\l)d(\l)+b(\l)\overline{c(\l)}=0.
\end{align}
Since $u_+ \neq 0$, we obtain $|a(\l)|^2-|b(\l)|^2=1$ and $ |\mathrm{Wr}(\lambda)|^2 - |w\{u_+, \overline{u_-}\} (\lambda)|^2 =4$. Especially, $|\mathrm{Wr}(\lambda)| \ge 2$. 

Finally, we prove the symbolic estimates of $\mathrm{Wr}(\lambda)$. To see this, it suffices to show $\mathrm{Wr}(\lambda) =-2i a_{-, -} (\infty, \lambda)$, which immediately yields them by the second inequality of Theorem \ref{2509021616} $(ii)$.
Using $\partial_x y(x, \lambda) = \sqrt{\lambda^2 -V(x)}$, we obtain
\begin{align*}
\partial_{x} u_{\pm} (x, \lambda) = i(\lambda^2 -V(x))^{\frac{1}{4}} (\tilde{a}_{\pm, +}(x,\l)e^{iy(x, \lambda)} -\tilde{a}_{\pm, -}(x,\l)e^{-iy(x, \lambda)} + o(1))
\end{align*}
as $x \to \infty$. Therefore
\begin{align*}
&u_+ (x, \lambda) \partial_{x} u_{-} (x, \lambda) =i\tilde{a}_{+,+} (x, \lambda)e^{iy(x, \lambda)} (\tilde{a}_{-, +} (x, \lambda)e^{iy(x, \lambda)} -\tilde{a}_{-, -} (x, \lambda)e^{-iy(x, \lambda)}) + o(1)\\
&u_- (x, \lambda)\partial_{x} u_{+} (x, \lambda) = (\tilde{a}_{-, +} (x, \lambda)e^{iy(x, \lambda)} +\tilde{a}_{-, -} (x, \lambda)e^{-iy(x, \lambda)})(i\tilde{a}_{+,+} (x, \lambda)e^{iy(x, \lambda)}) + o(1)
\end{align*}
as $x \to \infty$, where we have used $a_{+, -} (y, \lambda) \to 0$ as $y \to \infty$. Thus we obtain $\mathrm{Wr}(\lambda) =-2i \tilde{a}_{+, +} (\infty, \lambda) \tilde{a}_{-, -} (\infty, \lambda)$. As is proved in $(i)$, $ \tilde{a}_{+, +} (\infty, \lambda)=1$ and $\tilde{a}_{-, -} (\infty, \lambda)=a_{-,-}(\infty,\l)$. Thus the claim follows.
\end{proof}

\begin{remark}\label{2510182152}
In the proof of Theorem \ref{2509021637} $(ii)$, the crucial step is the third estimate in Theorem \ref{2509021616} $(ii)$. However it may not hold if $\a \ge 3$. This is because the matrix factor 
\begin{align*}
\begin{pmatrix}
\sin y \cos y & \sin^2 y \\
-\cos^2 y & -\sin y \cos y
\end{pmatrix}
\end{align*}
appearing in the definition of $B(y, \lambda)$ does not have symbolic decay when it is differentiated. There remains the possibility to obtain \eqref{2509031822} using some cancellation but we do not pursue this here because Theorem \ref{2509021637} $(ii)$ is sufficient to show the dispersive estimates.
\end{remark}

\subsection{WKB expressions of the outgoing resolvent and the spectral projection}\label{2510231110}
In this subsection we give a WKB expression of the spectral projection which is used to write the Schr\"odinger propagator as an oscillatory integral operator. First we prove that the outgoing resolvent of the Schr\"odinger operator $P=-\partial^2 _x +V$ is written using the Jost functions. This is an analogue of \cite[Chapters 4 and 5]{Y}, where the Schr\"odinger operator with short range potentials is considered. Note that the existence of the outgoing resolvent or the limiting absorption principle is already known \cite[Chapter 11]{Y}.
\begin{lemma}\label{2508241545}
The integral kernel of the outgoing resolvent $R(\lambda^2 +i0) (x, x')$ is given by
\begin{align}
R(\lambda^2 +i0) (x, x') = \left\{\begin{aligned}
&\frac{1}{\mathrm{Wr}(\lambda)}u_-(x,\lambda)u_+(x',\lambda)&&x<x', \\
&\frac{1}{\mathrm{Wr}(\lambda)}u_+(x,\lambda)u_-(x',\lambda)&&x>x' \label{2512101147}
\end{aligned}\right.
\end{align}
for any $\lambda >0$, where $u_{\pm}$ and $\mathrm{Wr}(\l)$ are given in Theorem \ref{2509021637}.
\end{lemma}

\begin{proof}
We fix $\lambda \in \R_{>0}$ and define an integral operator $T$ by
\begin{align*}
Tf(x):= \frac{1}{\mathrm{Wr}(\lambda)}u_- (x, \lambda) \int_{x}^{\infty} u_+ (y, \lambda) f(y)dy + \frac{1}{\mathrm{Wr}(\lambda)}u_+ (x, \lambda)\int_{-\infty}^{x} u_- (y, \lambda) f(y)dy
\end{align*}
for $f \in C^{\infty} _c (\R)$. The integrals on the right hand side is well-defined by Theorem \ref{2509021637}. Then a direct computation shows
\begin{align*}
&\partial_x Tf(x) = \frac{1}{\mathrm{Wr}(\lambda)}\partial_x u_- (x, \lambda)\int_{x}^{\infty} u_+ (y, \lambda) f(y)dy + \frac{1}{\mathrm{Wr}(\lambda)}\partial_x u_+ (x, \lambda)\int_{-\infty}^{x} u_- (y, \lambda)f(y)dy, \\
&\partial^2 _x Tf(x) = \frac{\partial^2 _x u_- (x, \lambda)}{\mathrm{Wr}(\lambda)} \int_{x}^{\infty} u_+ (y, \lambda) f(y)dy +\frac{\partial^2 _x u_+ (x, \lambda)}{\mathrm{Wr}(\lambda)} \int_{-\infty}^{x} u_- (y, \lambda)f(y)dy -f(x),
\end{align*}
where we have used $\mathrm{Wr}(\lambda) = u_{+} (x, \lambda) \partial_{x} u_{-} (x, \lambda) - \partial_{x} u_{x} (x, \lambda) u_{-} (x, \lambda)$. In particular, Theorem \ref{2509021637} yields $Tf \in H^2 _{\loc} (\R)$. By using $(-\partial^2 _x +V(x)-\lambda^2)u_{\pm} =0$, we obtain $(-\partial^2 _x +V(x)-\lambda^2)Tf(x) =f(x)$. 

Next we see that $Tf$ satisfies the outgoing radiation condition. Using $\frac{\partial r}{\partial x} = \sgn x$ for $r:=|x|$, we obtain
\begin{align*}
(\partial_r -i\lambda)Tf(x) = (\sgn x \hspace{1mm}\partial_x -i\lambda)Tf(x) &=  \frac{(\sgn x \hspace{1mm}\partial_x u_{-} (x, \lambda) -i\lambda u_{-} (x, \lambda))}{\mathrm{Wr}(\lambda)} \int_{x}^{\infty} u_+ (y, \lambda) f(y)dy \\
&\quad + \frac{(\sgn x \hspace{1mm}\partial_x u_{+} (x, \lambda) -i\lambda u_{+} (x, \lambda))}{\mathrm{Wr}(\lambda)} \int_{-\infty}^{x} u_- (y, \lambda) f(y)dy.
\end{align*}
Since $u_+ (\cdot, \lambda)f(\cdot) \in L^1 (\R)$, we have $\int_{x}^{\infty} u_+ (y, \lambda) f(y)dy \to 0$ as $x \to \infty$. Using the boundedness of $\partial_x u_{-} (\cdot, \lambda)$ and $u_{-} (\cdot, \lambda)$ in Theorem \ref{2509021637}, the first term vanishes as $x \to \infty$. On the other hand, as $x \to -\infty$, we know
\begin{align*}
\sgn x \hspace{1mm}\partial_x u_{-} (x, \lambda) -i\lambda u_{-} (x, \lambda) = i(\lambda^2 -V(x))^{-\frac{1}{4}} e^{-iy(x, \lambda)} \{\sqrt{\lambda^2 -V(x)} -\lambda\} \to 0,
\end{align*}
where we have used (\ref{2509031759}) and (\ref{2509031801}). Using $u_+ (\cdot, \lambda)f(\cdot) \in L^1 (\R)$, the first term goes to $0$ as $x \to -\infty$. By the symmetry, the second term also vanishes as $|x| \to \infty$ and we obtain $(\partial_r -i\lambda)Tf(x) = o(1)$ as $|x| \to \infty$. By Sommerfeld's radiation condition \cite[p.407]{Y},
\begin{align*}
Tf(x) = R(\lambda^2 +i0)f(x)
\end{align*}
holds and the claim follows.
\end{proof}

Now we prove Theorem \ref{2509041030}.

\begin{proof}[Proof of Theorem \ref{2509041030}]
First we show
\begin{align}\label{eq:spprojJost}
\tilde{E} (\lambda, x, x')=\frac{\lambda}{\pi i } \left\{\mathrm{Wr}(\lambda)^{-1}u_+ (x, \lambda)u_{-} (x', \lambda) -\overline{\mathrm{Wr} (\lambda)}^{-1}\overline{u_+ (x, \lambda)} \overline{u_{-} (x', \lambda)} \right\}.
\end{align}
Using the identity \eqref{2512101147} and $R(\lambda^2 -i0) (x, x') = \overline{R(\lambda^2 +i0) (x', x)}$, we have \eqref{eq:spprojJost} for $x>x'$ and
\begin{align}\label{eq:spprojJost2}
\tilde{E} (\lambda, x, x') = \frac{\lambda}{\pi i} \left\{\mathrm{Wr}(\lambda)^{-1} u_- (x, \lambda)u_+ (x', \lambda)  -\overline{\mathrm{Wr} (\lambda)}^{-1}\overline{u_- (x, \lambda)} \overline{u_+ (x', \lambda)} \right\}
\end{align}
for $x<x'$.
By \eqref{eq:Jostrelation} and \eqref{eq:Jostcoeff}, for $x<x'$,
\begin{align*}
\frac{\pi i}{\l}\tilde{E} (\lambda, x, x')=&\frac{i}{2}b(\l)^{-1}(-\overline{a(\l)}u_+(x,\l)+b(\l)\overline{u_+(x,\l)})(a(\l)u_-(x',\l)+b(\l)\overline{u_-(x',\l)})\\
&+\frac{i}{2}\overline{b(\l)}^{-1}(-a(\l)\overline{u_+(x,\l)}+\overline{b(\l)}u_+(x,\l))(\overline{a(\l)}\cdot\overline{u_-(x',\l)}+\overline{b(\l)}u_-(x',\l))\\
=&\frac{i}{2}b(\l)^{-1}(u_+(x,\l)u_-(x',\l)+\overline{u_+(x,\l)u_-(x',\l)})=\frac{\pi i}{\l}(\text{RHS of \eqref{eq:spprojJost}}),
\end{align*}
where we use $|a(\l)|^2-|b(\l)|^2=1$ as is proved in just after \eqref{eq:Jostcoeffrel}.

We recall $\tilde{a}_{\s_1,\s_2}$ and $\mathrm{Wr}(\l)$ from Theorem \ref{2509021637}.
By setting
\begin{align*}
&\tilde{b}_{+, +} (\lambda, x, x') = -i\pi^{-1}\mathrm{Wr} (\lambda)^{-1}\tilde{a}_{+, +} (x, \lambda) \tilde{a}_{-, +} (x', \lambda)+ i\pi^{-1} \overline{\mathrm{Wr} (\lambda)}^{-1} \overline{\tilde{a}_{+,-} (x, \lambda) \tilde{a}_{-, -} (x', \lambda)} , \\
&\tilde{b}_{+, -} (\lambda, x, x') = -i\pi^{-1} \mathrm{Wr} (\lambda)^{-1}\tilde{a}_{+, +} (x, \lambda) \tilde{a}_{-, -} (x', \lambda) + i\pi^{-1} \overline{\mathrm{Wr} (\lambda)}^{-1}\overline{\tilde{a}_{+, -} (x, \lambda) \tilde{a}_{-, +} (x', \lambda)}, \\
&\tilde{b}_{-, +} (\lambda, x, x') = -i\pi^{-1}\mathrm{Wr} (\lambda)^{-1}\tilde{a}_{+, -} (x, \lambda) \tilde{a}_{-, +} (x', \lambda) + i\pi^{-1}\overline{\mathrm{Wr} (\lambda)}^{-1}\overline{\tilde{a}_{+, +} (x, \lambda) \tilde{a}_{-, -} (x', \lambda)}, \\
&\tilde{b}_{-, -} (\lambda, x, x') = -i\pi^{-1}\mathrm{Wr} (\lambda)^{-1}\tilde{a}_{+, -} (x, \lambda) \tilde{a}_{-, -} (x', \lambda) + i\pi^{-1}\overline{\mathrm{Wr} (\lambda)}^{-1}\overline{\tilde{a}_{+, +} (x, \lambda) \tilde{a}_{-, +} (x', \lambda)},
\end{align*}
the spectral density is written as
\begin{align*}
\tilde{E} (\lambda, x, x') = \frac{\l}{(\lambda^2 -V(x))^{\frac{1}{4}} (\lambda^2 -V(x'))^{\frac{1}{4}}} \sum_{\sigma_{1}, \sigma_{2} \in \{\pm\}} \tilde{b}_{\sigma_{1}, \sigma_{2}} (\lambda, x, x') e^{iS_{\sigma_{1}, \sigma_{2}} (\lambda, x, x')},
\end{align*}
where $S_{\sigma_{1}, \sigma_{2}}$ are from (\ref{2510201028}). By Theorem \ref{2509021637} $(ii)$ and $(iv)$, we have $\left| \partial^{\alpha} _{\lambda} \tilde{b}_{\sigma_{1}, \sigma_{2}} (\lambda, x, x')\right| \lesssim |\lambda|^{-\alpha}$ for any $\alpha=0,1,2$ and
\begin{align*}
\left|\partial^{\alpha} _{\lambda}\left(\l  (\lambda^2 -V(x))^{-\frac{1}{4}}(\lambda^2 -V(x'))^{-\frac{1}{4}}\right)\right| \lesssim& (\lambda^2 -V(x))^{-\frac{1}{4}}(\lambda^2 -V(x'))^{-\frac{1}{4}}|\lambda|^{1-\alpha}\\
\lesssim& \min (|\l|^{1-\a}\jap{x}^{\frac{\mu}{4}}\jap{x'}^{\frac{\mu}{4}},|\lambda|^{-\alpha})
\end{align*}
for any $\alpha\in\mathbb{N}_0$. 
Therefore the claim follows by setting
\begin{align*} 
b_{\sigma_{1}, \sigma_{2}} (\lambda, x, x') =  \frac{\l}{(\lambda^2 -V(x))^{\frac{1}{4}} (\lambda^2 -V(x'))^{\frac{1}{4}}} \tilde{b}_{\sigma_{1}, \sigma_{2}} (\lambda, x, x')
\end{align*}
where \eqref{eq:ampaddecay} follows from \eqref{261271003}.
\end{proof}

\begin{remark}
We also have the expression
\begin{align*}
\tilde{E} (\lambda, x, x') = \frac{2\lambda}{\pi |\mathrm{Wr}(\lambda)|^2} \left(u_+ (x, \lambda) \overline{u_+ (x', \lambda)} + u_- (x, \lambda) \overline{u_- (x', \lambda)}\right).
\end{align*}
On the other hand, since our phase function satisfies $e^{iy(x,\l)}\sim e^{i|\l|x}$ (not $e^{i\l x}$) as $|x|\to \infty$, we do \textbf{not} have the expression
\begin{align*}
\frac{2\lambda}{\pi |\mathrm{Wr}(\lambda)|^2} \left(u_+ (x, \lambda)u_+ (x', -\lambda) + u_- (x, \lambda)u_- (x', -\lambda) \right)
\end{align*}
as \cite[p.211 Proposition 5.1.5]{Y}. Perhaps, we could obtain the same asymptotics for $\overline{u_{\pm} (x, \lambda)}$ and $u_{\pm} (x, -\lambda)$ using a different branch of $\sqrt{\cdot}$ like \cite{Y0}.
\end{remark}

\section{Estimates of oscillatory integrals associated with the spectral projection}
The purpose of this section is to apply the stationary phase theorems established in Section \ref{section:prelim} to our WKB expression of the spectral density \eqref{2509041043}. Throughout of this section, we suppose that Assumption \ref{assump:V} holds. In particular, we have
\begin{align}\label{assum:Vbound}
-C_{V,1}\jap{x}^{-\m}\leq V(x)\leq -C_{V,2}\jap{x}^{-\m}
\end{align}
for constants $C_{V,1},C_{V,2}>0$.

\subsection{Estimates of oscillatory integrals associated with diagonal terms}\label{2512101237}

We define the phase function
\begin{align*}
\Phi(\l):=\Phi(\l,t,x,x')=-t\l^2+S(\l,x,x'),\quad S(\l)=S(\l,x,x'):=\int_{x'}^x\sqrt{\l^2-V(s)}ds
\end{align*}
which correspond to the ones in \eqref{2509041043} for $(\s_1,\s_2)=(+,-)$. We consider the oscillatory integral
\begin{align}\label{eq:I(a)osc}
I(a):=I_{t,x,x'}(a)=\int_0^{\infty}a(\l)e^{i\Phi(\l)}d\l.
\end{align}

%
%
%

\begin{thm}\label{thm:WKBosc}
There exists $C>0$ such that
\begin{align*}
|I(a)|\leq C|t|^{-\frac{1}{2}}( \|a\|_{L^{\infty}}+ \|\l\pa_{\l}a\|_{L^{\infty}})+C\jap{\max(|x|,|x'|)}^{-\frac{\m}{2}}|t|^{-\frac{1}{2}}( \|\l^{-1}a\|_{L^{\infty}}+ \|\pa_{\l}a\|_{L^{\infty}})
\end{align*}
for all $a\in C^{1}((0,\infty))$ provided the right hand side is bounded, $t\neq 0$, and $x,x'\in \R$.

\end{thm}
We will use the following calculation frequently:
\begin{align}
\pa_{\l}\Phi(\l)=&-2t\l+\int_{x'}^x\frac{\l}{\sqrt{\l^2-V(s)}}ds,\quad \pa_{\l}^2\Phi(\l)=-2t+\int_{x'}^x\frac{-V(s)}{(\l^2-V(s))^{\frac{3}{2}}}ds,\label{eq:phaseder1}\\
\pa_{\l}^3\Phi(\l)=&3\l\int_{x'}^x\frac{V(s)}{(\l^2-V(s))^{\frac{5}{2}}}ds. \label{eq:phaseder2}
\end{align}
Moreover, we introduce important quantities
\begin{align}\label{eq:defM_1M_2}
M_1:=\left|\int_{x'}^x|V(s)|^{-\frac{1}{2}}ds\right|,\quad M_2:=\left|\int_{x'}^x|V(s)|^{-\frac{3}{2}}ds\right|.
\end{align}

To prove Theorem \ref{thm:WKBosc}, we assume that $t>0$ holds. The other case is similarly dealt with. Furthermore, the proof for $x\leq x'$ is immediate as the two terms in the above calculation of $\pa_{\l}\Phi(\l)$ have the same sign:

\begin{proof}[Proof of Theorem \ref{thm:WKBosc} for $t>0$ and $x\leq x'$]
Under the assumptions $t>0$ and $x\leq x'$, the two terms in the above calculation of $\pa_{\l}\Phi(\l)$ have the same sign. Therefore, 
\begin{align*}
&|\pa_{\l}\Phi(\l)|\geq 2t\l+|x-x'|,\quad |\pa_{\l}^2\Phi(\l)|\leq \l^{-1}(2t\l+|x-x'|)\quad \text{for $\l\geq \jap{\max(|x|,|x'|)}^{-\frac{\m}{2}}$},\\
&|\pa_{\l}\Phi(\l)|\geq (2t+M_1)\l,\quad |\pa_{\l}^2\Phi(\l)|\leq 2t+M_1\quad \text{for $\l\leq 2\jap{\max(|x|,|x'|)}^{-\frac{\m}{2}}$},
\end{align*}
where we use that $V$ is a negative function. Put $\tilde{K}_1:=\{\l\geq \jap{\max(|x|,|x'|)}^{-\frac{\m}{2}}\}$, $\tilde{K}_2:=\{\l\leq 2\jap{\max(|x|,|x'|)}^{-\frac{\m}{2}}\}$, $\tilde{K}_{1,1}:=\{\l\leq 2|x-x'|/t\}\cap \tilde{K}_1$ and $\tilde{K}_{1,2}:= \{\l\geq |x-x'|/t\}\cap \tilde{K}_1$. By Lemma \ref{lem:PUest}, we may assume that $a$ is supported in $\tilde{K}_{1,1}$, $\tilde{K}_{1,2}$ or $\tilde{K}_2$.

Firstly, when $a(\l)$ is supported in $\tilde{K}_{1,1}$, then we can apply Lemma \ref{Lem:Qstph1} with $m=|x-x'|/t$ $M=t/2$ under the condition \eqref{eq:stphaseas1}. Secondly, if $a(\l)$ is supported in $\tilde{K}_{1,2}$, then we use Lemma \ref{Lem:Qstph1} with $m=0$ and $M=t$ under the condition \eqref{eq:stphaseas2}. Finally, for $a(\l)$ supported in $\tilde{K}_2$, Lemma \ref{Lem:Qstph1} with $M=2t+M_1$ and $m=0$ under the assumption \eqref{eq:stphaseas2} can be applied. In this way, we have $|I(a)|\lesssim t^{-\frac{1}{2}}(\|a\|_{L^{\infty}}+\|\l\pa_{\l}a\|_{L^{\infty}})$ for all the cases. This completes the proof.
\end{proof}

In the following, we assume
\begin{align}\label{assu:txysign}
t>0,\quad x\geq x',\quad |x|\geq |x'|.
\end{align}
The case $|x|\leq |x'|$ is similarly treated.

Thanks to Lemma \ref{lem:PUest}, we have only to consider the following regimes separately:
\begin{align*}
K_{1,R}:=\{\l\in (0,\infty)\mid \l\geq  R\jap{x}^{-\frac{\m}{2}}\},\quad K_{2,R}:=\{\l\in (0,\infty)\mid \l\leq  R\jap{x}^{-\frac{\m}{2}}\}.
\end{align*}

\subsubsection{Estimates in high-energy regime}\label{262221527}

We write $K_{1,R}$ as a sum of three regions
\begin{align*}
K_{1,1,R}:=&\left\{\l\in K_{1,R}\,\middle|\,  \l\leq\frac{2|x-x'|}{\sqrt{R}t}   \right\},\quad K_{1,2,R}:=\left\{\l\in K_{1,R}\,\middle|\, \frac{|x-x'|}{\sqrt{R}t}\leq \l\leq \frac{2\sqrt{R}|x-x'|}{t} \right\},\\
K_{1,3,R}:=&\left\{\l\in K_{1,R}\,\middle|\, \l\geq \frac{\sqrt{R}|x-x'|}{t} \right\}.
\end{align*}

\begin{proposition}\label{prop:disphigh}
For sufficiently large $R\geq 1$, there exists $C>0$ such that
\begin{align*}
|I(a)|\leq Ct^{-\frac{1}{2}}( \|a\|_{L^{\infty}}+ \|\l\pa_{\l}a\|_{L^{\infty}})
\end{align*}
for all $a\in C^{1}((0,\infty))$ supported in $K_{1,R}$ provided the right hand side is bounded and $(t,x,x')$ satisfies \eqref{assu:txysign}.

\end{proposition}

The next two lemmas show that the phase function can be approximated as $\Phi\approx -t\l^2+\l(x-x')$ for $\l\in K_{1,R}$.

\begin{lemma}\label{lem:phasederiuniv}
There exists $c_0>0$ depending only on $\m$ and $C_{V,1}$ such that
\begin{align*}
\frac{1}{4}|x-x'|\leq \left|\int_{x'}^x\frac{\l}{\sqrt{\l^2-V(s)}}ds \right|\leq |x-x'|,\quad \left|\int_{x'}^x\frac{V(s)}{(\l^2-V(s))^{\frac{3}{2}}}ds \right|\leq \frac{|x-x'|}{\l}
\end{align*}
when $\l\geq  c_0\jap{x}^{-\frac{\m}{2}}$, $x\geq x'$ and $|x|\geq |x'|$ hold.

\end{lemma}

\begin{proof}

Since $V(s)<0$, we have $(\l^2-V(s))^{-\frac{1}{2}}\leq \l^{-1}$ and $(\l^2-V(s))^{-\frac{3}{2}}\leq \l^{-1}|V(s)|^{-1}$. Therefore,
\begin{align*}
\left|\int_{x'}^x\frac{\l}{\sqrt{\l^2-V(s)}}ds \right|\leq \left|\int_{x'}^x\frac{\l}{\l}ds \right|=|x-x'|,\quad \left|\int_{x'}^x\frac{V(s)}{(\l^2-V(s))^{\frac{3}{2}}}ds \right|\leq  \frac{|x-x'|}{\l}
\end{align*}
which prove the upper bound of the first inequality and the second one. Thus, it suffices to prove the lower bound of the first inequality.

Since the integrant of $\int_{x'}^x\frac{\l}{\sqrt{\l^2-V(s)}}ds $ is always non-negative, we have
\begin{align*}
\left|\int_{x'}^x\frac{\l}{\sqrt{\l^2-V(s)}}ds \right|\underbrace{=}_{x\geq x'}\int_{x'}^x\frac{\l}{\sqrt{\l^2-V(s)}}ds\geq \int_{\max(x',\frac{1}{2}x)}^x\frac{\l}{\sqrt{\l^2-V(s)}}ds.
\end{align*}
For $\max(x',\frac{1}{2}x)\leq s\leq x$, 
\begin{align*}
(\l^2-V(s))^{-\frac{1}{2}}\geq (\l^2+C_{V,1}\jap{s}^{-\m})^{-\frac{1}{2}}\underbrace{\geq}_{\mathclap{\max(x',\frac{1}{2}x)\leq s\leq x}} (\l^2+C_{V,1}\jap{x/2}^{-\m})^{-\frac{1}{2}}\underbrace{\geq}_{\mathclap{\l\geq c_0\jap{x}^{-\frac{\m}{2}},c_0\gg 1 }}  \frac{3}{4\l}.
\end{align*}
Integrating this from $\max(x',\frac{1}{2}x)$ to $x$, we obtain
\begin{align*}
\int_{\max(x',\frac{1}{2}x)}^x\frac{\l}{\sqrt{\l^2-V(s)}}ds\geq \frac{3(x-\max(x',\frac{1}{2}x))}{4\l}.
\end{align*}
Finally, we observe that $x-\max(x',\frac{1}{2}x)\geq \frac{1}{3}(x-x')$. Combining them, we can prove the lower bound of the first inequality.
\end{proof}

\begin{lemma}\label{lem:largephase}
For sufficiently large $R\geq 1$, we have the following:

\vspace{1mm}
\noindent$(i)$ $|\pa_{\l}\Phi(\l)|\gtrsim |x-x'|$ and $|\pa_{\l}^2\Phi(\l)|\lesssim \l^{-1}|x-x'|$ for $\l\in K_{1,1,R}$.

\vspace{1mm}
\noindent$(ii)$ $|\pa_{\l}\Phi(\l)|\gtrsim t\l$ and $|\pa_{\l}^2\Phi(\l)|\lesssim t$ for $\l\in K_{1,3,R}$.

\vspace{1mm}
\noindent$(iii)$ There exists $\l_0>0$ such that $|\pa_{\l}\Phi(\l)|\gtrsim t|\l-\l_0|$ and $|\pa_{\l}^2\Phi(\l)|\lesssim t$ for $\l\in K_{1,2,R}$.

\end{lemma}

\begin{proof}
Suppose that $R\geq \max(2^{10},(24c_0)^2)$ holds, where $c_0>0$ is a constant in Lemma \ref{lem:phasederiuniv}.

\noindent$(i)$ For $\l\in K_{1,1,R}$, we have $t\l\leq 2R^{-\frac{1}{2}}|x-x'|\leq 8^{-1}|x-x'|$ and $t\lesssim \l^{-1}|x-x'|$. By Lemma \ref{lem:phasederiuniv}, for $\l\in K_{1,1,R}$,
\begin{align*}
&|\pa_{\l}\Phi(\l)|\geq \left|\int_{x'}^x\frac{\l}{\sqrt{\l^2-V(s)}}ds \right|-2t\l\geq 4^{-1}|x-x'|-2t\l\geq  8^{-1}|x-x'|,\\
&|\pa_{\l}^2\Phi(\l)|\leq 2t+\left|\int_{x'}^x\frac{|V(s)|}{(\l^2-V(s))^{\frac{3}{2}}}ds \right|\lesssim \l^{-1}|x-x'|.
\end{align*}

\vspace{1mm}
\noindent$(ii)$ For $\l\in K_{1,3,R}$, we have $|x-x'|\leq R^{-\frac{1}{2}}t\l\leq t\l$.
By Lemma \ref{lem:phasederiuniv}, for $\l\in K_{1,3,R}$,
\begin{align*}
&|\pa_{\l}\Phi(\l)|\geq 2t\l-\left|\int_{x'}^x\frac{\l}{\sqrt{\l^2-V(s)}}ds \right|\geq 2t\l-|x-x'|\geq t\l,\\
&|\pa_{\l}^2\Phi(\l)|\leq 2t+\left|\int_{x'}^x\frac{|V(s)|}{(\l^2-V(s))^{\frac{3}{2}}}ds \right|\leq 2t+\frac{|x-x'|}{2\l}\lesssim t.
\end{align*}

\vspace{1mm}
\noindent$(iii)$ We may assume $K_{1,2,R}\neq \emptyset$. We observe that
\begin{align*}
&|\pa_{\l}^2\Phi(\l)|\leq 2t+\left|\int_{x'}^x\frac{|V(s)|}{(\l^2-V(s))^{\frac{3}{2}}}ds \right|\leq 2t+\frac{|x-x'|}{\l}\lesssim t
\end{align*}
for $\l\in K_{1,2,R}$ by Lemma \ref{lem:phasederiuniv}. The next task is to deduce a lower bound of $|\pa_{\l}\Phi(\l)|$.

Put $f(\l)=-2t+\int_{x'}^x(\l^2-V(s))^{-\frac{1}{2}}ds(=\l^{-1}\pa_{\l}\Phi(\l))$ for $\l>0$.
First, we show that $f(\l)$ has a unique zero $\l_0$, which satisfies $\l_0\in (\frac{x-x'}{12t},\frac{x-x'}{t})$. In this case, $f(\l)\leq -2t+\int_{x'}^x\l^{-1}ds=-2t+\l^{-1}(x-x')$. Then $f(\frac{x-x'}{t})\leq -t<0$. On the other hand, it turns out from $K_{1,2,R}\neq \emptyset$ that $R\jap{x}^{-\m}\leq \frac{2\sqrt{R}(x-x')}{t}$ and hence $\frac{x-x'}{12t}\geq \frac{\sqrt{R}}{24}\jap{x}^{-\m}\geq c_0\jap{x}^{-\m}$ by $R\geq (24c_0)^2$. Thus, we apply Lemma \ref{lem:phasederiuniv} with $\l=\frac{x-x'}{12t}$ and obtain 
\begin{align*}
-2t+\left.\int_{x'}^x(\l^2-V(s))^{-\frac{1}{2}}ds\right|_{\l=\frac{x-x'}{12t}}\geq -2t+\left.\frac{x-x'}{4\l}\right|_{\l=\frac{x-x'}{12t}}=t>0.
\end{align*}
Since $f'(\l)<0$ for $\l>0$ (due to the assumption $x\geq 0$ and \eqref{assu:txysign}), the intermediate value theorem yields the existence of a unique zero $\l_0\in (\frac{x-x'}{12t},\frac{x-x'}{t})$ of $f$. Now we have
\begin{align*}
|f(\l)|=&|f(\l)-f(\l_0)|=\left|\int_{x'}^x\left(\frac{1}{\sqrt{\l^2-V(s)}}-\frac{1}{\sqrt{\l_0^2-V(s)}}\right)ds \right|\\
=&\int_{x'}^x\frac{|\l^2-\l_0^2|}{\sqrt{\l^2-V(s)}\sqrt{\l_0^2-V(s)}\left(\sqrt{\l^2-V(s)}+\sqrt{\l_0^2-V(s)} \right) }ds.
\end{align*}
Since $\l,\l_0\geq c_1\jap{x}^{-\m}$ and $c_2^{-1}\l\leq \l_0\leq c_2\l$ for $\l\in K_{1,2,R}$ with constants $c_1,c_2>0$, we have 
\begin{align*}
|\pa_{\l}\Phi(\l)|=\l|f(\l)|\geq c_3\l_0^{-2}|\l^2-\l_0^2||x-x'|\geq c_4\l_0^{-1}|\l-\l_0||x-x'|\geq c_5t|\l-\l_0|
\end{align*}
with constants $c_3,c_4,c_5>0$. 
\end{proof}

\begin{proof}[Proof of Proposition \ref{prop:disphigh}]
We take $R\geq 1$ sufficiently large such that Lemma \ref{lem:largephase} can be applied.
By Lemma \ref{lem:PUest}, it suffices to show the existence of $C>0$ such that
\begin{align}\label{eq:bounda_j}
|I(a_j)|\leq Ct^{-\frac{1}{2}}( \|a_j\|_{L^{\infty}}+ \|\l\pa_{\l}a_j\|_{L^{\infty}})
\end{align}
when $a_j$ is supported in $K_{1,j,R}$ and $(t,x,x')$ satisfies \eqref{assu:txysign}. For $j=1$, the condition \eqref{eq:stphaseas1} with $m=|x-x'|$ and  $M=Rt/2$ is satisfied. For $j=3$, the condition \eqref{eq:stphaseas2} with $m=0$ and $M=t$ is fulfilled. For $j=2$, we can check the condition \eqref{eq:stphaseas2} with $m=\l_0$ and $M=t$. In any case, we obtain \eqref{eq:bounda_j} from Lemma \ref{Lem:Qstph1}.
\end{proof}

\subsubsection{Estimates in low-energy regime}\label{262221507}

We recall $K_{2,R}:=\{\l\in (0,\infty)\mid \l\leq  R\jap{x}^{-\frac{\m}{2}}\}$.

\begin{proposition}\label{prop:osclow}
Fix $R\geq  1$. Then there exists $C>0$ such that
\begin{align*}
|I(a)|\leq Ct^{-\frac{1}{2}}( \|a\|_{L^{\infty}}+ \|\l\pa_{\l}a\|_{L^{\infty}})+C\jap{x}^{-\frac{\m}{2}}t^{-\frac{1}{2}}( \|\l^{-1}a\|_{L^{\infty}}+ \|\pa_{\l}a\|_{L^{\infty}})
\end{align*}
for all $a\in C^{1}((0,\infty))$ supported in $K_{2,R}$ with the right hand side bounded and for $(t,x,x')$ satisfying \eqref{assu:txysign}. 

\end{proposition}

We recall that the quantities $M_1$ and $M_2$ depending on $x,x'$ are defined in \eqref{eq:defM_1M_2}.

\begin{lemma}\label{lem:phasederiunivlow}

\noindent$(i)$ There exists $c_1>1$ such that
\begin{align*}
c_1^{-1}\jap{x}^{\frac{\m}{2}}|x-x'|\leq M_1\leq c_1\jap{x}^{\frac{\m}{2}}|x-x'|,\quad c_1^{-1}\jap{x}^{\frac{3\m}{2}}|x-x'|\leq M_2\leq c_1\jap{x}^{\frac{3\m}{2}}|x-x'|
\end{align*}
when $x\geq x'$ and $|x|\geq |x'|$ hold.

\noindent$(ii)$ For each $R\geq 1$, there exists $c_2>0$ such that
\begin{align*}
&c_2M_1\leq \left|\int_{x'}^x\frac{1}{\sqrt{\l^2-V(s)}}ds \right|\leq M_1 ,\quad \left|\int_{x'}^x\frac{|V(s)|}{(\l^2-V(s))^{\frac{3}{2}}}ds \right|\leq M_1,\\
&c_2M_2\leq \left|\int_{x'}^x\frac{V(s)}{(\l^2-V(s))^{\frac{5}{2}}}ds \right|\leq M_2
\end{align*}
when $\l\in K_{2,R}$, $x\geq x'$ and $|x|\geq |x'|$ hold.

\end{lemma}

\begin{proof}

\noindent$(i)$ The assumption \eqref{assum:Vbound} yields $C_{V,2}^{\frac{1}{2}}\int_{x'}^x\jap{s}^{\frac{\m}{2}}ds\leq M_1\leq C_{V,1}^{\frac{1}{2}}\int_{x'}^x\jap{s}^{\frac{\m}{2}}ds$, which shows the upper bound of $M_1$. When $x'\in[x/2,x]$, then $\int_{x'}^x\jap{s}^{\frac{\m}{2}}ds\geq \jap{y}^{\frac{\m}{2}}|x-x'|\gtrsim \jap{x}^{\frac{\m}{2}}|x-x'|$. On the other hand, if $x'\leq x/2$, then $\int_{x'}^x\jap{s}^{\frac{\m}{2}}ds\geq \int_{x/2}^x\jap{s}^{\frac{\m}{2}}ds\gtrsim \jap{x}^{\frac{\m}{2}}x\geq \jap{x}^{\frac{\m}{2}}|x-x'|$. Combining them, we obtain $M_1\gtrsim \jap{x}^{\frac{\m}{2}}|x-x'|$. The proof for $M_2$ is identical.

\noindent$(ii)$ The upper bounds are immediate consequences of an elementary inequality $(\l^2-V(s))^{-\frac{j}{2}}\leq |V(s)|^{-\frac{j}{2}}$, where we recall that $V$ is negative. The lower bounds for the first and third inequalities follow from the observation that the conditions $\l\lesssim \jap{x}^{-\m/2}$, $x\geq x'$, and $|x|\geq |x'|$ imply $(\l^2-V(s))^{-\frac{j}{2}}\gtrsim |V(s)|^{-\frac{j}{2}}$ uniformly in $x'\leq s\leq x$.
\end{proof}

\begin{lemma}

\noindent$(i)$ $|\pa_{\l}\Phi(\l)|\gtrsim \max(t,M_1)\l$ and $|\pa_{\l}^2\Phi(\l)|\lesssim \max(t,M_1)$ for $t\leq \frac{1}{4}c_2M_1$ or $t\geq M_1$, where $c_2$ is the constant appearing in Lemma \ref{lem:phasederiunivlow} $(ii)$.

\vspace{1mm}
\noindent$(ii)$ There exists $\l_0\geq 0$ such that $|\pa_{\l}\Phi(\l)|\gtrsim M_2\l|\l^2-\l_0^2|$ and $|\pa_{\l}^2\Phi(\l)|\lesssim M_2(\l^2+\l_0^2)$ for $\frac{1}{4}c_2M_1\leq t\leq \frac{1}{2}M_1$.

\vspace{1mm}
\noindent$(iii)$ There exists $m<0$ such that $|\pa_{\l}\Phi(\l)|\gtrsim M_2\l|\l^2-m|$ and $|\pa_{\l}^2\Phi(\l)|\lesssim M_2(\l^2+|m|)$.

\end{lemma}

\begin{proof}
\noindent$(i)$ By Lemma \ref{lem:phasederiunivlow},
\begin{align*}
&|\pa_{\l}\Phi(\l)|\geq \left|\int_{x'}^x\frac{\l}{\sqrt{\l^2-V(s)}}ds \right|-2t\l\underbrace{\geq}_{\mathclap{\text{Lemma \ref{lem:phasederiunivlow}}}} c_2M_1\l-2t\l\underbrace{\geq}_{\mathclap{t\leq \frac{c_2}{4}M_1}} \frac{c_2}{2}M_1\l\quad (\l\in K_{2,R},\,\, t\leq \frac{1}{4}c_2M_1),\\
&|\pa_{\l}\Phi(\l)|\geq 2t\l-\left|\int_y^x\frac{\l}{\sqrt{\l^2-V(s)}}ds \right|\underbrace{\geq}_{\text{Lemma \ref{lem:phasederiunivlow}}} 2t\l -M_1\l\underbrace{\geq}_{\mathclap{t\geq M_1}} t\l \quad (\l\in K_{2,R},\,\, t\geq M_1).
\end{align*}
Therefore, we have $|\pa_{\l}\Phi(\l)|\gtrsim \max(M_1,t)\l$ when $t\leq \frac{1}{4}c_1M_1$ or $t\geq M_1$ holds. Using Lemma \ref{lem:phasederiunivlow} again,
\begin{align}\label{eq:lowder2Phi}
|\pa_{\l}^2\Phi(\l)|\leq 2t+\left|\int_y^x\frac{|V(s)|}{(\l^2-V(s))^{\frac{3}{2}}}ds \right|\leq 2t+M_1\leq 2\max(t,M_1).
\end{align}

\vspace{1mm}
\noindent$(ii)$ Put $f(\l)(=\l^{-1}\pa_{\l}\Phi(\l))=-2t+\int_{x'}^x(\l^2-V(s))^{-\frac{1}{2}}ds$ for $\l\geq 0$. We show that $f$ has a zero $\l_0\geq 0$ satisfying $\l_0\lesssim \jap{x}^{-\frac{\m}{2}}$. To see this, we observe 
\begin{align*}
&f(0)\underbrace{=}_{\mathclap{x\geq x'}}-2t+M_1\underbrace{>}_{\mathclap{t<\frac{1}{2}M_1}}0,\\
&f(4c_1c_2^{-1}\jap{x}^{-\frac{\m}{2}})\leq -2t+\l^{-1}(x-x')|_{\l=4c_1c_2^{-1}\jap{x}^{-\frac{\m}{2}}}=-2t+(4c_1c_2^{-1})^{-1}\jap{x}^{\frac{\m}{2}}(x-x')\\
&\leq-2t+(4c_1c_2^{-1})^{-1}c_1M_1\underbrace{\leq}_{t\geq \frac{1}{4}c_2M_1} -2t+t=-t<0,
\end{align*}
where we use Lemma \ref{lem:phasederiunivlow} $(i)$. Therefore, the intermediate value theorem yields the existence of a zero $\l_0$ of $f$ satisfying $0\leq \l_0\leq 4c_1c_2^{-1}\jap{x}^{-\frac{\m}{2}}$. 
Now we have
\begin{align*}
|f(\l)|=&|f(\l)-f(\l_0)|=\left|\int_{x'}^x\left(\frac{1}{\sqrt{\l^2-V(s)}}-\frac{1}{\sqrt{\l_0^2-V(s)}}\right)ds \right|\\
=&\int_{x'}^x\frac{|\l^2-\l_0^2|}{\sqrt{\l^2-V(s)}\sqrt{\l_0^2-V(s)}\left(\sqrt{\l^2-V(s)}+\sqrt{\l_0^2-V(s)} \right) }ds.
\end{align*}
We note that there exists $C'>0$ such that $|V(s)|^{\frac{1}{2}}\leq \sqrt{\l^2-V(s)}$ and $\sqrt{\l_0-V(s)}\leq C'|V(s)|^{\frac{1}{2}}$ for $\l\in K_{2,R}$. This implies $(C')^{-3}M_2|\l^2-\l_0^2|\leq |f(\l)|\leq M_2|\l^2-\l_0^2|$ and hence
\begin{align*}
(C')^{-3}M_2\l|\l^2-\l_0^2|\leq |\pa_{\l}\Phi(\l)|\leq M_2\l|\l^2-\l_0^2|
\end{align*}
for $\l\in K_{2,R}$. On the other hand,
\begin{align*}
|\pa_{\l}^2\Phi(\l)|=&\left|-2t+\int_{x'}^x\frac{1}{\sqrt{\l^2-V(s)}}ds-\int_{x'}^x\frac{\l^2}{(\l^2-V(s))^{\frac{3}{2}}}ds \right|\\
\leq& |f(\l)|+\int_{x'}^x\frac{\l^2}{(\l^2-V(s))^{\frac{3}{2}}}ds\leq M_2|\l^2-\l_0^2|+M_2\leq 2M_2(\l^2+\l_0^2)
\end{align*}

\vspace{1mm}
\noindent$(iii)$ First, we observe that $\pa_{\l}^3S(\l)=3\l\int_{x'}^xV(s)(\l^2-V(s))^{-\frac{5}{2}}ds\leq -3c_2M_2\l $ and $\pa_{\l}^3S(\l)\geq -3M_2\l$ for $\l\in K_{2,R}$. Integrating this inequality twice and using $\pa_{\l}S(0)=0$ and $\pa_{\l}^2S(0)=M_1$, we have $\pa_{\l}S(\l)\leq M_1\l-\frac{c_2}{2}M_2\l^3$ and $\pa_{\l}\Phi(\l)\leq(-2t+ M_1)\l-\frac{c_2}{2}M_2\l^3<0$ by $t>M_1/2$. Therefore,
\begin{align*}
|\pa_{\l}\Phi(\l)|\geq (M_1-2t)\l+\frac{c_2}{2}M_2\l^3=\frac{c_2}{2}M_2\l(\l^2-m)\quad (\l\in K_{2,R}),
\end{align*}
where we set $m:=2c_2^{-1}\frac{M_1-2t}{M_2}(<0)$. Similarly,
\begin{align*}
|\pa_{\l}^2\Phi(\l)|=&\left|-2t+M_1+ \int_0^{\l}\pa_{\m}^3S(\m)d\m \right|\underbrace{=}_{2t-M_1\geq 0,\pa_{\m}^3S(\m)\leq 0}2t-M_1-\int_0^{\l}\pa_{\m}^3S(\m)d\m\\
\leq& 2t-M_1+3M_2\int_0^{\l}\m d\m=2t-M_1+\frac{3}{2}M_2\l^2=\frac{c_2}{2}M_2|m|+\frac{3}{2}M_2\l^2\\
\lesssim&M_2(\l^2+|m|).
\end{align*}
\end{proof}

\begin{proof}[Proof of Proposition \ref{prop:osclow}]
For $t\leq \frac{1}{4}c_2M_1$ or $t\geq M_1$, Lemma \ref{Lem:Qstph1} with $M=\max(t,M_1)$ and $m=0$ implies $|I(a)|\lesssim t^{-\frac{1}{2}}( \|a\|_{L^{\infty}}+ \|\l\pa_{\l}a\|_{L^{\infty}})$.
When $\frac{1}{4}c_2M_1\leq t\leq \frac{1}{2}M_1$, the assumptions of Lemma \ref{lem:Qstph2} with $M=M_2$ and $m=\l_0^2$ are fulfilled. If $\frac{1}{2}M_1\leq t\leq M_1$, then we employ Lemma \ref{lem:Qstph2} with $M=M_2$. In the second and third cases, we have $|I(a)|\lesssim M_2^{-\frac{1}{2}}( \|\l^{-1}a\|_{L^{\infty}}+ \|\pa_{\l}a\|_{L^{\infty}})\lesssim\jap{x}^{-\frac{\m}{2}}t^{-\frac{1}{2}}( \|\l^{-1}a\|_{L^{\infty}}+ \|\pa_{\l}a\|_{L^{\infty}})$, where we have used $M_2\gtrsim \jap{x}^{\m}t$ that follows from Lemma \ref{lem:phasederiunivlow} $(i)$ and $t\gtrsim M_1$. In this way, the proof of Proposition \ref{prop:osclow} is completed.
\end{proof}

\begin{proof}[Proof of Theorem \ref{thm:WKBosc} when $t>0$, $x\geq x'$ and $|x|\geq |x'|$ hold]
The theorem directly follows from Lemma \ref{lem:PUest}, Propositions \ref{prop:disphigh} and \ref{prop:osclow}.
\end{proof}

\subsection{Estimates of oscillatory integrals associated with off-diagonal terms}\label{subsec:off-diagonal}

We define the phase function
\begin{align*}
\Phi(\l):=-t\l^2+S(\l,x,x'),\,\, S(\l)=S(\l,x,x'):=\int_{0}^x\sqrt{\l^2-V(s)}ds+\int_{0}^{x'}\sqrt{\l^2-V(s)}ds
\end{align*}
which correspond to the ones in \eqref{2509041043} for $(\s_1,\s_2)=(+,+)$. We consider the oscillatory integral of the form \eqref{eq:I(a)osc}. Then
\begin{align}
&\pa_{\l}\Phi(\l)=-2t\l+\int_{0}^x\frac{\l}{\sqrt{\l^2-V(s)}}ds+\int_{0}^{x'}\frac{\l}{\sqrt{\l^2-V(s)}}ds,\label{eq:phaseder3}\\
&\pa_{\l}^2\Phi(\l)=-2t+\int_{0}^x\frac{-V(s)}{(\l^2-V(s))^{\frac{3}{2}}}ds+\int_{0}^{x'}\frac{-V(s)}{(\l^2-V(s))^{\frac{3}{2}}}ds\label{eq:phaseder4}.
\end{align}

\begin{thm}\label{thm:WKBosc2}
Under the assumption \eqref{assum:Vbound}, we have the same bound of $I(a)$ as in Theorem \ref{thm:WKBosc} for $t\neq 0$ and $xx'\geq 0$. Moreover, there exists $C>0$ such that
\begin{align*}
|I(a)|\leq& C|t|^{-\frac{1}{2}}( \|a\|_{L^{\infty}}+ \|\l\pa_{\l}a\|_{L^{\infty}})+C\jap{\max(|x|,|x'|)}^{-\frac{\m}{2}}|t|^{-\frac{1}{2}}( \|\l^{-1}a\|_{L^{\infty}}+ \|\pa_{\l}a\|_{L^{\infty}})\\
&+C|t|^{-\frac{1}{2}}\min(|x|^2,|x'|^2)\|\l^2a\|_{L^{\infty}}
\end{align*}
for all $a\in C^{1}((0,\infty))$ provided the right hand side is bounded, $t\neq 0$ and $xx'<0$.

\end{thm}

\begin{remark}
When $V$ is an even function, $V(x)=V(-x)$, then we have
\begin{align*}
S_{+,+}(x,x')=S_{+,-}(x,-x'),\quad S_{-,-}(x,x')=S_{+,-}(-x',x),\quad S_{-,+}(x,x')=S_{+,-}(x',x).
\end{align*}
In this case, Theorem \ref{thm:WKBosc2} is a direct consequence of Theorem \ref{thm:WKBosc} without the additional term $|t|^{-\frac{1}{2}}\min(|x|^2,|x'|^2)\|\l^2a\|_{L^{\infty}}$. In our application, this term is controlled by the improved decay estimates \eqref{eq:ampaddecay} of $b_{+,+}$.

\end{remark}

\begin{proof}
We may assume $t>0$.
If $x, x' >0$ or $x, x' <0$, then the proof is an easy modification of that in Subsection \ref{2512101237} (with $x+x'$ instead of $x-x'$). Hence we assume $x' <0<x$. Moreover if $x$ or $x'$ is bounded, then we can include the oscillating factor $e^{iy(x, \lambda)}$ or $e^{iy(x', \lambda)}$ into the amplitude and perform the argument in Section \ref{2512101237}. 

In the following we suppose $t>0$, $x>M$ and $x' < -M$ for sufficiently large $M>0$. 
Firstly we split the oscillatory integral into
\begin{align*}
I(a)=I(\chi a)+I(\overline{\chi}a)=\mathcal{I}_0 + \mathcal{I}_1,
\end{align*}
where $\chi$ is a cut-off function as in Lemma \ref{lem:PUest}, $\chi_t=\chi(\sqrt{t}\l)$, and $\overline{\chi}_t:= 1-\chi_t$. We note that $\supp\chi_t\subset \{\lambda \le 2\sqrt{t}^{-1}\}$ and $\supp\overline{\chi}_t\subset \{\lambda \ge \sqrt{t}^{-1}\}$. By the support property of $\chi_t$, we have $|\mathcal{I}_0| \lesssim t^{-\frac{1}{2}}\|a\|_{L^{\infty}}$ . Thus, it suffices to estimate $\mathcal{I}_1$.

\underline{\textbf{Estimate when $|x'|\geq |x|$}} First, we consider the case where $t\lesssim x^2$.
By the support properties of $\overline{\chi}_t$,
\begin{align*}
|\mathcal{I}_1|\lesssim \int_{\sqrt{t}^{-1}}^{\infty}\l^{-2}d\l \cdot \|\l^2a\|_{L^{\infty}}= \sqrt{t}\|\l^2a\|_{L^{\infty}}\lesssim t^{-\frac{1}{2}} |x|^2\|\l^2a\|_{L^{\infty}}=t^{-\frac{1}{2}} \min(|x|^2,|x'|^2)\|\l^2a\|_{L^{\infty}}.
\end{align*}

Next, we assume that both $t\geq C_1^3x^2$ and $t \ge C_1 |x'|^2$ hold, where $C_1$ is a large constant determined later. If we take $C_1$ sufficiently large, then
\begin{align}\label{eq:anisophase1}
|\partial_{\lambda} \Phi (\lambda)| \ge \lambda (2t - C_V (x^{1+\frac{\mu}{2}} + |x'|^{1+\frac{\mu}{2}})) \gtrsim t\lambda,\quad |\partial^2 _{\lambda} \Phi (\lambda)|\lesssim t + x^{1+\frac{\mu}{2}} + |x'|^{1+\frac{\mu}{2}} \lesssim t
\end{align}
by \eqref{eq:phaseder3}, \eqref{eq:phaseder4} and $0<\m<2$, where $C_V >0$ depends only on $V$. Applying Lemma \ref{Lem:Qstph1} under the assumption \eqref{eq:stphaseas2} with $m=0$ and $M=t$, we obtain $|\mathcal{I}_1|\lesssim t^{-\frac{1}{2}}( \|a\|_{L^{\infty}}+ \|\l\pa_{\l}a\|_{L^{\infty}})$.

Finally, we assume that both $t\geq C_1^3x^2$ and $t \le C_1 |x'|^2$ hold. Taking $C_1$ sufficiently large such that $2t-\int_0^x\frac{1}{\sqrt{\l^2-V(s)}}ds\geq 2t-C_V|x|^{1+\frac{\m}{2}}\geq t$, we have
\begin{align*}
|\partial_{\lambda} \Phi (\lambda)| \ge t\lambda + \lambda \left| A(\l,x')\right|, \quad |\partial^2 _{\lambda} \Phi (\lambda)| \lesssim t + \left|A(\l,x')\right|
\end{align*}
by \eqref{eq:phaseder1}, where we set $A(\l,x'):=\int_{0}^{x'} \frac{ds}{\sqrt{\lambda^2 -V(s)}}$ and we have used that the two terms $-2t\l$ and $A(\l,x')$ have the same signatures for $t>0$ and $x'<0$. By the support property of $\overline{\chi}_t$ and the integration by parts, 
\begin{align*}
|\mathcal{I}_1|&\lesssim \int_{\sqrt{t}^{-1}}^{\infty} \left\{\frac{t +|A(\l,x')|}{\lambda^2 \left(t + |A(\l,x')|\right)^2} \|a\|_{L^{\infty}} + \frac{\|\lambda \partial_{\lambda} (\overline{\chi}_t a )\|_{L^{\infty}}}{\lambda^2 (t + |A(\l,x')|)} \right\}d\lambda \lesssim t^{-\frac{1}{2}}( \|a\|_{L^{\infty}}+ \|\l\pa_{\l}a\|_{L^{\infty}}).
\end{align*}

\underline{\textbf{Estimate when $|x|\geq |x'|$}} The estimate for $t\lesssim |x'|^2$ is similar to that when both $|x'|\geq |x|$ and $t\lesssim |x|^2$ hold. For $t\geq C_1x^2$ and $t\geq C_1^3|x'|^2$ with sufficiently large $C_1>0$, we have \eqref{eq:anisophase1}. Applying Lemma \ref{Lem:Qstph1} under the assumption \eqref{eq:stphaseas2} with $m=0$ and $M=t$, we obtain $|\mathcal{I}_1|\lesssim t^{-\frac{1}{2}}( \|a\|_{L^{\infty}}+ \|\l\pa_{\l}a\|_{L^{\infty}})$.

It remains to consider the case where both $t\leq C_1x^2$ and $t\geq C_1^3|x'|^2$, which yields $x\geq C_1|x'|$. Similarly to Subsection \ref{2512101237}, we split $\mathcal{I}_1$ into the low and high energy regime as in Section \ref{2512101237}. By virtue of Lemma \ref{lem:PUest}, it suffices for $a_1$ supported in $\{\lambda \ge R \jap{x}^{-\frac{\mu}{2}}\}$ and $a_2$ supported in $\{\lambda \le 2R \jap{x}^{-\frac{\mu}{2}}\}$ (with sufficiently large $R$) to prove
\begin{align*}
|I(a_1)|\lesssim& |t|^{-\frac{1}{2}}( \|a_1\|_{L^{\infty}}+ \|\l\pa_{\l}a_1\|_{L^{\infty}}),\,\, |I(a_2)|\lesssim \jap{\max(|x|,|x'|)}^{-\frac{\m}{2}}|t|^{-\frac{1}{2}}( \|\l^{-1}a\|_{L^{\infty}}+ \|\pa_{\l}a\|_{L^{\infty}})
\end{align*}
when all of $|x|\geq |x'|$, $t\leq C_1x^2$ and $t\geq C_1^3|x'|^2$ (with sufficiently large $C_1$) are satisfied.

Concerning $I(a_2)$, the argument in Subsubsection \ref{262221507} still works if we take $C_1$ sufficiently large. The point is that, using $x \ge C_1 |x'|$, we absorb the terms involving $x'$ into those with $x$. Almost all of the estimates in Subsubsection \ref{262221507} hold with $x+x'$ instead of $x-x'$ (since the corresponding terms to $M_1$ and $M_2$ are positive) and we omit the details.

To deal with $I(a_1)$, we further split the interval into three regions $J_{1, 1, R} := \{\lambda \le \frac{x}{\sqrt{R}t}\}$, $J_{1, 2, R}:= \{\frac{x}{\sqrt{R}t} \le \lambda \le \frac{\sqrt{R} x}{t}\}$ and $J_{1, 3, R}:= \{\lambda \ge \frac{\sqrt{R} x}{t}\}$ as in Subsubsection \ref{262221527}. If $\lambda \in J_{1, 1, R}$, we have
\begin{align*}
|\partial_{\lambda} \Phi (\lambda)| \ge \frac{x}{4} -3t\lambda \gtrsim x, \quad |\partial^2 _{\lambda} \Phi (\lambda)| \lesssim t + \frac{x}{\lambda} + \frac{|x'|}{\lambda} \lesssim \frac{x}{\lambda},
\end{align*}
which yield the decay rate $t^{-\frac{1}{2}}$ by Lemma \ref{Lem:Qstph1} with $m =x$ and $M = \sqrt{R} t$. On the other hand, if $\lambda \in J_{1, 3, R}$, we have
\begin{align*}
|\partial_{\lambda} \Phi (\lambda)| \ge 2t\lambda - x \gtrsim t\lambda, \quad |\partial^2 _{\lambda} \Phi (\lambda)| \lesssim t + \frac{x}{\lambda} + \frac{|x'|}{\lambda} \lesssim t,
\end{align*}
which also gives the decay rate $t^{-\frac{1}{2}}$ by Lemma \ref{Lem:Qstph1}.

Finally, if $\lambda \in J_{1, 2, R}$, we have $|\partial^2 _{\lambda} \Phi (\lambda)| \lesssim t + \frac{x}{\lambda} + \frac{|x'|}{\lambda} \lesssim t$ as above. Moreover, for $f(\lambda):= -2t + \int_{0}^{x} \frac{ds}{\sqrt{\lambda^2 -V(s)}} + \int_{0}^{x'} \frac{ds}{\sqrt{\lambda^2 -V(s)}} = \partial_{\lambda} \Phi (\lambda) / \lambda$, 
\begin{align*}
&f\left(\frac{x}{t}\right) \le -2t + \left. \frac{x}{\lambda} \right|_{\lambda = \frac{x}{t}} = -t <0, \quad f\left(\frac{x}{16t}\right) \ge -3t + \left. \frac{x}{4\lambda} \right|_{\lambda = \frac{x}{16t}} =t >0
\end{align*}
and the intermediate value theorem yields the existence of $\lambda_0 \in (\frac{x}{16t}, \frac{x}{t})$ such that $f(\lambda_0) =0$. Therefore we obtain
\begin{align*}
|f(\lambda)| &\ge |\lambda^2 -\lambda^2 _0| \left| \int_{0}^{x} \frac{ds}{\sqrt{\lambda^2 -V(s)} \sqrt{\lambda^2 _0 -V(s)}(\sqrt{\lambda^2 -V(s)} + \sqrt{\lambda^2 _0 -V(s)})}\right| \\
& \quad - |\lambda^2 -\lambda^2 _0| \left| \int_{0}^{x'} \frac{ds}{\sqrt{\lambda^2 -V(s)} \sqrt{\lambda^2 _0 -V(s)}(\sqrt{\lambda^2 -V(s)} + \sqrt{\lambda^2 _0 -V(s)})}\right| =: |\lambda^2 -\lambda^2 _0|(\mathrm{I}_1 -\mathrm{I}_2).
\end{align*}
Concerning the first term, we estimate
\begin{align*}
\mathrm{I}_1 \ge \int_{\frac{x}{2}}^{x} \frac{ds}{\sqrt{\lambda^2 -V(s)} \sqrt{R\lambda^2  -V(s)}(\sqrt{\lambda^2 -V(s)} + \sqrt{R\lambda^2 -V(s)})} \ge \frac{1}{2R} \int_{\frac{x}{2}}^{x} \frac{ds}{(\lambda^2 -V(s))^{\frac{3}{2}}} \ge \frac{x}{4R} \left(\frac{3}{4\lambda}\right)^3.
\end{align*}
For the second term we use $\mathrm{I}_2 \le \frac{|x'|}{\lambda^2 \lambda_0} \lesssim \frac{\sqrt{R} |x'|}{\lambda^3}$. Therefore, if $C_1$ is sufficiently large, we obtain
\begin{align*}
|\partial_{\lambda} \Phi (\lambda)| \ge \lambda |f(\lambda)| \gtrsim |\lambda^2 -\lambda^2 _0| \frac{x}{\lambda^2} \gtrsim |\lambda -\lambda_0| \frac{x}{\lambda} \sim t|\lambda -\lambda_0|, 
\end{align*}
which yields the decay rate $t^{-\frac{1}{2}}$ combined with Lemma \ref{Lem:Qstph1}.
\end{proof}


\subsection{Proof of the main theorem}

\begin{proof}[Proof of Theorem \ref{thm:main}]
We write $e^{-itP} E_{\mathrm{ac}}(P)=\int_{0}^{\infty} e^{-it\lambda^2} \tilde{E} (\lambda) d\lambda$, where $\tilde{E}$ is the spectral density defined in \eqref{def:spdensity}. We also note that the integral kernel of  $\tilde{E} (\lambda)$ has a WKB expression of the form \eqref{2509041043}. For $t\neq 0$ and $x,x'\in \R$ with $x\geq 0$ and $x'\leq 0$, it suffices to prove
\begin{align*}
\left|\int_0^{\infty}b_{\s_1,\s_2}(\l,x,x')e^{-it\l^2+iS_{\s_1,\s_2}(\l,x,x')}d\l\right|\lesssim \frac{1}{|t|^{\frac{1}{2}}}.
\end{align*}
When $t\neq 0$ and $x,x'\in \R$ with $x\leq 0$ and $x'\geq 0$, we use another WKB expression of $\tilde{E}(\l)$ explained in Remark \ref{rem:spprojWKBxx'<0} instead. Thus, we consider the case where $t\neq 0$, $x\geq 0$, and $x'\leq 0$ hold only.

First, we consider the $(+,-)$-case. By the symbolic estimates \eqref{2510201029} of $b_{\s_1,\s_2}$ and Theorem \ref{thm:WKBosc}, we have
\begin{align*}
\left|\int_0^{\infty}b_{+,-}(\l,x,x')e^{-it\l^2+iS_{+,-}(\l,x,x')}d\l\right|\lesssim |t|^{-\frac{1}{2}}+\jap{\max(|x|,|x'|)}^{-\frac{\m}{2}}|t|^{-\frac{1}{2}}\cdot \jap{x}^{\frac{\m}{4}}\jap{x'}^{\frac{\m}{4}}\lesssim |t|^{-\frac{1}{2}},
\end{align*}
where we note that $-t\l^2+S_{+,-}(\l,x,x')=\Phi(\l,x,x')$ in the notation in Theorem \ref{thm:WKBosc}.

Next, we consider the $(+,+)$-case. The estimate for $t\neq 0$ and $xx'\geq 0$ is similarly proved to the $(+,-)$-case due to the first statement of Theorem \ref{thm:WKBosc2}.
By \eqref{2510201029}, \eqref{eq:ampaddecay} and the second statement of Theorem \ref{thm:WKBosc2}, for $t\neq 0$ and $xx'<0$,
\begin{align*}
&\left|\int_0^{\infty}b_{+,+}(\l,x,x')e^{-it\l^2+iS_{+,+}(\l,x,x')}d\l\right|\\
&\lesssim |t|^{-\frac{1}{2}}+\jap{\max(|x|,|x'|)}^{-\frac{\m}{2}}|t|^{-\frac{1}{2}}\cdot \jap{x}^{\frac{\m}{4}}\jap{x'}^{\frac{\m}{4}}+|t|^{-\frac{1}{2}}\min(|x|^2,|x'|^2)\|\l^2b_{+,+}\|_{L^{\infty}_{\l}} \lesssim |t|^{-\frac{1}{2}},
\end{align*}
where we have used $\max(|x|^{-2},|x'|^{-2})=(\min(|x|^2,|x'|^2))^{-1}$ and $-t\l^2+S_{+,+}(\l,x,x')=\Phi(\l,x,x')$ in the notation in Theorem \ref{thm:WKBosc2}.

Since $S_{-,+}(\l,x,x')=S_{+,-}(\l,x',x)$ and $S_{-,-}(\l,x,x')=-S_{+,+}(\l,x,x')$, we obtain similar estimates for the $(-,+)$ and $(-,-)$ cases. This completes the proof.
\end{proof}

\section{Extension to potentials which may be positive on bounded intervals}\label{251271359}
In this section, we extend the dispersive estimates to those with slowly decaying attractive potentials which may be positive on a bounded interval. We omit the full details and provide only an outline of the proof.

\begin{assump}\label{251271401}
For some $\mu \in (0, 2)$ and any $\alpha \in \N_0$, a function $V \in C^{\infty} (\R; \R)$ satisfies $|\partial^{\alpha} _x V(x)| \lesssim \langle x \rangle^{-\mu-\alpha}$ uniformly in $x \in \R$.
Furthermore there exists $R_0 >0$ such that $-V(x) \gtrsim \langle x \rangle^{-\mu}$ uniformly in $|x| \ge \frac{R_0}{2}$.
\end{assump}

\begin{thm}\label{2512101131}
Under Assumption \ref{251271401}, the dispersive estimate \eqref{eq:dispersive} holds (for $d=1$).
\end{thm}

\subsection{Estimate in the high energy regime}

As is mentioned in Remark \ref{rem:spdensityWKBext}, the spectral density $\tilde{E}(\l,x,x')$ has a WKB expression of the form \eqref{2509041043} for $\l\geq C_0$ and $x,x'\in \R$, where we fix $C_0 \ge1$ such that $|V(x)| \le \frac{C_0}{2}$ for all $x \in \R$. Moreover, the arguments in Subsubsection \ref{262221527} and Subsection \ref{subsec:off-diagonal} still work. In this way, we can prove the dispersive estimate in this regime, that is,
\begin{align*}
\|e^{-itP}\g(P)\|_{L^1\to L^{\infty}}\lesssim |t|^{-\frac{1}{2}}\quad (\text{for $t\in \R\setminus \{0\}$}),
\end{align*}
where $\g\in C^{\infty}(\R;[0,1])$ satisfies $\g(\l)=1$ for $\l\ge 2C_0$ and $\g(\l)=0$ for $\l\leq C_0$.

\subsection{Estimate in the middle and low energy regime}

Below, we focus on the middle and low energy regime. 
Basically, we follow the argument in Section \ref{2508241544} to obtain the WKB expression of the spectral density although we have to be careful since the Liouville transform is not always well-defined for $|x|\leq R_0/2$ since $V$ is not assumed to be negative here.

\vspace{1mm}
\noindent\underline{\textbf{Liouville transform near the spatial infinity}}
Since $-V(x)\gtrsim \jap{x}^{-\mu}$ for $|x|\geq R_0/2$, we can still define the Liouville transform near the spatial infinity. For $x \in \R_{\ge R_0}$ and $x'\in \R_{\le -R_0}$, we define 
\begin{align*}
y_{+} (x, \lambda) = \int_{R_0}^x \sqrt{\lambda^2 -V(s)} ds,\quad y_{-} (x', \lambda) = \int_{-R_0}^{x'} \sqrt{\lambda^2 -V(s)} ds.
\end{align*}
Then $y_+ (\cdot, \lambda)$ has a smooth inverse $x_+ (\cdot, \lambda)$ defined on $[0, \infty)$, which also smoothly depends on $\lambda \in \R \setminus \{0\}$. Similarly we define $x_{-} (\cdot, \lambda)$ on $(-\infty, 0]$ as the inverse of $y_{-} (\cdot, \lambda)$. Its proof is identical to that of Lemma \ref{2509041100}. 
Moreover, by a similar calculation as in Lemma \ref{2509031823}, we have
\begin{align*}
|\partial^{\alpha} _{\lambda} x_{\pm}(y, \lambda)| \lesssim |\lambda|^{-\alpha} |x_{\pm} (y, \lambda)|
\end{align*}
for any $\alpha \in \N_0$, $\lambda \in \R \setminus \{0\}$ and $y \in [0, \infty)$ or $y \in (-\infty, 0]$ respectively.

\vspace{1mm}
\noindent\underline{\textbf{Construction of Jost functions}} 

\begin{proposition}\label{25129848}
For $0<|\lambda| \le 4C_0$, there exist smooth solutions $u_{\pm}$ of
\begin{align}
(-\partial^2 _x +V(x))u_{\pm} (x, \lambda) =\lambda^2 u_{\pm} (x, \lambda) \label{251271818}
\end{align}
such that the following holds:

We have
\begin{align}
u_{\pm} (x, \lambda) =
\left\{\begin{aligned} 
&\frac{1}{(\lambda^2 - V(x))^{\frac{1}{4}}} \left( \tilde{a}_{\pm, +}^+ (x, \lambda)e^{iy_+ (x, \lambda)} + \tilde{a}_{\pm, -}^+ (x, \lambda)e^{-iy_+ (x, \lambda)} \right)\quad (x\ge R_0)\\
&\frac{1}{(\lambda^2 - V(x))^{\frac{1}{4}}} \left( a_{\pm, +}^- (x, \lambda)e^{iy_- (x, \lambda)} + a_{\pm, -}^- (x, \lambda)e^{-iy_- (x, \lambda)} \right)\quad (x\leq -R_0)
\end{aligned}\right. ,\label{eq:upmexp2}
\end{align}
where $\tilde{a}_{\s_1, \s_2}^{\pm}$ satisfy \eqref{2509031822} uniformly in $\pm x \ge R_0$ and $0<|\lambda| \le 4C_0$ for $\alpha =0, 1, 2$. Moreover,

\vspace{1mm}
\noindent$(i)$ 
For sufficiently large $M>0$,
\begin{align*}
|\tilde{a}_{+,-}^+(x,\l)|\lesssim |\l|^{-2}|x|^{-2},\quad |\tilde{a}_{-, +}^- (x', \lambda)| \lesssim |\l|^{-2}|x'|^{-2}
\end{align*}
uniformly in $0<|\l|\leq 4C_0$, $x\geq M$ and $x'\leq -M$.

\vspace{1mm}
\noindent$(ii)$
For $\a=0,1,2$, we have
\begin{align}
|\pa_{\l}^{\a}u_{\pm} (x, \lambda)| \lesssim |\l|^{-\a}\label{eq:upmboundedsymb}
\end{align}
uniformly in $|x| \le 2R_0$ and $0<|\lambda| \le 4C_0$.

\vspace{1mm}
\noindent$(iii)$ Let $\mathrm{Wr}(\lambda) := u_{+} (x, \lambda) \partial_{x} u_{-} (x, \lambda) - \partial_{x} u_{+} (x, \lambda) u_{-} (x, \lambda)$ be the Wronskian. Then $\mathrm{Wr} (\lambda)$ satisfies
\begin{align*}
|\mathrm{Wr}(\lambda)| \sim 1, \quad \left| \left(\frac{d}{d\lambda} \right)^{\alpha} \mathrm{Wr}(\lambda) \right| \lesssim |\lambda|^{-\alpha} \quad \text{for any $\alpha=0,1,2$}
\end{align*}
uniformly in $0<|\l|\leq 4C_0$.

\end{proposition}

\begin{proof}
We only consider $u_+$. We define
\begin{align*}
&U(x,\lambda)=\frac{V''(x)}{4(\lambda^2-V(x))^2}-\frac{5(V'(x))^2}{16(\lambda^2-V(x))^{3}} ,\quad W_{\pm}(y,\lambda):=U(x_{\pm}(y,\l),\l),\\
&B_{\pm}(y, \lambda) = -W_{\pm}(y, \lambda)
\begin{pmatrix}
\sin y \cos y & \sin^{2} y \\
-\cos^2 y & -\sin y \cos y
\end{pmatrix}
\end{align*}
for $x\in \R_{\ge R_0}\cup \R_{\le -R_0}$, $\pm y\geq 0$ and $\l\in \R$ according to \eqref{eq:BdefW2} and \eqref{2510011536}. Similarly to Lemma \ref{lem:Wsymbol}, we see that $W_{\pm}$ satisfies \eqref{2509021619} and hence \eqref{eq:Bassump} holds for $B_{\pm}$.

By the proof of Theorems \ref{2508311359}, we can find a smooth solution $\tilde{\textbf{w}}_{+}$ of
\begin{align*}
\partial_{y} \tilde{\textbf{w}}_{+} (y, \lambda) = B_+(y, \lambda)\tilde{\textbf{w}}_{+} (y, \lambda)\quad (y\geq 0)
\end{align*}
such that $\sup\limits_{y \geq 0, |\lambda|\leq 4C_0 } \| \tilde{\textbf{w}}_{+} (y, \lambda)\|_{\C^2} < \infty$ and $\tilde{\textbf{w}}_{+} (\infty, \lambda):=\lim\limits_{y\to \infty} \tilde{\textbf{w}}_{+} (y, \lambda)={}^t(1,i)$.
Moreover, we have the uniform convergence
\begin{align*}
&\lim_{y\to \infty}\sup_{0<|\l|<4C_0} \|\tilde{\textbf{w}}_{+} (\infty, \lambda) - \tilde{\textbf{w}}_{+} (y, \lambda)\|_{\C^2} =0, \quad\lim_{y\to \infty}\sup_{\lambda \in \R} \| \partial_{y} \tilde{\textbf{w}}_{+} (y, \lambda)\|_{\C^2} = 0.
\end{align*}
Furthermore, for $\a,\b,\c\in\mathbb{N}_0$ satisfying $\b+\c\leq 2$,
\begin{align*}
\| \partial^{\alpha} _{\lambda} \tilde{\textbf{w}}_{\pm} (y, \lambda)\|_{\C^2} \lesssim | \lambda |^{- \alpha},\,\, \| \pa_y^{\b}\partial^{\c}_{\l} \tilde{\textbf{w}}_{\pm} (y, \lambda)\|_{\mathbb{C}^2} \lesssim | \lambda |^{-\c}\jap{y}^{-\b},\,\, |\tilde{w}_{1,+}(y,\l)+i\tilde{w}_{2,+}(y,\l)|\lesssim \jap{y}^{-2}
\end{align*}
uniformly in $y \ge 0$ and $0<|\l|\leq 4C_0$, where we write $\tilde{\textbf{w}}_{+}={}^t(\tilde{w}_{1,+},\tilde{w}_{2,+})$.
According to \eqref{2510161027} and \eqref{eq:vdefviau}, we define $\tilde{v}_+(y,\l):= \tilde{w}_{1, +} (y, \lambda) \cos y + \tilde{w}_{2, +} (y, \lambda) \sin y$ and $u_+(x):=(\lambda^2-V(x))^{-\frac{1}{4}}\tilde{v}_+(y_+(x,\l),\l)$, which is a solution of \eqref{251271818} for $x\geq R_0$.
Mimicking the proof of Theorems \ref{2509021616} and \ref{2509021637}, we can find $\tilde{a}_{+,\pm}^{+}$ satisfying \eqref{eq:upmexp2} for $x\geq R_0$ uniformly in $x \ge R_0$ and and $0<|\l|\leq 4C_0$. For sufficiently large $M>0$, we have
\begin{align*}
y_+ (x, \lambda) \ge \lambda |x-R_0| \gtrsim \lambda x\quad (x>M),
\end{align*}
which proves $(i)$.

We extend $u_+$ from $[R_0, \infty)$ to $\R$ as a solution to
\begin{align}
\partial_x
\begin{pmatrix}
u_+ (x, \lambda) \\
\partial_x u_+ (x, \lambda)
\end{pmatrix}
=
\begin{pmatrix}
0 & 1 \\
V(x) -\lambda^2 & 0
\end{pmatrix}
\begin{pmatrix}
u_+ (x, \lambda) \\
\partial_x u_+ (x, \lambda) \label{251281121}
\end{pmatrix}.
\end{align}
Next, we prove \eqref{eq:upmboundedsymb}. By (\ref{251281121}), we have
\begin{align*}
\left|\partial_x \|\alpha (x, \lambda)\|^2 _{\C^2}\right| \lesssim \|\alpha (x, \lambda)\|^2 _{\C^2}, \quad \alpha (x, \lambda) =
\begin{pmatrix}
u_+ (x, \lambda) \\
\partial_x u_+ (x, \lambda)
\end{pmatrix}
\end{align*}
uniformly in $|x| \le 2R_0$ and $0<|\lambda| \le 4C_0$. By the Gronwall inequality, we obtain $\|\alpha (x, \lambda)\|^2 _{\C^2} \lesssim \|\alpha (2R_0, \lambda)\|^2 _{\C^2}$ for $|x| \le 2R_0$. Since the right hand side is bounded in $0<|\lambda| \le 4C_0$, we have $\|\alpha (x, \lambda)\|^2 _{\C^2} \lesssim 1$, which implies $|u_{+} (x, \lambda)| \lesssim 1$. Next we use the relation
\begin{align*}
\partial_{\lambda} \partial_{x} \alpha (x, \lambda)=
\begin{pmatrix}
0 & 0 \\
-2\lambda & 0
\end{pmatrix}
\alpha (x, \lambda) + 
\begin{pmatrix}
0 & 1 \\
V(x) -\lambda^2 & 0
\end{pmatrix}
\partial_{\lambda} \alpha (x, \lambda).
\end{align*}
Then, for $|x| \le 2R_0$ and $0<|\lambda| \le 4C_0$, we have
\begin{align*}
\left| \partial_x \|\partial_{\lambda} \alpha (x, \lambda)\|^2 _{\C^2}\right| \lesssim \|\partial_{\lambda} \alpha (x, \lambda)\|_{\C^2} \|\partial_{\lambda} \partial_{x} \alpha (x, \lambda)\|_{\C^2} \lesssim \|\partial_{\lambda} \alpha (x, \lambda)\|^2 _{\C^2} +1,
\end{align*}
which leads to $\|\partial_{\lambda} \alpha (x, \lambda)\|^2 _{\C^2} \lesssim \|\partial_{\lambda} \alpha (2R_0, \lambda)\|^2 _{\C^2} +1$ by the Gronwall inequality. By differentiating the asymptotics in \eqref{eq:upmexp2}, we can check $\|\partial_{\lambda} \alpha (2R_0, \lambda)\|_{\C^2} \lesssim |\lambda|^{-1}$. Hence we obtain $\|\partial_{\lambda} \alpha (2R_0, \lambda)\|^2 _{\C^2} \lesssim |\lambda|^{-2}$, in particular, $|\partial_{\lambda} u_{+} (x, \lambda)| \lesssim |\lambda|^{-1}$. Similarly, we have $|\pa_{\l}^2u_+(x,\l)|\lesssim |\l|^{-2}$. Thus we have proved \eqref{eq:upmboundedsymb}.

Now we define $v_+ (y, \lambda) = (\l-V(x))^{\frac{1}{4}} u_+ (x, \lambda) |_{x=x_- (y, \lambda)}$ for $y  \le 0$. As in the proof of Theorem \ref{2509021637}, we find that $v_+$ satisfies
\begin{align*}
(-\partial^2 _{y} +W_- (y, \lambda) -1)v_+ (y, \lambda) =0\quad (y\leq 0).
\end{align*}
Setting
\begin{align*}
\textit{\textbf{w}}_+ (y, \lambda) = R(y)
\begin{pmatrix}
v_+ (y, \lambda) \\
\partial_y v_+ (y, \lambda)
\end{pmatrix}\,\, \text{for}\quad y\le 0
\end{align*}
we obtain $\partial_y \textit{\textbf{w}}_+ (y, \lambda) =B_-(y, \lambda) \textit{\textbf{w}}_+ (y, \lambda)$ on $(-\infty, 0]$. Since we have $|W_{-} (y, \lambda)| \lesssim \langle y \rangle^{-2}$,
\begin{align*}
\left|\partial_y \|\textit{\textbf{w}}_+ (y, \lambda)\|^2 _{\C^2} \right| \lesssim \langle y \rangle^{-2} \|\textit{\textbf{w}}_+ (y, \lambda)\|^2 _{\C^2},
\end{align*}
which leads to $\|\textit{\textbf{w}}_+ (y, \lambda)\|^2 _{\C^2} \lesssim \|\textit{\textbf{w}}_+ (0, \lambda)\|^2 _{\C^2}$ for $y \le 0$. Now we recall
\begin{align*}
\textit{\textbf{w}}_+ (0, \lambda) =
\begin{pmatrix}
(\lambda^2 -V(-R_0))^{\frac{1}{4}} u_{+} (-R_0, \lambda) \\
(\lambda^2 -V(-R_0))^{-\frac{1}{4}} \partial_x u_+ (-R_0, \lambda) -\frac{1}{4} \partial_x V(-R_0) (\lambda^2 -V(-R_0))^{-\frac{5}{4}} u_{+} (-R_0, \lambda)
\end{pmatrix}
\end{align*}
to conclude $\|\textit{\textbf{w}}_+ (y, \lambda)\|^2 _{\C^2} \lesssim \|\textit{\textbf{w}}_+ (0, \lambda)\|^2 _{\C^2} \lesssim 1$ for $y \le 0$ and\footnote{Here we use the boundedness of $u_{+} (-R_0, \lambda)$ and $\partial_x u_{+} (-R_0, \lambda)$ with respect to $|\lambda| \le 4C_0$, which follow from those of $u_{+} (2R_0, \lambda)$, $\partial_x u_{+} (2R_0, \lambda)$ (see the asymptotics (\ref{2509031759}) and (\ref{2509031801})) and the fact that $u_+$ is extended by (\ref{251281121}), whose coefficients are uniformly bounded in $|\lambda| \le 4C_0$ and $x \in \R$.} $|\lambda| \le 4C_0$. This and the bound $\|B_-(y,\l)\|_{\C^2\to \C^2}\lesssim \jap{y}^{-2}$ shows that the limit
\begin{align*}
\textit{\textbf{w}}_+(-\infty,\l):=\textit{\textbf{w}}_+(0,\l)+\lim_{y\to -\infty}\int_{y}^0B_-(y,\l)\textit{\textbf{w}}_+(s,\l)ds
\end{align*}
exists and the uniform convergence
\begin{align*}
\|\textit{\textbf{w}}_+ (y, \lambda) - \textit{\textbf{w}}_+ (-\infty, \lambda)\|_{\C^2} \to 0, \quad\sup_{0<|\lambda|\le 4C_0} \|\partial_y \textit{\textbf{w}}_+ (y, \lambda)\|_{\C^2} \to 0\quad \text{as $y\to -\infty$}.
\end{align*}
It turns out from \eqref{eq:upmboundedsymb} that $\| \partial^{\alpha} _{\lambda} \textit{\textbf{w}}_{+} (0, \lambda)\|_{\C^2}\lesssim |\l|^{-\a}$ for $\a=0,1,2$. By using this and the proof of Theorems \ref{2508311359}, if $\a,\b,\c\in\mathbb{N}_0$ satisfy $\a\leq 2$ and $\b+\c\leq 2$,
\begin{align*}
\| \partial^{\alpha} _{\lambda} \textit{\textbf{w}}_{+} (y, \lambda)\|_{\C^2} \lesssim | \lambda |^{- \alpha},\,\, 
\| \partial^{\alpha} _{\lambda} \textit{\textbf{w}}_{+} (-\infty, \lambda)\|_{\C^2} \lesssim | \lambda |^{- \alpha},\,\,
\| \pa_y^{\b}\partial^{\c}_{\l} \textit{\textbf{w}}_{+} (y, \lambda)\|_{\mathbb{C}^2} \lesssim | \lambda |^{-\c}\jap{y}^{-\b}
\end{align*}
for $y\leq 0$ and $0<|\l|\leq 4C_0$. Going back to $u_+$, we can find $\tilde{a}_{+,\pm}^{-}$ satisfying \eqref{2509031822} and \eqref{eq:upmexp2} as in the proof of Theorems \ref{2509021616} and \ref{2509021637}. The estimates for the Wronskian (the part $(iii)$) are proved similarly to Theorem \ref{2509021637}. The proof is completed.
\end{proof}

Now we give a proof of Theorem \ref{2512101131}. Firstly we consider the case where $V$ is even at spatial infinity. Then, for non-symmetric $V$, we give a supplementary argument.
\begin{proof}[Proof of Theorem \ref{2512101131}]
\underline{\textbf{The case $V(x)=V(-x)$ for $|x|\ge \frac{R_0}{2}$}} It suffices to consider
\begin{align*}
e^{-itP} \chi (P \le C^2 _0) := \int_{0}^{\infty} e^{-it\lambda^2} \tilde{E} (\lambda) \chi (\lambda \le C_0) d\lambda,
\end{align*}
where $ \chi (\lambda \le C_0) := 1-\chi (\lambda \ge C_0)$ and $\supp \chi (\cdot \le C_0) \subset (-\infty, 2C_0]$. We decompose the proof into four regions and reduce them to the arguments in Section \ref{2512101237}.

\noindent(i) If $x, x' \ge 2R_0$, we can use Proposition \ref{25129848} to obtain an analogous WKB expression to (\ref{2509041043}), where the phase functions are replaced by $\phi_{\sigma_1, \sigma_2} (\lambda, x, x') = \sigma_{1} \int_{R_0}^{x} \sqrt{\lambda^2 -V(s)} ds + \sigma_{2} \int_{R_0}^{x'} \sqrt{\lambda^2 -V(s)} ds$ but the amplitude satisfies the same bounds (\ref{2510201029}). We introduce $\tilde{V} \in C^{\infty} (\R; \R)$ such that $\tilde{V} (x) =V(x)$ if $|x| \ge \frac{3}{4}R_0$, $\tilde{V} (x) <0$ for all $x \in \R$ and $\tilde{V}$ is an even function. Then we rewrite
\begin{align*}
\phi_{\sigma_1, \sigma_2} (\lambda, x, x') &= \sigma_{1} \int_{0}^{x} \sqrt{\lambda^2 -\tilde{V}(s)} ds + \sigma_{2} \int_{0}^{x'} \sqrt{\lambda^2 -\tilde{V}(s)} ds \\
& \quad - \sigma_{1} \int_{0}^{R_0} \sqrt{\lambda^2 -\tilde{V}(s)} ds-\sigma_{2} \int_{0}^{R_0} \sqrt{\lambda^2 -\tilde{V}(s)} ds
\end{align*}
and include the extra factor $\exp \left(-i\left(\sigma_{1} \int_{0}^{R_0} \sqrt{\lambda^2 -\tilde{V}(s)} ds+\sigma_{2} \int_{0}^{R_0} \sqrt{\lambda^2 -\tilde{V}(s)} ds\right)\right)$ coming from the third and fourth terms into the amplitude. Under this procedure, the bounds (\ref{2510201029}) do not change and we can employ the same calculations in Section \ref{2512101237} for this regime. The same is also true for the regime $x, x' \le -2R_0$.

\noindent (ii) If $x \ge 2R_0$ and $x' \le -2R_0$, we proceed similarly to (i) but modify the phase function as
\begin{align*}
&\phi_{\sigma_1, \sigma_2} (\lambda, x, x'):= \sigma_{1} \int_{R_0}^{x} \sqrt{\lambda^2 -V(s)} ds + \sigma_{2} \int_{-R_0}^{x'} \sqrt{\lambda^2 -V(s)} ds \\
&\quad = \sigma_{1} \int_{0}^{x} \sqrt{\lambda^2 -\tilde{V}(s)} ds + \sigma_{2} \int_{0}^{x'} \sqrt{\lambda^2 -\tilde{V}(s)} ds- \sigma_{1} \int_{0}^{R_0} \sqrt{\lambda^2 -\tilde{V}(s)} ds-\sigma_{2} \int_{0}^{-R_0} \sqrt{\lambda^2 -\tilde{V}(s)} ds
\end{align*}
and include the extra factor coming from the third and fourth terms into the amplitude. The regime $x' \ge 2R_0$ and $x \le -2R_0$ can be handled similarly.

\noindent (iii) If $x \ge 2R_0$ and $|x'| \le 2R_0$, we rewrite $u_{\pm} (x', \lambda)$ as
\begin{align}
u_{\pm} (x', \lambda)= \frac{1}{(\lambda^2 - \tilde{V}(x'))^{\frac{1}{4}}} \underbrace{\left((\lambda^2 - \tilde{V}(x'))^{\frac{1}{4}} u_{\pm} (x', \lambda) e^{i\int_{0}^{x'} \sqrt{\lambda^2 -\tilde{V}(s)} ds}\right)}_{:= b_{\pm} (\lambda, x, x')}e^{-i\int_{0}^{x'} \sqrt{\lambda^2 -\tilde{V}(s)} ds} \label{2512101309}
\end{align}
and we include the amplitude $b_{\pm}$ into that of the spectral density. Next we use Proposition \ref{25129848} for $u_{\pm} (x, \lambda)$ to obtain analogous WKB expressions to (\ref{2509041043}). Then the amplitude of the spectral density satisfies (\ref{2510201029}) for $\alpha = 0, 1$ by \eqref{eq:upmboundedsymb}. Hence it suffices to argue as in Section \ref{2512101237}. The case $x \le -2R_0$ and $|x'| \le 2R_0$ is treated similarly. 

\noindent (iv) If $|x|, |x'| \le 2R_0$, we rewrite both $u_{\pm} (x, \lambda)$ and $u_{\pm} (x', \lambda)$ as in (\ref{2512101309}) and argue similarly.

\underline{\textbf{The case $V$ is not symmetric}} We take $\tilde{V} \in C^{\infty} (\R; \R)$ such that $\tilde{V} (x) =V(x)$ if $|x| \ge \frac{3}{4}R_0$ and $\tilde{V} (x) <0$ for all $x \in \R$. We may assume $|x|, |x'| >M>2R_0$ for sufficiently large $M >1$ since otherwise the proof is ruduced to that for symmetric potentials. We have only to consider the oscillatory integral for which the phase function comes from the sign $(\sigma_1, \sigma_2) = (+, +)$. If $x, x' >M$ or $x, x' <-M$, we use the phase as in (i) above and the proof is identical to that in Subsubsection \ref{262221507}. For the case $x, -x'>M$, we only need to use the off-diagonal decay of the Jost functions: Proposition \eqref{25129848} $(i)$.  Therefore, with the phase as in (ii) above, we can follow the argument in the proof of Theorem \ref{thm:WKBosc2}.
\end{proof}

\end{document}